\documentclass{amsart}
\usepackage[left=1in,right=1in,bottom=1in,top=1in]{geometry}
\usepackage{graphicx} 
\usepackage{amsaddr}
\usepackage[misc]{ifsym}
\usepackage{amsmath,amsfonts,amsthm,amssymb,tikz-cd,hyperref, mathrsfs}
\usepackage{tikz}
\usepackage{graphicx,fontawesome}
\usetikzlibrary{arrows}
\usepackage{mathtools}
\usepackage{caption}
\usepackage{stmaryrd}
\usepackage{comment}
\usepackage{xcolor}
\usepackage{extarrows}
\hypersetup{
  colorlinks = true,
  linkcolor  = blue,
  urlcolor = blue,
  citecolor = blue,
  linkbordercolor = {white},
  urlbordercolor = {white},
  citebordercolor = {white}
}
\usepackage[backend=bibtex,style=numeric,sorting=nyt]{biblatex}
\usepackage[normalem]{ulem}
\usepackage[depth=3]{bookmark}
\usepackage[nosort]{cleveref}
\usepackage[shortlabels,inline]{enumitem}
\allowdisplaybreaks

\newtheorem{theorem}{Theorem}[section]
\newtheorem{proposition}[theorem]{Proposition}
\newtheorem{lemma}[theorem]{Lemma}
\newtheorem{corollary}[theorem]{Corollary}

\theoremstyle{definition}
\newtheorem*{theorem*}{Theorem}
\newtheorem{example}[theorem]{Example}
\newtheorem{definition}[theorem]{Definition}

\newtheorem{remark}[theorem]{Remark}

\numberwithin{equation}{section}

\DeclareMathOperator{\Lie}{Lie}
\DeclareMathOperator{\Pic}{Pic}
\DeclareMathOperator{\Div}{Div}
\DeclareMathOperator{\ddiv}{div}
\DeclareMathOperator{\GL}{GL}

\DeclareMathOperator{\SL}{SL}

\DeclareMathOperator{\tr}{tr}
\DeclareMathOperator{\id}{id}

\DeclareMathOperator{\gr}{gr}
\DeclareMathOperator{\Supp}{Supp}

\DeclareMathOperator{\Inv}{Inv}
\DeclareMathOperator{\API}{API}

\DeclareMathOperator{\AV}{AV}
\DeclareMathOperator{\red}{red}
\DeclareMathOperator{\ad}{ad}
\DeclareMathOperator{\Ad}{Ad}

\DeclareMathOperator{\diag}{diag}
\DeclareMathOperator{\reg}{reg}
\DeclareMathOperator{\sreg}{sreg}
\DeclareMathOperator{\even}{even}
\DeclareMathOperator{\odd}{odd}

\DeclareMathOperator{\codim}{codim}
\DeclareMathOperator{\Spec}{Spec}
\DeclareMathOperator{\Proj}{Proj}
\DeclareMathOperator{\im}{im}
\DeclareMathOperator{\Hom}{Hom}

\DeclareMathOperator{\rank}{rank}
\DeclareMathOperator{\End}{End}

\DeclareMathOperator{\Aut}{Aut}
\DeclareMathOperator{\Irr}{Irr}

\DeclareMathAlphabet{\mb}{U}{bbold}{m}{n}

\DeclareCaptionType{mtable}[Picture][Pictures]
\crefname{mtable}{picture}{pictures}
\Crefname{mtable}{Picture}{Pictures}

\newcommand{\la}{\leftarrow}
\newcommand{\gitquo}{/\!/}

\newcommand{\hO}{\mathcal{O}}
\newcommand{\pO}{\partial\mathbb{O}}
\newcommand{\bO}{\overline{\mathbb{O}}}
\newcommand{\hA}{\mathcal{A}}
\newcommand{\hI}{\mathcal{I}}

\newcommand{\M}{\mathcal{M}}
\newcommand{\N}{\mathcal{N}}
\newcommand{\F}{\mathcal{F}}
\newcommand{\hL}{\mathcal{L}}
\newcommand{\hb}{\mathfrak{b}}
\newcommand{\hg}{\mathfrak{g}}
\newcommand{\hk}{\mathfrak{k}}
\newcommand{\hp}{\mathfrak{p}}

\newcommand{\hn}{\mathfrak{n}}
\newcommand{\A}{\mathbb{A}}
\newcommand{\C}{\mathbb{C}}
\newcommand{\R}{\mathbb{R}}
\newcommand{\Z}{\mathbb{Z}}
\newcommand{\pr}{\mathbb{P}}
\newcommand{\Q}{\mathbb{Q}}
\newcommand{\OO}{\mathbb{O}}

\newcommand{\U}{\mathcal{U}}

\newcommand{\ttheta}{\tilde{\theta}}
\newcommand{\wt}{\widetilde}
\newcommand{\gl}{\mathfrak{gl}}
\newcommand{\so}{\mathfrak{so}}
\newcommand{\su}{\mathfrak{su}}

\renewcommand{\sp}{\mathfrak{sp}}
\renewcommand{\sl}{\mathfrak{sl}}

\newcommand{\minus}{\scalebox{0.75}[1.0]{$-$}}

\DeclareMathSymbol{\sminus}{\mathbin}{AMSa}{"39}

\newcommand{\RMN}[1]{%
\textup{\uppercase\expandafter{\romannumeral#1}}%
}
\newcommand{\rmn}[1]{%
\textup{\lowercase\expandafter{\romannumeral#1}}%
}

\makeatletter
\newcommand*{\RNum}{} 
\let\RNum\@Roman      
\makeatother

\begin{document}
\title{On certain Lagrangian subvarieties in minimal resolutions of Kleinian singularities}
\author{\vspace{-2em}Mengwei Hu}

\begin{abstract}
Kleinian singularities are quotients of $\mathbb{C}^2$ by finite subgroups of $\mathrm{SL}_2(\mathbb{C})$. They are in bijection with the simply-laced Dynkin diagrams via the McKay correspondence. Anti-Poisson involutions and their fixed point loci appear naturally in the classification of irreducible Harish-Chandra modules over Kleinian singularities. There are three goals of this paper. The first is to classify anti-Poisson involutions of Kleinian singularities up to conjugation by graded Poisson automorphisms. The second is to describe the scheme-theoretic fixed point loci of Kleinian singularities under anti-Poisson involutions. The third and main goal is to describe the scheme-theoretic preimages of the fixed point loci under minimal resolutions of Kleinian singularities. These preimages are singular Lagrangian subvarieties in the minimal resolutions whose irreducible components are $\mathbb{P}^1$'s and $\mathbb{A}^1$'s. Additionally, we discuss the connections between anti-Poisson involutions and Lie algebra involutions.
\end{abstract}
\maketitle
\tableofcontents
\setcounter{section}{0}
\section{Introduction}
\subsection{Kleinian singularities}\label{Kleinianintro} Kleinian singularities are quotients of $\C^2$ by finite subgroups $\Gamma\subset\SL_2(\C)$. They can be described as surfaces in $\C^3$ with isolated singularities at $0$. The McKay correspondence states that the finite subgroups of $\SL_2(\C)$ are in bijection with the simply-laced Dynkin diagrams. Let $\pi\colon\wt{X}\to X:=\C^2/\Gamma$ denote the minimal resolution of a Kleinian singularity. The reduced exceptional fiber is of the form
\[
\pi^{-1}(0)_{\red}=C_1\cup\cdots\cup C_n,\ C_i\simeq \pr^1.
\]
Each $C_i$ has self-intersection $C_i\cdot C_i=-2$, and the curves intersect pairwise transversely with dual graph a Dynkin diagram of types $ADE$.

The algebra $\C[u,v]$ is a graded Poisson algebra with respect to the usual bracket and grading. The algebra of regular functions on a Kleinian singularity $\C[X]=\C[u,v]^\Gamma$ is a graded Poisson subalgebra of $\C[u,v]$. Let $\theta$ be an \emph{anti-Poisson involution} on $X$, which means that $\theta$ is a graded algebra involution on $\C[X]$ such that $\theta(\{\cdot,\cdot\})=-\{\theta(\cdot),\theta(\cdot)\}$. We would like to study the \emph{scheme-theoretic fixed point locus} $X^{\theta}:=\Spec\C[X]/(\theta(f)-f|f\in\C[X])$ and its scheme-theoretic preimage $\pi^{-1}(X^{\theta})$ under the minimal resolution. There are three goals of this paper.
\begin{enumerate}[label=(\alph*)]
    \item\label{main1} Classify anti-Poisson involutions on $X$ up to conjugation by graded Poisson automorphisms.
    \item\label{main2} Describe $X^{\theta}$ as a subscheme, i.e., determine its irreducible components, how they intersect, and whether the scheme is reduced.
    \item\label{main3} Describe the preimage $\pi^{-1}(X^{\theta})$ under the minimal resolution as a subscheme.
\end{enumerate}

\subsection{Motivations}\label{motivation}
We discuss the motivation to study anti-Poisson involutions of Kleinian singularities and their fixed point loci. They are related to the classification of irreducible Harish-Chandra $(\hg,K)$-modules annihilated by certain unipotent ideals. We review some results on nilpotent orbits in \Cref{nilp}, which will be needed for the discussion on Harish-Chandra modules in \Cref{HCmod}.

\subsubsection{Nilpotent orbits and symmetric pairs}\label{nilp} 
Let $\hg$ be a semisimple Lie algebra over $\C$ and $\OO\subset\hg$ be a nilpotent orbit. Then $\OO$ can be identified with a coadjoint orbit in $\hg^*$ using the Killing form, and therefore carries the Kostant-Kirillov symplectic $2$-form. Let $\bO$ be the closure of $\OO$ in $\hg$ and $\OO'$ be an orbit of codimension $2$ in $\bO$. Take $e\in\OO'$ and let $\{e,f,h\}$ be an $\sl_2$-triple containing $e$. Consider the \emph{Slodowy slice} 
\begin{equation}\label{Slodowysliceintro}
    S:=e+\ker(\ad f),
\end{equation}
which is a transversal slice to the orbit through $e$ \cite[Section 7.4]{slodowy1980simple}. Assume that $e$ is a normal point in $\bO$. Then $S\cap\bO$ is a rational Gorenstein (due to Panyushev \cite[Theorem 1]{Panyushev}) singularity of dimension $2$, hence isomorphic to a Kleinian singularity $X$. 
Note that $\bO$ is always normal in type $A$ \cite[Theorem]{KPconj}, and we can pass to its normalization in the general situation.

Anti-Poisson involutions arise naturally from involutions of Lie algebras. Let $\sigma\colon\hg\to\hg$ be a Lie algebra involution. We can decompose $\hg=\hk\oplus\hp$, where $\mathfrak{k}$ is the $1$-eigenspace and $\mathfrak{p}$ is the $(-1)$-eigenspace of $\sigma$. The pair $(\hg,\hk)$ is called a \textit{symmetric pair}. Set $\theta:=-\sigma$; then $\theta$ is an anti-Poisson involution of $\hg$. 

If we further assume that the $\sl_2$-triple $\{e,f,h\}$ is \emph{normal}, i.e., $\sigma(e)=-e,\ \sigma(f)=-f,\ \sigma(h)=h$, then $S$ is $\theta$-stable. As long as $\OO'\cap\hp\neq\emptyset$, such a normal triple exists \cite[Theorems 9.4.1, 9.5.1]{CMnilp}. Assume further that $\sigma$ preserves $\OO$. This condition is satisfied for all regular and subregular nilpotent orbits in any $\mathfrak{g}$, as well as for all nilpotent orbits in type $A$. Then $\theta$ restricts to an anti-Poisson involution on $S\cap\bO\simeq X$. We have $S\cap\bO\cap\hp=(S\cap\bO)^{\theta}\simeq X^{\theta}$, where the fixed point locus shows up. 


Fixed point loci of this form for symmetric pairs in the classical types were analyzed by Sekiguchi \cite[Theorem 7, Tables IV--VI]{Sekiguchi} and Ohta \cite[Table VIII]{Ohta} (cf. \Cref{typeAalt} for type $A$ and \Cref{SekiOhtacompare} for type $D$). Our method also treats type $E$. For a detailed discussion of the connection between Lie algebra involutions and anti-Poisson involutions, see \Cref{liealgsection}.



\subsubsection{Harish-Chandra modules}\label{HCmod} 
Let $\hg,\hk,\sigma,\theta$ be as in \Cref{nilp}. Let $G$ be the simply connected algebraic group with Lie algebra $\hg$ and $K\subset G$ the connected algebraic subgroup with Lie subalgebra $\hk$. Let $\U(\hg)$ be the universal enveloping algebra of $\hg$, equipped with the PBW filtration. By a \emph{Harish-Chandra} (shortly HC) $(\hg,K)$-module, one means a finitely generated $\U(\hg)$-module $M$ such that $\hk$ acts locally finitely and the action of $\hk$ integrates to that of $K$. A \emph{good} filtration on an HC $(\hg,K)$-module $M$ is a $K$-stable filtration that is compatible with the filtration on $\U(\hg)$ and such that $\gr M$ is finitely generated over $\gr\U(\hg)$. Pick a good filtration on $M$. We can define the \emph{associated variety} $\AV(M):=\Supp(\gr M)$ as a set. Note that $\AV(M)$ is independent of the choice of the good filtration (cf. \cite[Section 2]{Vogan}). When $M$ is irreducible, we have $\AV(M)\subset\N\cap\hp$, where $\N\subset\hg$ is the nilpotent cone.

We focus on a special class of irreducible HC $(\hg,K)$-modules, namely those annihilated by unipotent ideals. 
Let $\OO\subset\hg$ be a nilpotent orbit. The algebra of regular functions $\C[\OO]$ is a graded Poisson algebra. It makes
sense to speak about filtered quantizations of $\C[\OO]$, and they are parameterized by $H^2(Y^{\reg},\C)$, where $Y$ is a $\Q$-factorial terminalization of $\Spec(\C[\OO])$ \cite[Theorem 3.4]{Losevdeformation}. Consider the \emph{canonical} (in the terminology of \cite[Section 1.3]{LLDunipotent}) quantization $\hA_0$, which corresponds to the parameter $0\in H^2(Y^{\reg},\C)$. The kernel of the quantum comoment map $\hI_0(\OO):=\ker(\U(\hg)\to\hA_0)$ is the \emph{unipotent ideal associated with $\OO$} in the terminology of \cite[Section 1.3]{LLDunipotent}. All such ideals are maximal \cite[Theorem 5.0.1]{MBMalgebra}. Let $M$ be an irreducible HC $(\hg,K)$-module annihilated by $\hI_0(\OO)$. Then we have $\AV(M)\subset\bO\cap\hp$. Let us further assume that 
\begin{equation}
    \codim_{\bO}\pO\geqslant 4.
\end{equation}
A result by Vogan states that $\AV(M)$ is irreducible and hence is the closure of a single $K$-orbit in $\OO\cap\hp$ \cite[Theorem 4.6]{Vogan}. Let $\OO_K$ denote the open $K$-orbit in $\AV(M)$. Losev and Yu prove that there is a bijection between the set of irreducible HC $(\hg,K)$-modules annihilated by $\hI_0(\OO)$ with associated variety $\bO_K$ and the set of irreducible $K$-equivariant twisted (by half-canonical twist) local systems on $\OO_K$ \cite[Theorem 1.3.1]{LosevYu}. The bijection sends an HC module $M$ to the twisted local system $(\gr M)|_{\OO_K}$.

However, when $\codim_{\bO}\pO\geqslant 2$, the classification of irreducible HC $(\hg,K)$-modules annihilated by $\hI_0(\OO)$ is unknown. One difficulty is that $\AV(M)$ may not be irreducible. For example, consider the Lie algebra $\sl_2(\C)=\C\langle e,h,f\rangle$ with the inner involution $\sigma:=\Ad(\diag\{1,-1\})$. Then $\hk=\hg^{\sigma}=\C h$ is a Cartan subalgebra. Let $\OO$ be the regular nilpotent orbit in $\sl_2(\C)$. The unipotent ideal $\hI_0(\OO)$ is the two-sided ideal generated by the Casimir element $C=(h+1)^2+4fe$. Let $M$ be an irreducible $\sl_2(\C)$-module such that $h$ acts semisimply with even integral eigenvalues and the Casimir element $C$ acts by $0$. Then $M$ has neither a highest weight nor a lowest weight vector. Equip $M$ with a good filtration. Then $\AV(M)$ is a union of two lines. Another difficulty is that the condition $\codim_{\bO}(\pO)\geqslant 4$ is essential to show that non-isomorphic HC modules correspond to non-isomorphic twisted local systems in the proof of \cite[Theorem 1.3.1]{LosevYu}. 

\subsubsection{Springer resolutions and Beilinson-Bernstein localization} The geometric data needed in the unsolved codimension $2$ case are therefore much more complicated than a twisted local system on an open $K$-orbit in $\bO\cap\hp$. Assume that $\bO$ admits a symplectic resolution given by the Springer resolution $\pi\colon T^*(G/P)\to\bO$, where $P\subset G$ is a parabolic subgroup. This holds for the regular nilpotent orbit $\OO_{\reg}$ in every type, and the Springer resolution is given by $T^*(G/B)\to\N=\bO_{\reg}$, where $B\subset G$ is a Borel subgroup and $\N$ denotes the nilpotent cone. It also holds for every nilpotent orbit in type $A$. For other orbits in other types, further analysis is required (cf. \cite[Sections 3.3--3.4]{Funilp}). We instead consider $K$-equivariant $\mathcal{D}$-modules on the flag variety $G/P$, whose global sections are HC $(\hg,K)$-modules by the Beilinson-Bernstein localization theorem \cite{BBpaper}. It is expected that the $\mathcal{D}$-modules are more manageable and exhibit refined geometric properties since they live on the smooth variety $G/P$.
Recall that the associated varieties of HC $(\hg,K)$-modules annihilated by $\hI_0(\OO)$ are contained in $\bO\cap\hp$. Thus, their Beilinson-Bernstein localizations have singular supports contained in $\pi^{-1}(\bO\cap\hp)$.


To proceed further with the classification in the codimension $2$ case, it is necessary to understand the structure of the variety $S\cap\bO\cap\hp$ as well as its preimage $\pi^{-1}(S\cap\bO\cap\hp)$ under the Springer resolution, where $S$ is the Slodowy slice to a point in the codimension $2$ orbit $\OO'$, as defined in \cref{Slodowysliceintro}. As we explained at the end of \Cref{nilp}, $S\cap\bO$ is isomorphic to a Kleinian singularity $X$, and $S\cap\bO\cap\hp\simeq X^{\theta}$ is the fixed point locus. The restriction of the Springer resolution $\pi\colon T^*(G/P)\to\bO$ to $S\cap\bO\simeq X$ gives rise to the minimal resolution of the Kleinian singularity $X$. We have $\pi^{-1}(S\cap\bO\cap\hp)\simeq\pi^{-1}(X^{\theta})$, where the preimage of the fixed point locus shows up.

\subsection{Main results}
This paper provides complete answers to the three questions \ref{main1}-\ref{main3} posed in \Cref{Kleinianintro}. We summarize our main results below. Recall that $X$ denotes a Kleinian singularity. The following proposition answers the question \ref{main1}.
\begin{proposition}\label[proposition]{APIclassification}
    There are finitely many anti-Poisson involutions on $X:=\C^2/\Gamma$ up to conjugation by graded Poisson automorphisms. They can all be written out explicitly in terms of generators of $\C[X]$.
\end{proposition}
The finiteness of the conjugacy classes of anti-Poisson involutions is established in \Cref{apiconjfinite}.
Explicit representatives of the anti-Poisson involutions are given in \Cref{AnAPIodd,AnAPIeven,DnevenAPI,DnoddAPI,E6API,E7API,E8API}. The next proposition answers the question \ref{main2}.
\begin{proposition}\label[proposition]{descriptionfixedloci}
    Let $\theta$ be an anti-Poisson involution on $X$. The fixed point locus $X^{\theta}$ is reduced. If $X^{\theta}$ is not a single point, each irreducible component of $X^{\theta}$ is either $\A^1$ or a cusp.
\end{proposition}
This follows from the case-by-case calculations in \Cref{Anfixedloci,Dnevenfixedloci,Dnoddfixedloci,E6fixedloci,E7fixedlocus,E8fixedlocus}. Recall that $\pi\colon\wt{X}\to X$ denotes the minimal resolution. \Cref{liftexists} is crucial to the description of $\pi^{-1}(X^{\theta})$ as a subvariety, which constitutes the major part of the answer to the question \ref{main3}.
\begin{theorem}\label{liftexists}
   There exists a unique anti-symplectic involution $\ttheta:\wt{X}\to\wt{X}$ such that $\pi\circ\ttheta=\theta\circ\pi$. It can be constructed via involutions of quiver varieties. We call $\ttheta$ a \emph{lift} of $\theta$.
\end{theorem}

Note that the lift $\ttheta$ always exists due to the definition and uniqueness of the minimal resolution $\pi\colon\wt{X}\to X$. However, to analyze the action of $\ttheta$ on $\wt{X}$, especially the action on the exceptional fiber $\pi^{-1}(0)_{\red}=\bigcup_{1\leqslant i\leqslant n}C_i$ with $C_i\simeq\pr^1$, we need an explicit construction of the lift $\ttheta$. The proof of \Cref{liftexists} is based on a case-by-case study in \Cref{Anliftx-y,Anliftz,Anevenliftx-y,Anliftxz,Dnevenliftz,Dnevenliftyy',Dnoddlifty,Dnoddliftzz',E6liftx,E6liftz,E7lift,E8lift}. The key ingredient is to realize both $\wt{X}$ and $X$ as Nakajima quiver varieties explicitly. The technical details in types $D_n$ and $E_6$ are treated in \Cref{Dngeneratorprop,E6generatorprop}. The proofs in types $E_7$ and $E_8$ can be simplified with the help of \Cref{halfdeglemma}. With \Cref{liftexists}, we can describe $\pi^{-1}(X^{\theta})_{\red}$ following the recipe given in \Cref{liftsectionq2}. 


Denote the one-dimensional irreducible components of $X^{\theta}$ by $L_1,\cdots,L_m$. Set 
\[
\tilde{L}_j:=\overline{\pi^{-1}(L_j\setminus\{0\})},\ 1\leqslant j\leqslant m.
\]
Then we have $\tilde{L}_j\simeq\A^1$ (cf. \Cref{LA1}). Define $b_i:=\#\{1\leqslant j\leqslant m|\tilde{L}_j\cap C_i\neq\emptyset\}$. Recall that $\mathcal{C}$ denotes the Cartan matrix of the Dynkin diagram labeling $\Gamma$. \Cref{multintro} allows us to determine the multiplicities of each component in $\pi^{-1}(X^{\theta})$ as a divisor, which provides the remaining answer to the question \ref{main3}.
\begin{proposition}[\ref{multofcomponent}]\label[proposition]{multintro}
       If $X^{\theta}\subset X$ is a principal divisor, then $\pi^{-1}(X^{\theta})=\sum_{j=1}^m\tilde{L}_j+\sum_{i=1}^n a_iC_i$ with
   \begin{equation*}
       (a_1,\cdots,a_n)^t=\mathcal{C}^{-1}(b_1,\cdots,b_n)^t.
   \end{equation*}
\end{proposition}
We comment that $X^{\theta}$ is a principal divisor in $X$ except in two cases in type $A$, for which we use alternative methods to determine the multiplicities (cf. \Cref{Anoddcase3sub,Anevencase2sub}). The descriptions of the scheme-theoretic preimages $\pi^{-1}(X^{\theta})$ are given in \Cref{Anoddcase1preimage,Anoddcase2preimage,Anevencase1preimage,Anevencase2preimage,Dnevencase1preimage,Dnevencase2preimage,Dnoddcase1preimage,Dnoddcase2preimage,E6case1preimage,E6case2preimage,E7preimage,E8preimage}.

We would like to point out that a recent parallel and independent work by Bhaduri et al. \cite{McKayGLn} studied the set-theoretic preimage $\pi^{-1}(X^{\theta})_{\red}$ for some of the anti-Poisson involutions (half of the cases in types $D_n$ and $E_6$, and all the cases in types $E_7$ and $E_8$) using a different approach. Their motivation is to generalize the McKay correspondence to some complex reflection groups in $\GL_2(\C)$.

Finally, we discuss the connections between anti-Poisson involutions and Lie algebra involutions. In \Cref{nilp}, we mentioned that a Lie algebra involution $\sigma\colon\hg\to\hg$ induces an anti-Poisson involution $\theta=-\sigma$ on the Kleinian singularity $S\cap\N\simeq X$. This raises the natural question of which anti-Poisson involutions arise in this manner. The following proposition provides a complete characterization.

\begin{proposition}[\ref{Liealgbij}]
    All anti-Poisson involutions of a Kleinian singularity $X$ arise from Lie algebra involutions. Moreover, they are in bijection with quasi-split real forms of $\hg$, except in type $A_n$ with $n$ odd.
\end{proposition}

\subsection{Structure of the paper}
We recall some basic facts about Kleinian singularities in \Cref{Kleinianreview} and construct the Poisson structures on them in \Cref{pbsubsection}. In \Cref{APIsection}, we classify the anti-Poisson involutions on $X$ up to conjugation by graded Poisson automorphisms and compute their fixed point loci $X^{\theta}$. In \Cref{repofquiver,quiverapp}, we review the construction of Nakajima quiver varieties and apply it to Kleinian singularities. In \Cref{liftsection}, we discuss the procedures to determine the scheme-theoretic preimage $\pi^{-1}(X^{\theta})$. In \Cref{Anquiver,Dnquiver,E6quiver,E7quiver,E8quiver}, we realize Kleinian singularities as Nakajima quiver varieties explicitly and study their preimages $\pi^{-1}(X^{\theta})$, following the recipe given in \Cref{liftsection}. In \Cref{liealgsection}, we discuss the connections between anti-Poisson involutions and Lie algebra involutions.

\section{Preliminaries}
\subsection{Reminder on Kleinian singularities}\label{Kleinianreview}
We recall some basic facts about Kleinian singularities, following \cite[\S1,2]{SlodowyCassens} and \cite[\S 5]{SlodowyPlato}. Let $\Gamma\subset\SL_2(\C)$ denote a finite subgroup. Up to conjugacy, there are five classes of such groups. We list the conjugacy classes below, where $\epsilon_n$ denotes a primitive $n$-th root of unity for $n=5,8,10$.
 
($A_n$) The cyclic group of order $n+1$, i.e., $\big\{\begin{pmatrix} \epsilon & 0 \\ 0 & \epsilon^{-1} \end{pmatrix}|\epsilon^{n+1}=1,\ n\geqslant 1\big\}$.

($D_n$) The binary dihedral group of order $4(n-2)$, i.e., $\big\{\begin{pmatrix} \epsilon & 0 \\ 0 & \epsilon^{-1} \end{pmatrix},\ \begin{pmatrix} 0 & \epsilon \\ -\epsilon^{-1} & 0 \end{pmatrix}|\epsilon^{2(n-2)}=1,\ n\geqslant 4\big\}$.

($E_6$) The binary tetrahedral group of order $24$. 

$\ \ \ \ \ \ $ It is generated by $\begin{pmatrix} i & 0 \\ 0 & -i \end{pmatrix},\ \begin{pmatrix} 0 & i \\ i & 0 \end{pmatrix},\ \dfrac{1}{1-i}\begin{pmatrix} 1 & i \\ 1 & -i \end{pmatrix}$ 

($E_7$) The binary octahedral group of order $48$.

$\ \ \ \ \ \ $ It is generated by $\begin{pmatrix} \epsilon_8 & 0 \\ 0 & \epsilon_8^{-1} \end{pmatrix},\ \begin{pmatrix} 0 & i \\ i & 0 \end{pmatrix},\ \dfrac{1}{1-i}\begin{pmatrix} 1 & i \\ 1 & -i \end{pmatrix}$,

($E_8$) The binary icosahedral group of order $120$.

$\ \ \ \ \ \ $ It is generated by $\begin{pmatrix} \epsilon_{10} & 0 \\ 0 & \epsilon_{10}^{-1} \end{pmatrix},\ \begin{pmatrix} 0 & i \\ i & 0 \end{pmatrix},\ \dfrac{1}{\sqrt{5}}\begin{pmatrix} \epsilon_5-\epsilon_5^4 & \epsilon_5^2-\epsilon_5^3 \\ \epsilon_5^2-\epsilon_5^3 & -\epsilon_5+\epsilon_5^4. \end{pmatrix}$.

The \emph{Kleinian singularity} attached to a finite subgroup $\Gamma$ is the quotient singularity 
\[
X:=\C^2/\Gamma=\Spec\C[u,v]^{\Gamma}.
\]
Klein showed that $X$ can be viewed as a hypersurface in $\C^3$
\[
  X\simeq\{(x,y,z)\in\C^3|F(x,y,z)=0\},
\]
with an isolated singularity at $0$ (\cite{Klein}). Here $F$ denotes the relation among the three fundamental invariants $x,y,z$ in $\C[u,v]^{\Gamma}$. We list the most commonly used forms of these relations below.

($A_n$) $x_1y_1+z_1^{n+1}=0$ with $n\geqslant 1$.

($D_n$) $x_1^{n-1}+x_1y_1^2+z_1^2=0$ with $n\geqslant 4$.

($E_6$) $x_1^4+y_1^3+z_1^2=0$.

($E_7$) $x_1^3y_1+y_1^3+z_1^2=0$.

($E_8$) $x_1^5+y_1^3+z_1^2=0$.

We use the subscripts in $x_1,y_1,z_1$ because $x,y,z$ are reserved for different choices of generators later in the paper (cf. \Cref{ADExyz}). One can observe that the above list of finite subgroups of $\SL_2(\C)$ is in bijection with Dynkin diagrams of types $ADE$. There are two ways to attach a simply-laced Dynkin diagram to a Kleinian singularity $X$. McKay observed a representation-theoretic way, called the McKay correspondence, in 1979 \cite{McKay79}. Let $\Irr(\Gamma):=\{V_0,\cdots, V_n\}$ denote the set of irreducible representations of $\Gamma$ and assume that $V_0$ is the trivial representation. Set
\[
a_{ij}:=\dim\Hom_{\Gamma}(V_i\otimes\C^2,V_j),
\]
where $\C^2$ is the tautological representation of $\SL_2(\C)$. Then we have
\begin{equation}\label{McKay2}
\sum_{j=0}^na_{ij}\dim V_{j}=\dim (V_i\otimes \C^2)=2\dim V_i.
\end{equation}
Moreover, it is easy to check that $a_{ij}=a_{ji}$ because $\C^2\simeq(\C^2)^*$ as $\Gamma$-representations. Consider a graph with vertices $0,\cdots,n$ and the number of unoriented edges between $i$ and $j$ equal to $a_{ij}$. This is called the \emph{McKay graph} of $\Gamma$, which is exactly the extended Dynkin diagram in the corresponding type with the trivial representation $V_0$ corresponding to the extended vertex $0$. 

The other identification is algebro-geometric. Recall that by a resolution of $X$, one means a smooth variety $\wt{X}$ equipped with a projective birational morphism to $X$. Since we are in dimension $2$, there is a \emph{minimal} resolution $\pi\colon\wt{X}\to X$; the minimality condition means that any other resolution factors through $\wt{X}$. The minimal resolution $\pi\colon\wt{X}\to X$ was essentially studied by Du Val \cite{DuVal}. For a Kleinian singularity $X:=\C^2/\Gamma$, the reduced fiber over the isolated singularity $0$ under the minimal resolution $\pi\colon\wt{X}\to X$ is a connected union of projective lines 
\[
\pi^{-1}(0)_{\red}=C_1\cup\cdots\cup C_n,\ C_i\simeq\pr^1,
\]
with pairwise transversal intersection whose dual graph is the Dynkin diagram labeling $\Gamma$. More precisely, if $C_i\cap C_j\neq\emptyset$, then it is a single point, and each $C_i$ has self-intersection number $\minus2$. Let $\mathcal{C}$ denote the Cartan matrix of the corresponding type $ADE$. Then the intersection matrix $(C_i\cdot C_j)$ equals $\minus\mathcal{C}$. The intersection pairing will be defined rigorously in \Cref{pairingsection}. 

\begin{remark}\label[remark]{exceptionalnonreduced}
    The scheme-theoretic exceptional fiber $\pi^{-1}(0)$ is not reduced except in type $A$. Let $\hg$ be a simple Lie algebra of types $ADE$, and $\{\alpha_1,\cdots,\alpha_n\}$ be the set of simple roots. There is a unique maximal root $\alpha=\sum_{i=1}^n\delta_i\alpha_i$ in the adjoint representation. Then $\pi^{-1}(0)=\sum_{i=1}^n\delta_iC_i$ as a divisor \cite[Theorem 4]{Artin66}. For example, in type $A_n$, we have $\delta_i=1$ for all $i$, so $\pi^{-1}(0)$ is generically reduced. (Actually, $\pi^{-1}(0)$ is reduced in type $A_n$, which can be seen from iterated blow-ups.) In type $D_4$, we have $\delta_i=1,i\neq 2$ and $\delta_2=2$, so $\pi^{-1}(0)$ is not reduced.
\end{remark}

\subsection{Intersection pairing}\label{pairingsection}
We define an intersection pairing on $\wt{X}$ rigorously in this section. The main purpose of the discussion is to pave the way for the proof of \Cref{multofcomponent}. Note that the variety $\wt{X}$ is not projective, so it is not possible to define a pairing with good properties on the entire group of Cartier divisors on $\wt{X}$. Instead, we define the pairing on a smaller subset of divisors on $\wt{X}$ (cf. \Cref{ZSpairing}). We consider a slightly more general situation, following the exposition in \cite[\S 2.1]{Dolgachev}. Let $S$ be a smooth surface, and $Z$ be an effective Cartier divisor on $S$ with support projective over $\C$. We further assume that each irreducible component in the support of $Z$ is smooth. Though we will only care about the situation $S=\wt{X}$ and $Z=\pi^{-1}(0)$, it is easier to write down the results in a more general setting.

Let $\Div(S)$ be the group of Cartier divisors on $S$, which coincides with the group of Weil divisors on $S$ because $S$ is smooth. Let $\Div_{Z}(S)$ denote the subgroup of $\Div(S)$ that consists of divisors with support contained in the support of $Z$. We identify effective divisors on $S$ with closed, not necessarily reduced, subschemes of $S$ defined by the ideal sheaf $\hO_{S}(-D)$. Let $\F$ be a coherent sheaf on the projective scheme $Z$. We can define the Euler characteristic of $\F$ by $\hat{\chi}(\F):=\sum_i(-1)^i\dim_{\C}H^i(Z,\F)$ (cf. \cite[Exercise III.5.1]{Hartshorne}). Let $C$ be an effective divisor in $\Div_Z(S)$. For any invertible sheaf $\hL$ on $C$, we set
\begin{equation}\label{linebundledegree}
\deg(\hL):=\hat{\chi}(\hL)-\hat{\chi}(\hO_C).
\end{equation}
For any divisor $D$ on $S$, we set
\begin{equation}\label{CDintersect}
    C\cdot D:=\deg\hO_S(D)\otimes_{\hO_S}\hO_C.
\end{equation}
Let $P=\sum_i n_iP_i$ be a Weil divisor on $C$. Define $\deg P:=\sum_in_i$ as the degree of the divisor $P$. If $C$ is irreducible, then $C$ is a smooth projective curve by our assumption. By the Riemann-Roch Theorem \cite[Theorem IV.1.3]{Hartshorne}, we have
\[
   \deg(\hO_C(P))=\deg P.
\]
Let $f$ be a rational function on $S$. We denote by $\ddiv(f)$ its corresponding divisor. Any divisor which is equal to the divisor of a function is called a \emph{principal divisor}.

\begin{proposition}\cite[Example 2.4.9]{FultonIntersection}\label[proposition]{ZSpairing}
The pairing 
\begin{equation}\label{pairingdef}
    \Div_Z(S)\times\Div(S)\to\Z,\ (C,D)\mapsto C\cdot D 
\end{equation}
satisfies the following properties:
   \begin{enumerate}[label=(\arabic*)]
       \item\label{pair1} If $C$ and $D$ are smooth curves that meet transversely, then $C\cdot D=\#(C\cap D)$.
       \item\label{pair2} If $C,D\in\Div_Z(S)$, then $C\cdot D=D\cdot C$.
       \item\label{pair3} $C\cdot (D_1+D_2)=C\cdot D_1+C\cdot D_2$.
       \item\label{pair4} If $D$ is a principal divisor, then $C\cdot D=0$.
   \end{enumerate}
\end{proposition}
We apply \Cref{ZSpairing} to the situation $S=\wt{X},\ Z=\pi^{-1}(0)$. The properties in \Cref{ZSpairing} will be useful in the proof of \Cref{multofcomponent}.

\subsection{Poisson brackets on Kleinian singularities}\label{pbsubsection}
Consider the algebra $\C[u,v]$, with grading given by the degree of the polynomial function in variables $u,v$. There is a standard Poisson bracket of degree $(\minus 2)$ on $\C[u,v]$.
\begin{equation}\label{poissonbracket}
        \{f_1,f_2\}=\frac{\partial f_1}{\partial u}\frac{\partial f_2}{\partial v}-\frac{\partial f_1}{\partial v}\frac{\partial f_2}{\partial u}. 
\end{equation}
Let $\Gamma$ be a finite subgroup of $\SL_2(\C)$ and consider the Kleinian singularity $X:=\C^2/\Gamma$. Since the finite subgroup $\Gamma$ preserves the Poisson bracket \cref{poissonbracket}, the algebra of invariant functions $\C[X]=\C[u,v]^{\Gamma}$ is a graded Poisson subalgebra of $\C[u,v]$. We use $\deg f$ to denote the degree of a homogeneous element $f\in\C[X]$ and $\C[X]_d$ to denote the degree-$d$ component of $\C[X]$. 

Next, we write down the fundamental invariants in $\C[X]=\C[u,v]^{\Gamma}$, and calculate the Poisson brackets among the generators. We take the fundamental invariants of $\C[u,v]^{\Gamma}$ to be of the following form, adapted from \cite[\S1.3]{Dolgachev} and better aligned with the fundamental invariants obtained through Nakajima quiver varieties in \Cref{quiverapp}.

\begin{proposition}\label[proposition]{ADExyz}
Let $\Gamma\subset\SL_2(C)$ be a finite subgroup. The subalgebra $\C[u,v]^{\Gamma}$ is generated by the following polynomials.

$(A_n,\ n\geqslant 1)$ $x=u^{n+1}$,

$\ \ \ \ \ \ \ y=v^{n+1}$,

$\ \ \ \ \ \ \ z=uv$,

$\ \ \ \ \ \ \ $with relation $xy=z^{n+1}$.

$(D_n,\ n\geqslant 4\ \even)$ $x=u^2v^2$,

$\ \ \ \ \ \ \ y=-\frac{1}{4}(u^{2n-4}-2u^{n-2}v^{n-2}+v^{2n-4})$,

$\ \ \ \ \ \ \ z=\frac{1}{4}uv(u^{2n-4}-v^{2n-4})$,

$\ \ \ \ \ \ \ $with relation $xy(y-x^{\frac{n-2}{2}})=z^2$.

$(D_n,\ n\geqslant 5\ \odd)$ $x=u^2v^2$,

$\ \ \ \ \ \ \ y=\frac{1}{4}(u^{2n-4}-v^{2n-4})$,

$\ \ \ \ \ \ \ z=-\frac{1}{4}uv(u^{2n-4}-2u^{n-2}v^{n-2}+v^{2n-4})$, 

$\ \ \ \ \ \ \ $with relation $xy^2=z(z-x^{\frac{n-1}{2}})$.

$(E_6)$ $x=108\sqrt{3}(1+i)uv(u^4-v^4)$, 

$\ \ \ \ \ \ \ y=-108(u^8+14u^4v^4+v^8)$,

$\ \ \ \ \ \ \ z=648(\sqrt{3}(u^{12}-33u^8v^4-33u^4v^8+v^{12})-12iu^2v^2(u^4-v^4)^2)$,

$\ \ \ \ \ \ \ $with relation $z^2+zx^2+y^3=0$.

$(E_7)$ $x=-2^23(u^8+14u^4v^4+v^8)$,

$\ \ \ \ \ \ \ y=2^43^3(u^5v-uv^5)^2$,

$\ \ \ \ \ \ \ z=2^53^3uv(u^8-v^8)(u^8-34u^4v^4-v^8)$, 

$\ \ \ \ \ \ \ $with relation $x^3y+y^3+z^2=0$.

$(E_8)$ $x=-3^3uv(u^{10}+11u^5v^5-v^{10})$,

$\ \ \ \ \ \ \ y=-\frac{3^4}{2^2}(u^{20}+v^{20}-228(u^{15}v^5-u^5v^{15})+494u^{10}v^{10})$,

$\ \ \ \ \ \ \ z=\frac{3^6}{2^3}(u^{30}+v^{30}+522(u^{25}v^5-u^5v^{25})-10005(u^{20}v^{10}+u^{10}v^{20}))$, 

$\ \ \ \ \ \ \ $with relation $x^5+y^3+z^2=0$.
\end{proposition}

It is easy to see that the relations among the fundamental invariants $x,y,z$ in \Cref{ADExyz} are equivalent to those relations among the variables $x_1,y_1,z_1$ listed in \Cref{Kleinianreview} up to a change of coordinates. For example, in type $D_n$, with $n$ even, a change of coordinates can be given by $x_1=x,\ y_1=-2iy+ix^{\frac{n-2}{2}},\ z_1=2z$. In type $D_n$, with $n$ odd, a change of coordinates can be given by $x_1=x,\ y_1=2y,\ z_1=-2iz+ix^{\frac{n-1}{2}}$. In type $E_6$, a change of coordinates can be given by $x_1=\frac{\epsilon_8x}{\sqrt{2}},\ y_1=y,\ z_1=z+\frac{x^2}{2}$. 
The Poisson brackets among the fundamental invariants can be calculated either by hand or by software such as Mathematica. We summarize the results in the following proposition.

\begin{proposition}\label[proposition]{ADEpoissonbracket}
    The Poisson brackets among the fundamental invariants in \Cref{ADExyz} are given by
\begin{flalign*}
(A_n,\ n\geqslant 1)\qquad \{x,y\}&=(n+1)^2 z^n,\\
    \{x,z\}&=(n+1)x, \\
    \{y,z\}&=-(n+1)y.\\
(D_n,\ n\geqslant 4\ \even)\qquad \{x,y\}&=(4n-8)z, \\
    \{x,z\}&=(2n-4)x(2y-x^{\frac{n-2}{2}}), \\
    \{y,z\}&=(n-2)y(nx^{\frac{n-2}{2}}-2y). \\
(D_n,\ n\geqslant 5\ \odd)\qquad \{x,y\}&=(2n-4)(2z-x^{\frac{n-1}{2}}), \\
    \{x,z\}&=(4n-8)xy, \\
    \{y,z\}&=-(n-2)((n-1)x^{\frac{n-3}{2}}z+2y^2). \\
(E_6)\qquad \{x,y\}&=24\sqrt{3}(1+i)(2z+x^2), \\
    \{x,z\}&=-72\sqrt{3}(1+i)y^2, \\
    \{y,z\}&=48\sqrt{3}(1+i)xz. \\ 
(E_7)\qquad\{x,y\}&=-96z, \\
    \{x,z\}&=48(x^3+3y^2), \\
    \{y,z\}&=-144x^2y. \\
(E_8)\qquad\{x,y\}&=-120z, \\
    \{x,z\}&=180y^2, \\
    \{y,z\}&=-300x^4. &&
\end{flalign*}
\end{proposition}

\section{Anti-Poisson involutions and their fixed point loci}\label{APIsection}

We classify the conjugacy classes of anti-Poisson involutions of a Kleinian singularity and compute their fixed point loci in this section. We use $X=\C^2/\Gamma$ to denote a Kleinian singularity throughout. Recall from \Cref{pbsubsection} that $\C[X]=\C[u,v]^{\Gamma}$ is a graded Poisson algebra with grading given by the degree of the invariant polynomials in variables $u,v$. The fundamental invariants of $\C[X]$ are listed in \Cref{ADExyz}, and the Poisson brackets are given in \Cref{ADEpoissonbracket}.

\subsection{Anti-Poisson involutions}
\begin{definition}\label{api}
By an \emph{anti-Poisson involution} on $X$, we mean a graded algebra involution $\theta\colon \C[X]\to \C[X]$ such that
\begin{equation}\label{apiequation}
    \theta(\{f_1,f_2\})=-\{\theta(f_1),\theta(f_2)\},\ \forall\ f_1, f_2 \in \C[X].
\end{equation}
\end{definition}

\begin{example}\label{A1-1}
    Consider the Kleinian singularity of type $A_1\colon X=\C^2/\Gamma$ with $\Gamma=\Z/2\Z\simeq\{\pm I_2\}$. We have $\C[X]=\C[u,v]^{\Gamma}=\C[x=u^2,y=v^2,z=uv]\simeq\C[x,y,z]/(xy-z^2)$. The brackets among the generators are
    \[
      \{x,y\}=4z,\ \{x,z\}=2x,\ \{y,z\}=-2y.
    \]
    The graded automorphism $x\mapsto y,\ y\mapsto x,\ z\mapsto z$ is an anti-Poisson involution.
\end{example}

Similar to the anti-Poisson involution on Kleinian singularities (or any Poisson algebra), we can define the \emph{anti-symplectic involution} on a smooth symplectic variety $(N,\omega)$, which is an involution $\eta\colon N\to N$ such that $\eta^*\omega=-\omega$. The following proposition will be very useful in \Cref{liftsection}.

\begin{proposition}\label[proposition]{antisymtangent}
    Let $(N,\omega)$ be a smooth symplectic variety with an anti-symplectic involution $\eta$. Then the scheme-theoretic fixed point locus $N^{\eta}$ is either empty or a smooth Lagrangian subvariety.
\end{proposition}
\begin{proof}
    The smoothness of the scheme-theoretic fixed point locus $N^{\eta}$ is a corollary of Luna's slice theorem (cf. \cite[Proposition 5.8]{Luna04}). The fact that $N^{\eta}$ is a Lagrangian subvariety is an easy exercise in linear algebra, which we leave to the reader.
\end{proof}

\begin{definition}\label{apifixedlocus}
    Let $X=\C^2/\Gamma$ be a Kleinian singularity with anti-Poisson involution $\theta$. The \emph{scheme-theoretic fixed point locus} of $X$ under $\theta$ is the scheme
    \begin{equation}\label{fixedlocusdef}
        X^{\theta}=\Spec \C[X]/J,
    \end{equation}
    where $J=(\theta(f)-f|f\in \C[X])$.
\end{definition}

\begin{remark}
    The scheme-theoretic fixed point locus $X^{\theta}$ turns out to be reduced, as follows from a case-by-case calculation in \Cref{Anfixedloci,Dnevenfixedloci,Dnoddfixedloci,E6fixedloci,E7fixedlocus,E8fixedlocus}.
\end{remark}

We have $0\in X^{\theta}$ because $\theta$ is a graded automorphism. Similar to \Cref{antisymtangent}, one can show that the fixed point locus $X^{\theta}$ is either a singular Lagrangian subvariety (one-dimensional) or consists only of the singular point. 
The latter case only appears in the Type \RMN{3} anti-Poisson involutions for the Kleinian singularity of type $A_n$ when $n$ is odd (cf. \Cref{Anfixedloci}). 

If two anti-Poisson involutions $\theta_i,\ i=1,2$, are conjugate to each other by a graded Poisson automorphism $\sigma$, then $\sigma$ induces an isomorphism between the fixed point loci $X^{\theta_1}\xrightarrow[]{\sim}X^{\theta_2}$. We want to classify anti-Poisson involutions on $\C^2/\Gamma$ up to conjugation by graded Poisson automorphisms. It turns out that the conjugacy classes form a finite set; this is proved in \Cref{apiconjfinite}.

\subsection{Normalizers and diagram automorphisms}\label{normalizersection} In this section, we discuss the action of the normalizer $N_{\GL_2(\C)}(\Gamma)$ on the Kleinian singularity $\C^2/\Gamma$ and its induced action on the McKay graph associated to the Kleinian singularity. Recall that $\C[X]=\C[u,v]^{\Gamma}$ is a graded subalgebra of $\C[u,v]$. The grading induces $\C^*$-actions on both $\C^2$ and $X=\C^2/\Gamma$, and the quotient morphism $q\colon\C^2\to X$ is $\C^*$-equivariant. Let $X^0:=X\setminus\{0\}$ denote the smooth locus of $X$ and $q^0\colon \C^2\setminus\{0\}\to X^0$ denote the restriction of the quotient morphism. Then $q^0$ is an (analytic) universal covering map. Let us denote the group of graded automorphisms of the Kleinian singularity by $\Aut_{\gr}(\C^2/\Gamma)$. Given an element $g\in N_{\GL_2(\C)}(\Gamma)$, we obtain an element in $\Aut_{\gr}(\C^2/\Gamma)$ by
\begin{align*}
    \theta_g\colon\C[u,v]^{\Gamma}&\to\C[u,v]^{\Gamma}, \\
    f(-)&\mapsto g\cdot f:=f(g^{-1}-).
\end{align*}
Clearly, $\Gamma$ induces the identity automorphism of $\C[u,v]^{\Gamma}$.


\begin{proposition}\label[proposition]{autogrbij}
The map
\begin{equation}\label{autogrmap}
   N_{\GL_2(\C)}(\Gamma)/\Gamma\xrightarrow[]{}\Aut_{\gr}(\C^2/\Gamma),\ \ 
        g \mapsto \theta_g,
\end{equation}
is a group isomorphism.
\end{proposition}

\begin{proof}
Note that $\Gamma$ is a normal subgroup of $N_{\GL_2(\C)}(\Gamma)$. Hence \cref{autogrmap} is a well-defined group homomorphism. We only need to prove that it is a bijection. 

We first show that the map in \cref{autogrmap} is injective. Let $g\in N_{\GL_2(\C)}(\Gamma)$ such that $g$ acts trivially on $X=\C^2/\Gamma$. Then $g$ preserves the $\Gamma$-orbits of $\C^2$. Consider the linear subspaces 
    \[
    V_\gamma:=\{v\in\C^2|gv=\gamma v\},\ \gamma\in\Gamma.
    \]
    For each $v\in \C^2$, there exists a $\gamma\in\Gamma$ such that $gv=\gamma v$. Therefore, we have 
    \begin{equation}\label{gammaunion}
        \bigcup_{\gamma\in\Gamma} V_{\gamma}=\C^2.
    \end{equation}
    The left-hand side of \cref{gammaunion} is a finite union of linear subspaces of $\C^2$, so one of them, say $V_{\gamma_0}$, must be the entire $\C^2$. It follows that $g=\gamma_0\in\Gamma$.
    
     Next, we show that the map in \cref{autogrmap} is surjective. Let $\theta$ be a graded automorphism of $X$. Consider the restriction $\theta\colon X^0\to X^0$, and lift it to the universal covering space $\hat{\theta}\colon \C^2\setminus\{0\}\to\C^2\setminus\{0\}$, i.e., we have $\theta\circ q^{0}=q^{0}\circ\hat{\theta}$. The lift $\hat{\theta}$ is still holomorphic (cf. \cite[Theorem 4.9]{Forster}). We claim that $\hat{\theta}$ commutes with the $\C^*$-action on $\C^2\setminus\{0\}$ induced from that of $\C^2$. It suffices to show that $\lambda\hat{\theta}=\hat{\theta}\lambda$, for any $\lambda\in\C^*$, and we prove this pointwise. First, we have that
    \[
    q(\lambda^{-1}\hat{\theta}(\lambda v))=\theta(q(v))=q(\hat{\theta}(v)),\ \forall\ v\in\C^2\setminus\{0\},
    \]
    because $\theta$ and $q$ are both $\C^*$-equivariant. 
    Hence, $\lambda^{-1}\hat{\theta}(\lambda v)$ and $\hat{\theta}(v)$ lie in the same fiber $q^{-1}(q(\hat{\theta}(v)))$. Next, let $\lambda(t),\ 0\leqslant t\leqslant 1$ be a continuous path connecting $\lambda(0)=1$ and $\lambda(1)=\lambda$ in $\C^*$. Then consider the following continuous map
    \begin{align*}
        \phi\colon [0,1] &\rightarrow q^{-1}(q(\hat{\theta}(v))) \\
        t &\mapsto \lambda(t)^{-1}\hat{\theta}(\lambda(t)v).
    \end{align*}
    The range is a discrete space, so $\phi$ has to be constant. Thus, we obtain that $\lambda^{-1}\hat{\theta}(\lambda v)=\phi(1)=\phi(0)=\hat{\theta}(v)$. The final step is to extend the lifting $\hat{\theta}$ to a holomorphic map from $\C^2$ to $\C^2$ by Hartogs's Extension Theorem (cf. \cite[Theorem 2.3.2]{Hormander}). The only possible choice is to send $0$ to $0$ because $\hat{\theta}\colon\C^2\setminus\{0\}\to\C^2\setminus\{0\}$ commutes with the $\C^*$-action. Now we have a holomorphic automorphism $\hat{\theta}$ on $\C^2$ that commutes with the $\C^*$-action, which has to be linear and invertible by looking at its Taylor expansion. Thus, $\hat{\theta}$ must come from an element $g\in\GL_2(\C)$. 
    Note that $\hat{\theta}$ maps $\Gamma$-orbits to $\Gamma$-orbits, hence so does $g$. By an argument similar to the one used for the injectivity of \cref{autogrmap}, we obtain that $g\in N_{\GL_2(\C)}(\Gamma)$.
\end{proof}

\begin{remark}\label[remark]{detauto}
    Let $g\in N_{\GL_2(\C)}(\Gamma)$. By \cref{poissonbracket}, we obtain that
    \begin{equation}\label{det-1}
        \{g\cdot f_1,g\cdot f_2\}=\det(g)^{-1}g\cdot\{f_1,f_2\}.
    \end{equation}
    There are two kinds of graded automorphisms in $\Aut_{\gr}(\C^2/\Gamma)$ that are worth special attention. The first is graded Poisson automorphisms, which come from elements in $N_{\SL_2(\C)}(\Gamma)$. The second is graded anti-Poisson automorphisms, which come from elements in $N_{\GL_2(\C)}(\Gamma)$ with determinant $\minus1$, which we denote by $N_{\GL_2(\C)}(\Gamma)^-$. 
\end{remark}

Let us denote the set of anti-Poisson involutions on $X=\C^2/\Gamma$ by $\API(\C^2/\Gamma)$ and the conjugacy classes of anti-Poisson involutions by $\API(\C^2/\Gamma)/\sim$. An element $g\in N_{\GL_2(\C)}(\Gamma)$ gives rise to an anti-Poisson involution if and only if $\det(g)=-1$ and $g^2\in\Gamma$. \Cref{autogrbij} tells us that there is a bijection between sets
\begin{equation}\label{apibijsets}
    \{g\in N_{\GL_2(\C)}(\Gamma)^-|g^2\in\Gamma\}/\Gamma\xrightarrow[]{\sim}\API(\C^2/\Gamma).
\end{equation}
Considering the conjugation action of $N_{\SL_2(\C)}(\Gamma)$ on the left-hand side of \cref{apibijsets}, i.e., the set of $\Gamma$-cosets, then we further obtain the following bijection between sets
\begin{equation}\label{apiconjbijsets}
    \big(\{g\in N_{\GL_2(\C)}(\Gamma)^-|g^2\in\Gamma\}/\Gamma\big)\big/N_{\SL_2(\C)}(\Gamma)\xrightarrow[]{\sim}\API(\C^2/\Gamma)/\sim.
\end{equation}

\begin{corollary}\label[corollary]{apiconjfinite}
    There are finitely many anti-Poisson involutions on $\C^2/\Gamma$ up to conjugation by graded Poisson automorphisms.
\end{corollary}
\begin{proof}
    We prove the type $A$ case and the non-type $A$ cases separately. In type $A_n,$ with $n\geqslant 1$, we have 
    \[
    \Gamma=\Bigg\{\begin{pmatrix} \epsilon & 0 \\ 0 & \epsilon^{-1} \end{pmatrix}|\epsilon^{n+1}=1\Bigg\}
    \]
    is the cyclic group. We can calculate that 
    \[
    N_{\GL_2(\C)}(\Gamma)=\Bigg\{\begin{pmatrix}
        a & 0 \\
        0 & d 
    \end{pmatrix}|ad\neq0\Bigg\}\bigsqcup\Bigg\{\begin{pmatrix}
        0 & c \\
        b & 0
    \end{pmatrix}|bc\neq 0\Bigg\}.
    \]
    Following the bijection in \cref{apibijsets}, we have
    \[
    \Bigg\{\begin{pmatrix}
        a & 0 \\
        0 & -a^{-1} \\
    \end{pmatrix},\begin{pmatrix}
        0 & c \\
        c^{-1} & 0
    \end{pmatrix}\bigg\vert\ a^{2n+2}=1 \Bigg\}\bigg/\Gamma\xrightarrow[]{\sim}\API(\C^2/\Gamma).
    \]
    Note that 
    \[
    \Bigg\{\begin{pmatrix}
        0 & c \\
        c^{-1} & 0\end{pmatrix}\bigg|\ c\neq 0\Bigg\}
    \]
        forms a single orbit under the conjugation action of $N_{\SL_2(\C)}(\Gamma)$. Taking
        \[
        h=\begin{pmatrix}
        a' & 0 \\
        0 & a'^{-1}
    \end{pmatrix}\in N_{\SL_2(\C)}(\Gamma)
    \]
    such that $a'^2c=1$, we can calculate that 
    \[
    h\begin{pmatrix}
        0 & c \\
        c^{-1} & 0
    \end{pmatrix}h^{-1}=\begin{pmatrix}
        0 & 1 \\
        1 & 0
    \end{pmatrix}.
    \] Therefore, $\API(\C^2/\Gamma)/\sim$ is a finite set.

    For the non-type $A$ cases, we first prove that $N_{\SL_2(\C)}(\Gamma)$ is a finite group. Note that $N_{\SL_2(\C)}(\Gamma)$ is an algebraic group. It suffices to show that the connected component $N_{\SL_2(\C)}(\Gamma)^{\circ}$ is a singleton. This will follow from the following two observations.
    \begin{enumerate}
        \item The action of the connected group $N_{\SL_2(\C)}(\Gamma)^{\circ}$ on the finite group $\Gamma$ has to be trivial, so $N_{\SL_2(\C)}(\Gamma)^{\circ}$ is contained in the centralizer $Z_{\SL_2(\C)}(\Gamma)$.
        \item The tautological representation $\C^2$ is irreducible as a $\Gamma$-representation. According to Schur's Lemma, $Z_{\SL_2(\C)}(\Gamma)$ consists of scalar matrices of determinant $1$.
    \end{enumerate}
    Therefore, $N_{\SL_2(\C)}(\Gamma)^{\circ}$ consists only of the identity matrix because it is connected. Next, observe that there is a bijection $N_{\GL_2(\C)}(\Gamma)^-\xrightarrow[]{\sim}N_{\SL_2(\C)}(\Gamma)$ by sending 
    \[
    g\mapsto \begin{pmatrix}
        i & 0 \\
        0 & i
    \end{pmatrix}\cdot g.
    \]
    Thus $N_{\GL_2(\C)}(\Gamma)^-$ is also finite. It follows that the set of anti-Poisson automorphisms on $\C^2/\Gamma$ is finite. Hence $\API(\C^2/\Gamma)$, the set of anti-Poisson involutions, is also finite. As a corollary, the set of conjugacy classes $\API(\C^2/\Gamma)/\sim$ is also finite.
\end{proof}

Moreover, each conjugacy class of anti-Poisson involutions has a representative of simple form that can be written explicitly in terms of the generators of $\C[X]=\C[u,v]^{\Gamma}$. We list all the conjugacy classes of anti-Poisson involutions on Kleinian singularities in \Cref{AnAPIodd,AnAPIeven,DnevenAPI,DnoddAPI,E6API,E7API,E8API}. We also explain which element in $N_{\GL_2(\C)}(\Gamma)^-$ gives the corresponding anti-Poisson involution in the remark following each proposition.

We would also like to study how an anti-Poisson involution, or more generally, an element in the normalizer, acts on the McKay graph of $\Gamma$. Recall that $\Irr(\Gamma)=\{V_0,\cdots, V_n\}$ denotes the set of irreducible representations of $\Gamma$. Let $g\in N_{\GL_2(\C)}(\Gamma)$ and $V_i\in\Irr(\Gamma)$. We define $V_i^g$, the \emph{twist} of the representation $V_i$ under $g$, as follows. As a vector space $V_i^g=V_i$, but the action of $\Gamma$ is different and is given by
\begin{equation}\label{reptwist}
  \gamma\cdot v:=g\gamma g^{-1}(v),\ \gamma\in\Gamma,\ v\in V^g_i.
\end{equation}
It is easy to see that $V^g_i$ is also an irreducible $\Gamma$-representation. Moreover, $\forall g\in N_{\GL_2(\C)}(\Gamma)$, we have 
\begin{equation}
    \C^2 \to (\C^2)^g,\ \ v \mapsto gv
\end{equation}
is a $\Gamma$-equivariant isomorphism. Therefore, if $V_i\otimes\C^2\simeq\oplus_{j}V_j$ as $\Gamma$-representations, then  $V^g_i\otimes\C^2\simeq\oplus_{j}V^g_j$ as $\Gamma$-representations. 

An automorphism of a graph is a permutation of the vertices that preserves adjacency. Hence, the twist \cref{reptwist} induces an automorphism of the McKay graph of $\Gamma$, an extended Dynkin diagram. 
The twist of the trivial representation $V_0$ is still the trivial representation, no matter what $g$ is. Thus, we obtain an automorphism of the finite Dynkin diagram $\Delta_{\Gamma}$ labeling $\Gamma$ (removing the vertex $0$ and its adjacent edges). Together with the identification given in \Cref{autogrbij}, we obtain a group homomorphism 
\begin{equation}\label{diagramautohom}
    N_{\GL_2(\C)}(\Gamma)/\Gamma\xrightarrow{\sim}\Aut_{\gr}(\C^2/\Gamma) \to \Aut(\Delta_{\Gamma}).
\end{equation}
Note that $\Aut(\Delta_{\Gamma})$ is either trivial or $ \Z/2\Z$, except in type $D_4$, where $\Aut(\Delta_{\Gamma}) \simeq S_3$. Let
\begin{equation}
    \Inv(\Delta_{\Gamma}):=\{\tau\in\Aut(\Delta_{\Gamma})|\tau^2=\id\}
\end{equation}
denote the set of involutions in $\Delta_{\Gamma}$. Let $\Inv(\Delta_{\Gamma})/\sim$ denote the set of conjugacy classes of $\Inv(\Delta_{\Gamma})$ under the conjugation action of $\Aut(\Delta_{\Gamma})$.

Restricting \cref{diagramautohom} to the set of anti-Poisson involutions $\API(\C^2/\Gamma)$ induces a well-defined map between the sets of conjugacy classes
\begin{equation}\label{apidiagramauto}
\phi_{\Gamma}\colon\API(\C^2/\Gamma)/\sim \to \Inv(\Delta_{\Gamma})/\sim.
\end{equation}

\begin{remark}
The map in \cref{apidiagramauto} is a bijection except in type $A_n$ with $n$ odd. This follows from our analysis of the diagram involutions induced by the anti-Poisson involutions, as discussed in \Cref{AnAPIrem,DnevenAPIrem,DnoddAPIrem,E6APIrem,E7APIrem,E8APIrem}. 
\end{remark}

\subsection{Type $A_n$}\label{AnAPIsubsection}
The Poisson algebra of interest is $\C[x,y,z]/(xy-z^{n+1})$. Recall from \Cref{ADExyz} and \Cref{ADEpoissonbracket} that  the degrees of the generators are $\deg x=\deg y=n+1,\ \deg z=2$, and the Poisson brackets among them are given by
\[
\{x,y\}=(n+1)^2 z^n,\ \{x,z\}=(n+1)x,\ \{y,z\}=-(n+1)y.
\]

\begin{proposition}\label[proposition]{AnAPIodd}
    The Kleinian singularity of Type $A_n$, with odd $n>1$, has three conjugacy classes of anti-Poisson involutions.
    \begin{enumerate} [label=\Roman*.]
       \item $\theta(x)=y,\ \theta(y)=x,\ \theta(z)=z$.
       \item $\theta(x)=x,\ \theta(y)=y,\ \theta(z)=-z$.
    \item $\theta(x)=-x,\ \theta(y)=-y,\ \theta(z)=-z$.
    \end{enumerate}
    When $n=1$, the Poisson automorphism $x\mapsto \frac{1}{2}(x+y)-z,\ y\mapsto \frac{1}{2}(x+y)+z,\ z\mapsto \frac{1}{2}(x-y)$ conjugates the Type $\RMN{2}$ anti-Poisson involution to Type $\RMN{1}$, so there are two conjugacy classes.
\end{proposition}
\begin{proof}
    When $n=1$, we have $\Gamma=\{\pm I_2\}$ and $N_{\GL_2(\C)}(\Gamma)=\GL_2(\C)$. An element $g\in\GL_2(\C)$ gives rise to an anti-Poisson involution if and only if $\det(g)=-1$ and $g^2=\pm I_2$ by \Cref{autogrbij} and \Cref{detauto}. It follows that $g$ is diagonalizable, and its eigenvalues lie in $\{1,-1,i,-i\}$ with the requirement that the product of the two eigenvalues is $-1$. Then up to multiplication by $\Gamma$ and conjugation by $N_{\SL_2(\C)}(\Gamma)=\SL_2(\C)$, the element $g$ is either $\begin{pmatrix}
        0 & 1 \\
        1 & 0
    \end{pmatrix}$
    or $\begin{pmatrix}
        i & 0 \\
        0 & i
    \end{pmatrix}$, which correspond to Type $\RMN{1}$ and Type $\RMN{3}$ anti-Poisson involutions respectively. It is obvious that they are not conjugate to each other.

    When $n>1$, it is easy to verify that the involutions $\RMN{1}-\RMN{3}$ are anti-Poisson using the Poisson brackets given in \Cref{ADEpoissonbracket}. Conversely, we show that any anti-Poisson involution can be conjugated to one of Types $\RMN{1}-\RMN{3}$. First, notice that we have $\theta(z)=\pm z$ because $\theta$ is a graded involution, and we have $\C[X]_2=\C z$. Consider the operator defined by $\{z,\cdot\}$ on $\C[X]$. It is diagonalizable with integral eigenvalues, therefore produces a new grading on $\C[X]$, which we denote by $\deg_z:=\{z,\cdot\}$. Under this new grading, we have $\deg_z(x)=-(n+1),\ \deg_z(y)=n+1,\ \deg_z(z)=0$. When $\theta(z)=z$, it follows from 
    \begin{equation}\label{newdegz}
        \{z,\theta(f)\}=\{\theta(z),\theta(f)\}=-\theta(\{z,f\}),\ \forall f\in\C[X],
    \end{equation}
    that $\theta$ reverses the new grading $\deg_z$. Thus we must have
    \[
    \theta(x)=ay,\ \theta(y)=bx,\ a,b\in \C.
    \]
    Moreover, using that $\theta^2(x)=x,\ \theta^2(y)=y$, we get $ab=1$. It is easy to verify that $\theta: x\mapsto ay,\ y\mapsto a^{-1}x,\ z\mapsto z$ is an anti-Poisson involution on $\C[X]$ and its conjugate $\sigma\theta\sigma^{-1}$ is the Type $\RMN{1}$ anti-involution, where $\sigma$ is the Poisson automorphism $x \mapsto a'x,\ y\mapsto a'^{-1}y,\ z\mapsto z$ such that $a'^2=a$. 
    
    When $\theta(z)=-z$, similar to \cref{newdegz}, we can show that $\theta$ preserves the new grading $\deg_z$. Thus we must have
    \begin{equation}\label{TypeAaxby}
        \theta(x)=ax,\ \theta(y)=by,\ a,b\in \C.
    \end{equation}
    Moreover, using that $\theta^2(x)=x,\ \theta^2(y)=y$, we get $a^2=b^2=1$. Plugging \cref{TypeAaxby} into
    \[
    \theta(\{x,y\})=-\{\theta(x),\theta(y)\},
    \]
    we get $ab=1$. The two solutions $a=b=1$ and $a=b=-1$ correspond to Type $\RMN{2}$ and Type $\RMN{3}$ anti-Poisson involutions respectively.
    
    Finally, we show that Types $\RMN{1}-\RMN{3}$ anti-Poisson involutions are not conjugate to each other. Any graded Poisson automorphism has to send $z$ to a scalar multiple of $z$ because $\C[X]_2=\C z$, so it cannot conjugate an anti-Poisson involution sending $z$ to $-z$ (Types $\RMN{2},\ \RMN{3}$) to those sending $z$ to $z$ (Type $\RMN{1}$). Furthermore, the fixed point locus of Type $\RMN{3}$ anti-Poisson involution is the singular point, while the fixed point locus of Type $\RMN{2}$ anti-Poisson clearly contains more than one point. Hence, Types $\RMN{2}$ and $\RMN{3}$ anti-Poisson involutions are not conjugate to each other, either.
\end{proof}

\begin{proposition}\label[proposition]{AnAPIeven}
    The Kleinian singularity of Type $A_n$, with $n$ even, has two conjugacy classes of anti-Poisson involutions.
    \begin{enumerate}[label=\Roman*.]
        \item $\theta(x)=y,\ \theta(y)=x, \theta(z)=z$.
        \item $\theta(x)=-x,\ \theta(y)=y, \theta(z)=-z$.
    \end{enumerate}
\end{proposition}
\begin{proof}
    The proof is similar to that of \Cref{AnAPIodd}. We leave it to the reader.
\end{proof}

\begin{remark}\label[remark]{AnAPIrem}
    \begin{enumerate}
    \item According to \Cref{detauto}, every anti-Poisson involution on $X$ comes from an element in $N_{\GL_2(\C)}(\Gamma)$ with determinant $-1$. We list here the elements in the normalizer that give rise to the anti-Poisson involutions for type $A_n$. When $n$ is odd, the matrices $\begin{pmatrix}
        0 & 1 \\
        1 & 0
    \end{pmatrix},\ \begin{pmatrix}
        -1 & 0 \\
        0 & 1
    \end{pmatrix},\ \begin{pmatrix}
        \epsilon & 0 \\
        0 & -\epsilon^{-1}
    \end{pmatrix}$ give rise to anti-Poisson involutions of Types $\RMN{1},\RMN{2},\RMN{3}$ in \Cref{AnAPIodd} respectively, where $\epsilon$ is a $(2n+2)$-th primitive root of unity. When $n$ is even, the matrices $\begin{pmatrix}
        0 & 1 \\
        1 & 0
    \end{pmatrix},\ \begin{pmatrix}
        -1 & 0 \\
        0 & 1
    \end{pmatrix}$ give rise to anti-Poisson involutions of Types $\RMN{1},\RMN{2}$ in \Cref{AnAPIeven} respectively. 

    \item It is easy to see that the matrix $\begin{pmatrix}
        0 & 1 \\
        1 & 0
    \end{pmatrix}$ induces the automorphism of Dynkin diagram $A_n$ by swapping the vertices $i,n+1-i$ (in both $n$ even and odd cases), via twisting the irreducible representations of $\Gamma$ as defined in \cref{reptwist}. The matrices $\begin{pmatrix}
        -1 & 0 \\
        0 & 1
    \end{pmatrix},\ \begin{pmatrix}
        \epsilon & 0 \\
        0 & -\epsilon^{-1}
    \end{pmatrix}$ act trivially on the set $\Irr(\Gamma)$ because they commute with $\Gamma$, hence induce the trivial automorphisms on the Dynkin diagram $A_n$.

    \end{enumerate}
\end{remark}

A direct calculation gives the following description of the fixed point loci in Type $A_n$.

\begin{corollary}\label[corollary]{Anfixedloci}
    The fixed point loci of anti-Poisson involutions in Type $A_n$ are of the following forms.

    When $n$ is odd,
    \begin{enumerate} [label=\Roman*.]
        \item $J=(x-y)$. $X^{\theta}=\Spec\C[x,z]/(x^2-z^{n+1})$, the union of two copies of $\A^1$'s.
        \item $J=(z)$. $X^{\theta}=\Spec\C[x,y]/(xy)$, the union of two copies of $\A^1$'s.
        \item $J=(x,y,z)$. $X^{\theta}=\Spec\C$, a point.
    \end{enumerate}
    
    When $n$ is even,
    \begin{enumerate}[label=\Roman*.]
        \item $J=(x-y)$. $X^{\theta}=\Spec\C[x,z]/(x^2-z^{n+1})$, a cusp.
        \item $J=(x,z)$. $X^{\theta}=\Spec\C[y]\simeq\A^1$.
    \end{enumerate}
    Moreover, all the fixed point loci are reduced as schemes.
\end{corollary}


\subsection{Type $D_n,\ n$ even}\label{DnevenAPIsubsection}

The Poisson algebra of interest is $\C[X]=\C[x,y,z]/(xy(y-x^{\frac{n-2}{2}})-z^2)$. Recall from \Cref{ADExyz} and \Cref{ADEpoissonbracket} that the degrees of the generators are $\deg x=4,\ \deg y=2n-4,\ \deg z=2n-2$, and the Poisson brackets among them are given by
\[
\{x,y\}=(4n-8)z,\ \{x,z\}=(2n-4)x(2y-x^{\frac{n-2}{2}}),\ \{y,z\}=(n-2)y(nx^{\frac{n-2}{2}}-2y).
\]

\begin{proposition}\label[proposition]{DnevenAPI}
    The Kleinian singularity of Type $D_n$, with $n$ even, has two conjugacy classes of anti-Poisson involutions.
    \begin{enumerate} [label=\Roman*.]
        \item $\theta(x)=x,\ \theta(y)=y,\ \theta(z)=-z$.
        \item $\theta(x)=x,\ \theta(y)=x^{\frac{n-2}{2}}-y,\ \theta(z)=z$.
    \end{enumerate}
\end{proposition}
\begin{proof}
    When $n=4$, we have $\Gamma=\big\{\pm I_2,\ \pm\begin{pmatrix}
        i & 0 \\
        0 & -i
    \end{pmatrix},\ \pm\begin{pmatrix}
        0 & 1 \\
        -1 & 0
    \end{pmatrix},\ \pm\begin{pmatrix}
        0 & i \\
        i & 0
    \end{pmatrix}\big\}$. An element $g\in N_{\GL_2(\C)}(\Gamma)$ gives rise to an anti-Poisson involution if and only if $\det(g)=-1$ and $g^2\in\Gamma$ by \Cref{autogrbij} and \Cref{detauto}. Then up to multiplication by $\Gamma$ and conjugation by $N_{\SL_2(\C)}(\Gamma)$, the element $g$ is either $\begin{pmatrix}
        1 & 0 \\
        0 & -1
    \end{pmatrix},\ \begin{pmatrix}
        \epsilon_8 & 0 \\
        0 & \epsilon_8^3
    \end{pmatrix}$, where $\epsilon_8$ is an eighth primitive root of unity. Using the formulas of the generators $x,y,z$ given in \Cref{ADExyz}, we see that the two matrices correspond to the Type $\RMN{1}$ and the Type $\RMN{2}$ anti-Poisson involutions respectively. It is obvious that they are not conjugate to each other. 
    
    When $n>4$, it is easy to verify that the involutions $\RMN{1},\RMN{2}$ are anti-Poisson using the Poisson brackets given in \Cref{ADEpoissonbracket}.
    Conversely, let $\theta$ be an anti-Poisson involution of $\C[X]$. Note that $\deg z$ is not an integral linear combination of $\deg x$ and $\deg y$. 
    Thus we may assume
    \begin{equation*}
        \theta(x)=ax,\quad
        \theta(y)=by+b'x^{\frac{n-2}{2}},\quad
        \theta(z)=cz,\quad
    \end{equation*}
    as $\theta$ is graded. We have $\theta^2(x)=x,\ \theta^2(y)=y,\ \theta^2(z)=z$ as $\theta$ is an involution, which implies that $a^2=b^2=c^2=1$. From
    \begin{equation*}
        0=\theta(xy(y-x^{\frac{n-2}{2}})-z^2)=axy^2+ab(2b'-a^{\frac{n-2}{2}})yx^{\frac{n}{2}}+ab'(b'-a^{\frac{n-2}{2}})x^{n-1}-z^2,
    \end{equation*}
    we get $a=1,\ ab(2b'-a^{\frac{n-2}{2}})=-1,\ ab'(b'-a^{\frac{n-2}{2}})=0$ because $xy(y-x^{\frac{n-2}{2}})-z^2=0$ is the only relation of degree $4n-4$ in $\C[X]$ up to scaling. By plugging $a=1$ into the latter two equations, we get $b(2b'-1)=-1,\ b'(b'-1)=0$. There are only two solutions. The first is $b'=0,\ b=1$. Using
    \begin{equation}\label{Drxy}
        \{\theta(x),\theta(y)\}=-\theta(\{x,y\})
    \end{equation}
    as $\theta$ is anti-Poisson, we can further deduce that $c=-1$. This corresponds to the Type $\RMN{1}$ anti-Poisson involution. The second solution is $b'=1,\ b=-1$. With a similar procedure as above, we can obtain that $c=1$. This corresponds to the Type $\RMN{2}$ anti-Poisson involution.

    To see that Types $\RMN{1},\RMN{2}$ anti-Poisson involutions are non-conjugate, note that any graded automorphism has to send $z$ to a scalar multiple of $z$ because $\C[X]_{2n-2}=\C z$. Thus, an anti-involution sending $z$ to $-z$ (Type $\RMN{1}$) cannot be conjugate to that sending $z$ to $z$ (Type $\RMN{2}$).
\end{proof}

\begin{remark}\label[remark]{DnevenAPIrem}
    \begin{enumerate}[label=(\arabic*)]
    \item\label{Dnevenremtemp2} According to \Cref{autogrbij}, every anti-Poisson involution on $\C^2/\Gamma$ should come from an element in $N_{\GL_2(\C)}(\Gamma)^-$. We list here the elements in the normalizer that give rise to the anti-Poisson involutions for type $D_n$, with $n$ even. The matrix $\begin{pmatrix}
        -1 & 0 \\
        0 & 1
    \end{pmatrix}$ gives rise to the Type $\RMN{1}$ anti-Poisson involution in \Cref{DnevenAPI}, and the matrix $\begin{pmatrix}
        \epsilon & 0 \\
        0 & -\epsilon^{-1}
    \end{pmatrix}$ gives rise to the Type $\RMN{2}$ anti-Poisson involution in \Cref{DnevenAPI}, where $\epsilon$ is a $(4n-8)$-th primitive root of unity.
    \item To determine what kinds of automorphisms the matrices in \ref{Dnevenremtemp2} give rise to, we only need to look at their twists on the four one-dimensional representations of $\Gamma$ (including one trivial representation). Then it is not hard to discover that the matrix $\begin{pmatrix}
        -1 & 0 \\
        0 & 1
    \end{pmatrix}$ induces the trivial automorphism of the Dynkin diagram $D_n$, and the matrix $\begin{pmatrix}
        \epsilon & 0 \\
        0 & -\epsilon^{-1}
    \end{pmatrix}$ swaps two short arms of the Dynkin diagram $D_n$ (when $n$ is even). We can further determine which two vertices/$1$-dimensional representations are being swapped in type $D_4$, which is left as an exercise to the reader.
    \end{enumerate}
\end{remark}

\begin{corollary}\label[corollary]{Dnevenfixedloci}
    The fixed point loci of anti-Poisson involutions in Type $D_n$, with $n$ even, are of the following forms.
    \begin{enumerate} [label=\Roman*.]
        \item $J=(z)$. The fixed point locus $X^{\theta}=\Spec\C[x,y]/(xy(y-x^{\frac{n-2}{2}}))$ is the union of three copies of $\A^1$'s intersecting at the singular point.
        \item $J=(2y-x^{\frac{n-2}{2}})$. The fixed point locus $X^{\theta}=\Spec\C[x,z]/(x^{n-1}+4z^2)$ is a cusp.
    \end{enumerate}
    Moreover, both fixed point loci are reduced as schemes.
\end{corollary}

\subsection{Type $D_n,\ n$ odd}\label{DnoddAPIsubsection}

The Poisson algebra of interest is $\C[X]=\C[x,y,z]/(xy^2-z(z-x^{\frac{n-1}{2}}))$. Recall from \Cref{ADExyz} and \Cref{ADEpoissonbracket} that the degrees of the generators are $\deg x=4,\ \deg y=2n-4,\ \deg z=2n-2$, and the Poisson brackets among them are given by
\[
\{x,y\}=(2n-4)(2z-x^{\frac{n-1}{2}}),\ \{x,z\}=(4n-8)xy,\ \{y,z\}=-(n-2)((n-1)x^{\frac{n-3}{2}}z+2y^2).
\]

\begin{proposition}\label[proposition]{DnoddAPI}
    The Kleinian singularity of Type $D_n$, with $n$ odd, has two conjugacy classes of anti-Poisson involutions.
    \begin{enumerate} [label=\Roman*.]
        \item $\theta(x)=x,\ \theta(y)=-y,\ \theta(z)=z$.
        \item $\theta(x)=x,\ \theta(y)=y,\ \theta(z)=x^{\frac{n-1}{2}}-z$.
    \end{enumerate}
\end{proposition}
\begin{proof}
    The proof is similar to that of \Cref{DnevenAPI}. We leave it to the reader.
\end{proof}

\begin{remark}\label[remark]{DnoddAPIrem}
    \begin{enumerate}
    \item  According to \Cref{detauto}, every anti-Poisson involution on $\C^2/\Gamma$ comes from an element in $N_{\GL_2(\C)}(\Gamma)^-$. We list here the elements in the normalizer that give rise to the anti-Poisson involutions for type $D_n$, with $n$ odd. The matrix $\begin{pmatrix}
        \epsilon & 0 \\
        0 & -\epsilon^{-1}
    \end{pmatrix}$ gives rise to the Type $\RMN{1}$ anti-Poisson involution in \Cref{DnoddAPI}, where $\epsilon$ is a $(4n-8)$-th primitive root of unity. The matrix $\begin{pmatrix}
        -1 & 0 \\
        0 & 1
    \end{pmatrix}$ gives rise to the Type $\RMN{2}$ anti-Poisson involution in \Cref{DnoddAPI}.
    \item Similar to \Cref{DnevenAPIrem}, one can show that the matrix $\begin{pmatrix}
        \epsilon & 0 \\
        0 & -\epsilon^{-1}
    \end{pmatrix}$ induces the trivial automorphism of the Dynkin diagram $D_n$, and the matrix  
    $\begin{pmatrix}
        -1 & 0 \\
        0 & 1
    \end{pmatrix}$ swaps the two short arms of the Dynkin diagram $D_n$ (when $n$ is odd).
    \end{enumerate}
\end{remark}

\begin{corollary}\label[corollary]{Dnoddfixedloci}
    The fixed point loci of anti-Poisson involutions in Type $D_n$, with $n$ odd, are of the following forms.
    \begin{enumerate} [label=\Roman*.]
        \item $J=(y)$. The fixed point locus $X^{\theta}=\Spec\C[x,z]/(z(z-x^{\frac{n-1}{2}}))$ is the union of two copies of $\A^1$'s intersecting at the singular point.
        \item $J=(2z-x^{\frac{n-1}{2}})$. The fixed point locus $X^{\theta}=\Spec\C[x,y]/(x(4y^2+x^{n-2}))$ is the union of $\A^1$ and a cusp, intersecting at the singular point.
    \end{enumerate}
     Moreover, both fixed point loci are reduced as schemes.
\end{corollary}

We have finished the classification of conjugacy classes of anti-Poisson involutions in the classical types. Next, we consider the three exceptional types.

\subsection{Type $E_6$}\label{E6APIsubsection}

The Poisson algebra of interest is $\C[X]=\C[x,y,z]/(z^2+zx^2+y^3)$. Recall from \Cref{ADExyz} and \Cref{ADEpoissonbracket} that the degrees of the generators are $\deg x=6,\ \deg y=8,\ \deg z=12$, and the Poisson brackets among them are given by 
\[
\{x,y\}=24\sqrt{3}(1+i)(2z+x^2),\ \{x,z\}=-72\sqrt{3}(1+i)y^2,\ \{y,z\}=48\sqrt{3}(1+i)xz.
\]

\begin{proposition}\label[proposition]{E6API}
    The Kleinian singularity of Type $E_6$ has two conjugacy classes of anti-Poisson involutions.
    \begin{enumerate} [label=\Roman*.]
        \item $\theta(x)=-x,\ \theta(y)=y,\ \theta(z)=z$.
        \item $\theta(x)=x,\ \theta(y)=y,\ \theta(z)=-z-x^2$.
    \end{enumerate}
\end{proposition}
\begin{proof}
    It is easy to verify that the involutions $\RMN{1},\RMN{2}$ are anti-Poisson using the Poisson brackets given in \Cref{ADEpoissonbracket}. Conversely, let $\theta$ be an anti-Poisson involution of $\C[X]$. We may assume
    \[
    \theta(x)=ax,\ \theta(y)=by,\ \theta(z)=cz+c'x^2,
    \]
    as $\theta$ is graded. We have $\theta^2(x)=x,\ \theta^2(y)=y,\ \theta^2(z)=z$, which implies that $a^2=b^2=c^2=1$. From
    \begin{equation*}
    0=\theta(z^2+zx^2+y^3)=c^2z^2+c(2c'+a^2)zx^2+b^3y^3+c'(c'+a^2)x^4,
    \end{equation*}
    we get $b^3=1,\ c(2c'+a^2)=1$ using the relation $c^2=1$ and the fact that $z^2+zx^2+y^3=0$ is the only relation of degree $24$ in $\C[X]$ up to scaling. Together with $b^2=1$, we can deduce that $b=1$. Using
    \begin{equation*}
        \{\theta(x),\theta(z)\}=-\theta(\{x,z\})
    \end{equation*}
    as $\theta$ is anti-Poisson, we can further determine that $ac=-1$. Setting $a=-1,\ c=1$, we get $c'=0$ from $c(2c'+a^2)=1$, which corresponds to the Type \RMN{1} anti-Poisson involution. Setting $a=1,\ c=-1$, we get $c'=-1$ from $c(2c'+a^2)=1$, which corresponds to the Type \RMN{2} anti-Poisson involution.

    To see that Types $\RMN{1},\RMN{2}$ anti-Poisson involutions are non-conjugate, note that any graded automorphism has to send $x$ to a scalar multiple of $x$ because $\C[X]_6=\C x$. Thus, an anti-involution sending $x$ to $-x$ (Type $\RMN{1}$) cannot be conjugate to that sending $x$ to $x$ (Type $\RMN{2}$). 
\end{proof}

\begin{remark}\label[remark]{E6APIrem}
\begin{enumerate}[label=(\arabic*)]
    \item\label{E6APIremtemp2}  According to \Cref{detauto}, every anti-Poisson involution on $\C^2/\Gamma$ comes from an element in $N_{\GL_2(\C)}(\Gamma)^-$. We list here the elements in the normalizer that give rise to the anti-Poisson involutions for the Kleinian singularity of type $E_6$. The matrix $\begin{pmatrix}
        -1 & 0 \\
        0 & 1
    \end{pmatrix}$ gives rise to Type $\RMN{1}$ anti-Poisson involution in \Cref{E6API}, and the matrix $\begin{pmatrix}
        \epsilon_8 & 0 \\
        0 & -\epsilon_8^{-1}
    \end{pmatrix}$ gives rise to Type $\RMN{2}$ anti-Poisson involution in \Cref{E6API}, where $\epsilon_8$ is an eighth primitive root of unity. It is not quite obvious to see that both $\begin{pmatrix}
        -1 & 0 \\
        0 & 1
    \end{pmatrix},\ \begin{pmatrix}
        \epsilon_8 & 0 \\
        0 & -\epsilon^{-1}_8
    \end{pmatrix}$ are in $N_{\GL_2(\C)}^-(\Gamma)$. One way is to verify that they preserve $\C[X]=\C[u,v]^{\Gamma}$ by using the formulas of the fundamental invariants $x,y,z$ given in \Cref{ADExyz} and deduce that they lie in $N_{\GL_2(\C)}^-(\Gamma)$ by \Cref{autogrbij}, which is faster than computing that they normalize $\Gamma$ directly.

    \item Similar to \Cref{DnevenAPIrem}, we can obtain the actions of the matrices in \ref{E6APIremtemp2}  on the Dynkin diagram $E_6$ by studying their twists on the three one-dimensional representations of the finite subgroup $\Gamma\subset\SL_2(\C)$ of type $E_6$, or by consulting the character table of $\Gamma$ given in \cite[Example 5.2.3]{Dolgachev}. The matrix $\begin{pmatrix}
        -1 & 0 \\
        0 & 1
    \end{pmatrix}$ induces the trivial action on the Dynkin diagram $E_6$, and the matrix $\begin{pmatrix}
        \epsilon_8 & 0 \\
        0 & -\epsilon^{-1}_8
    \end{pmatrix}$
    acts on the finite Dynkin diagram $E_6$ by swapping the two longer arms.
    \end{enumerate}
\end{remark}

\begin{corollary}\label[corollary]{E6fixedloci}
    The fixed point loci of anti-Poisson involutions in Type $E_6$ are of the following forms.
    \begin{enumerate} [label=\Roman*.]
        \item $J=(x)$. The fixed point locus $X^{\theta}=\Spec\C[y,z]/(z^2+y^3)$ is a cusp.
        \item $J=(2z+x^2)$. The fixed point locus $X^{\theta}=\Spec\C[x,y]/(x^4-4y^3)$ is a cusp.
    \end{enumerate}
    Moreover, both fixed point loci are reduced as schemes.
\end{corollary}

\subsection{Type $E_7$}\label{E7APIsubsection}

The Poisson algebra of interest is $\C[X]=\C[x,y,z]/(x^3y+y^3+z^2)$. Recall from \Cref{ADExyz} and \Cref{ADEpoissonbracket} that the degrees of the generators are $\deg x=8,\ \deg y=12,\ \deg z=18$, and the Poisson brackets among them are given by 
\[
\{x,y\}=-96z,\ \{x,z\}=48(x^3+3y^2),\ \{y,z\}=-144x^2y.
\]

\begin{proposition}\label[proposition]{E7API}
    The Kleinian singularity of Type $E_7$ has a unique anti-Poisson involution (up to conjugation).
    \[
    \theta(x)=x,\ \theta(y)=y,\ \theta(z)=-z
    \]
\end{proposition}
\begin{proof}
    It is easy to verify the above automorphism is indeed an anti-Poisson involution of $\C[X]$. Conversely, let $\theta$ be an anti-Poisson involution of $\C[X]$. We may assume
    \[
    \theta(x)=ax,\ \theta(y)=by,\ \theta(z)=cz,
    \]
    as $\theta$ is graded. $\theta^2=\id$ gives us that $a^2=b^2=c^2=1$. From
    \begin{equation*}
       0=\theta(x^3y+y^3+z^2)=abx^3y+b^3y^3+c^2z^2=abx^3y+by^3+z^2
    \end{equation*}
    we get $ab=b=1$, hence $a=b=1$, because $x^3y+y^3+z^2=0$ is the only relation of degree $36$ in $\C[X]$ up to scaling. Using the condition that $\theta$ is anti-Poisson, we have
    \[
    \{\theta(x),\theta(y)\}=-\theta(\{x,y\}),
    \]
    from which we can determine that $c=-1$.
\end{proof}

\begin{remark}\label[remark]{E7APIrem}
    The matrix $\begin{pmatrix}
        -1 & 0 \\
        0 & 1
    \end{pmatrix}\in N_{\GL_2(\C)}(\Gamma)^-$ gives rise to the unique anti-Poisson involution of the Kleinian singularity of type $E_7$. There is no nontrivial automorphism on the Dynkin diagram $E_7$, thus the matrix $\begin{pmatrix}
        -1 & 0 \\
        0 & 1
    \end{pmatrix}$ induces the trivial action on the Dynkin diagram $E_7$.
\end{remark}

\begin{corollary}\label[corollary]{E7fixedlocus}
    The vanishing ideal of the anti-Poisson involution in Type $E_7$ is $J=(z)$. The corresponding fixed point locus is $X^{\theta}=\Spec\C[x,y]/(y(x^3+y^2))$, which is the union of $\A^1$ and a cusp, intersecting at the singular point. Moreover, the fixed point locus is reduced as a scheme.
\end{corollary}

\subsection{Type $E_8$}\label{E8APIsubsection}

In Type $E_8$, the Poisson algebra of interest is $\C[X]=\C[x,y,z]/(x^5+y^3+z^2)$. Recall from \Cref{ADExyz} and \Cref{ADEpoissonbracket} that the degrees of the generators are $\deg x=12,\ \deg y=20,\ \deg z=30$, and the Poisson brackets among them are given by 
\[
\{x,y\}=-120z,\ \{x,z\}=180y^2,\ \{y,z\}=-300x^4.
\]

\begin{proposition}\label[proposition]{E8API}
    The Kleinian singularity of Type $E_8$ has a unique anti-Poisson involution (up to conjugation).
    \[
    \theta(x)=x,\ \theta(y)=y,\ \theta(z)=-z
    \]
\end{proposition}
\begin{proof}
   The proof repeats that of \Cref{E7API} for type $E_7$.
\end{proof}

\begin{remark}\label[remark]{E8APIrem}
    The matrix $\begin{pmatrix}
        i & 0 \\
        0 & i
    \end{pmatrix}\in N_{\GL_2(\C)}(\Gamma)^-$ gives rise to the unique anti-Poisson involution of the Kleinian singularity of type $E_8$. There is no nontrivial automorphism on the Dynkin diagram $E_8$, thus the matrix $\begin{pmatrix}
        i & 0 \\
        0 & i
    \end{pmatrix}$ induces the trivial action on the Dynkin diagram $E_8$.
\end{remark}

\begin{corollary}\label[corollary]{E8fixedlocus}
    The vanishing ideal of the anti-Poisson involution in Type $E_8$ is $J=(z)$, and the corresponding fixed point locus $X^{\theta}=\Spec\C[x,y]/(x^5+y^3)$ is a cusp. Moreover, the fixed point locus is reduced as a scheme.
\end{corollary}

\section{Nakajima quiver varieties}\label{quiver}

We will describe the preimages of the fixed point loci from \Cref{Anfixedloci,Dnevenfixedloci,Dnoddfixedloci,E6fixedloci,E7fixedlocus,E8fixedlocus} through the realization of Kleinian singularities as Nakajima quiver varieties. Let us briefly review the construction of Nakajima quiver varieties, following \cite{Nakajima94, Nakajima98}. 

\subsection{Representations of quivers}\label{repofquiver} By a quiver $Q=(I,\Omega)$ we mean an oriented graph with the vertex set $I=\{0,1,\cdots,n\}$ and the edge set $\Omega$. For an arrow $a\colon i\to j\in \Omega$, we denote by $t(a)=i\in I$ its initial vertex and by $h(a)=j\in I$ its terminal vertex. The quivers we care about are extended Dynkin quivers of types $ADE$. Let $V=(V_i)_{i\in I}$ be a collection of finite-dimensional vector spaces corresponding to the vertices. The dimension of $V$ is a row vector
\[
\dim V=(\dim V_0,\dim V_1,\cdots,\dim V_n) \in \Z^{n+1}_{\geqslant 0}.
\]

For two collections of vector spaces $V$ and $W$ with $\delta=\dim V,\ w=\dim W$, we can consider the \textit{framed representations} of the quiver $Q$.
\[
R(Q,\delta,w):=\bigoplus_{a\in \Omega}\Hom_{\C}(V_{t(a)},V_{h(a)})\oplus\bigoplus_{i\in I}\Hom_{\C}(W_i,V_i).
\]

Next, we consider the doubled quiver $\bar{Q}=(I,\bar{\Omega})$, where $\bar{\Omega}=\Omega\sqcup\Omega^*$, and $\Omega\xrightarrow[]{\sim}\Omega^*,\ a\mapsto a^*$, is the reversal of the orientation. The space of framed representations of the doubled quiver $\bar{Q}$ is
\begin{equation}\label{quiverrepdef}
    M(\delta,w):=\bigoplus_{a\in\Omega}\big(\Hom(V_{t(a)},V_{h(a)})\oplus\Hom(V_{h(a)},V_{t(a)})\big)\oplus\bigoplus_{i\in I}\Hom(W_i,V_i)\oplus\bigoplus_{i\in I}\Hom(V_i,W_i).
\end{equation}
An element of $M(\delta,w)$ can be considered as a quadruple $(B,B^*,l,k)$ with
\begin{equation}\label{quadruple}
\begin{aligned}
    B&=(B_a)_{a\in\Omega}\in\bigoplus_{a\in\Omega}\Hom(V_{t(a)},V_{h(a)}),\ \ B^*=(B_{a^*})_{a\in\Omega}\in\bigoplus_{a\in\Omega}\Hom(V_{h(a)},V_{t(a)}), \\
    l&=(l_i)_{i\in I}\in\bigoplus_{i\in I}\Hom(W_i,V_i),\qquad k=(k_i)_{i\in I}\in\bigoplus_{i\in I}\Hom(V_i,W_i).
\end{aligned}
\end{equation}
We have $M(\delta,w)\simeq T^*R(Q,\delta,w)$ by identifying $\Hom_{\C}(V_i,V_j)^*$ with $\Hom_{\C}(V_j,V_i)$ via the trace form $(A,B):=\tr(AB)$. Then the cotangent bundle $M(\delta,w)$ is naturally endowed with a symplectic form
\begin{equation}\label{quiversymform}
  \omega((B_a,B_{a^*},l,k),(B'_a,B'_{a^*},l',k'))=\sum_{i\in I}\bigg(\sum_{\substack{a\in\Omega \\ h(a)=i}}\tr(B_aB'_{a^*})-\sum_{\substack{a\in\Omega \\ t(a)=i}}\tr(B_{a^*}B'_a)+\tr(l_ik'_i-k_il'_i)\bigg),
\end{equation}
where $(B_a,B_{a^*},l,k),(B'_a,B'_{a^*},l',k')$ are regarded as tangent vectors.

Set $\GL(\delta):=\prod_{i\in I}\GL(V_i)$ and $\gl(\delta):=\prod_{i\in I}\gl(V_i)$. Given $g=(g_i)_{i\in I}\in\GL(\delta)$, it acts on the symplectic vector space $M(\delta,w)$ via
\begin{equation}\label{actionquadruple}
    g\cdot B_a=g_{h(a)} B_a g_{t(a)}^{-1},\ g\cdot B_{a^*}=g_{t(a)}B_{a^*} g_{h(a)}^{-1},\ g\cdot l_i=g_il_i,\  g\cdot k_i=k_ig_i^{-1}.
\end{equation}
The $\GL(\delta)$-action on $M(\delta,w)$ preserves the symplectic form in \cref{quiversymform}, and is Hamiltonian. We can consider the corresponding moment map $\mu\colon M(\delta,w)\to\gl(\delta)$, where $\mu=(\mu_i)_{i\in I}$, with
\begin{equation}\label{momentmap}
\mu_i(B,B^*,l,k)=\sum_{\substack{a\in\Omega,\\ h(a)=i}}B_aB_{a^*}-\sum_{\substack{a\in\Omega,\\ t(a)=i}}B_{a^*}B_a+l_ik_i.
\end{equation}
For a reference to \cref{momentmap}, one may consult \cite[\S 3.i]{Nakajima98}. The equation $\mu=0$ is also called the ``ADHM equation'', referring to Atiyah-Drinfeld-Hitchin-Manin. A \emph{character} is a group homomorphism $\chi\colon\GL(\delta)\to\C^*$. It is given by
\[
\chi(g)=\prod_{i\in I}\det(g_i)^{\chi_i},\ \text{where}\ g=(g_i)_{i\in I},\ \chi=(\chi_i)_{i\in I}\in\Z^I
\]
because the only nontrivial characters of $\GL_m(\C)$ are integral powers of the determinant map and $\GL(\delta)$ is just a product of $\GL_m(\C)$'s. Let $\lambda=(\lambda_i)_{i\in I}\in\gl(\delta)^{\GL(\delta)}$. Then the fiber $\mu^{-1}(\lambda)$ is $\GL(\delta)$-invariant. Consider the \emph{$\chi$-semi-invariants} of weight $m$, 
\begin{equation}\label{semim}
    \C[\mu^{-1}(\lambda)]^{\GL(\delta),m\chi}:=\{f\in\C[\mu^{-1}(\lambda)]|f(g^{-1}(B,B^*,l,k))=\chi(g)^m f((B,B^*,l,k))\}.
\end{equation}
The direct sum of $\C[\mu^{-1}(\lambda)]^{\GL(\delta),m\chi}$ over $m\in\Z_{\geqslant0}$ is a graded algebra. A \emph{Nakajima quiver variety} is the GIT quotient
\begin{equation}\label{GITquotient}
     \M^{\lambda}_{\chi}(\delta,w):=\mu^{-1}(\lambda)/\!/^{\chi}\GL(\delta)=\Proj\bigoplus_{m\geqslant 0}\C[\mu^{-1}(\lambda)]^{\GL(\delta),m\chi}.
\end{equation}
When $\chi$ equals the trivial character $(0,\cdots,0)$, we have
\[
 \M^{\lambda}_0(\delta,w)=\mu^{-1}(\lambda)\gitquo\GL(\delta)=\Spec\C[\mu^{-1}(\lambda)]^{\GL(\delta)}.
\]
Moreover, there is a natural projective morphism 
\begin{equation}\label{piproj}
   \pi\colon\M_{\chi}^{\lambda}(\delta,w)\to\M_0^{\lambda}(\delta,w).
\end{equation}

In this paper, we mainly focus on the zero fiber of the moment map. We denote $\M_{\chi}(\delta,w):=\M^{0}_{\chi}(\delta,w)$ for simplicity. We use $[B,B^*,l,k]$ to denote the $\GL(\delta)$-orbit of a point $(B,B^*,l,k)$.

\begin{remark}
    If the framing vector $w$ is zero, we recover the usual representation space of the doubled quiver. The $\GL(\delta)$-action and the moment map $\mu$ are obtained by setting $l=k=0$. We follow our previous notation to denote this space as $M(\delta,0)$.
\end{remark}

Next, we discuss the $\chi$-semistable points. If $f$ is a semi-invariant, then the set of nonvanishing locus $\mu^{-1}(0)_f$ is a $\GL(\delta)$-invariant open subset of $\mu^{-1}(0)$. We define the \textit{$\chi$-semistable locus}
\begin{equation}\label{chisemistable}
    \mu^{-1}(0)^{\chi-ss}:=\{u\in\mu^{-1}(0)|\exists f\in\C[\mu^{-1}(0)]^{\GL(\delta),m\chi}\ \text{for}\ m>0\ \text{such that}\ f(u)\neq 0\}.
\end{equation}

We have $\mu^{-1}(0)^{\chi-ss}=\bigcup_f\mu^{-1}(0)_f$, where $f$ runs through all the nonzero semi-invariants. King first interprets $\chi$-semistability of representations in terms of subrepresentations in \cite{King94}. Nakajima gives the following description of the semistable locus for the character $\chi=(-1,\cdots,-1)$.

\begin{proposition}[\cite{Nakajima98}, Lemma 3.8]\label[proposition]{stTFAE}
    Let $\chi=(-1,\cdots,-1)$. Then a point $(B,B^*,l,k)$ is $\chi$-semistable if and only if $\im(l)$ generates the entire collection of vector spaces $V=(V_i)_{i\in I}$ under the action of $(B,B^*)$.
\end{proposition}

\begin{corollary}\label[corollary]{stcor}
    Let $\chi=(-1,\cdots,-1)$. Then we have
    \begin{enumerate}
        \item\label{stcor1} $\GL(\delta)$ acts on $\mu^{-1}(0)^{\chi-ss}$ freely, and each $\GL(\delta)$-orbit is closed. 
        \item We have $\M_{\chi}(\delta,w)\simeq\mu^{-1}(0)^{\chi-ss}\gitquo\GL(\delta)$, and it is smooth.
        \item The symplectic form \cref{quiversymform} induces a symplectic form on $\M_{\chi}(\delta,w)$.
\end{enumerate}
\end{corollary}

\subsection{Applications to Kleinian singularities}\label{quiverapp} Take $Q=(I,\Omega)$ to be an extended Dynkin quiver of types $ADE$, i.e., the vertex set $I$ is given by the affine simple roots and the edge set $\Omega$ is given by the extended Dynkin diagram with some orientation. Let $\alpha_0,\alpha_1,\cdots,\alpha_n$ denote the set of simple affine roots. Then there is a unique relation $\sum_{0\leq i\leq n}\delta_i\alpha_i=0$ where $\delta_i$ are nonnegative integers with $\delta_0=1$. Following the notation in \Cref{repofquiver}, we consider the space $M(\delta,w)$ with the dimension and framing vectors 
\begin{equation}\label{dimvec}
    \delta=(\delta_0,\delta_1,\cdots,\delta_n),\ w=(1,0,\cdots,0),
\end{equation}
i.e., there is a single framing of dimension $1$ attached to the extended vertex $0$. We have $l=(l_0,0,\cdots,0),\ k=(k_0,0,\cdots, 0)$. Moreover, the moment map \cref{momentmap} becomes
\begin{equation}\label{ADEmomentmap}
    \begin{aligned}
    \mu_i(B_a,B_{a^*},l,k)=\sum_{\substack{a\in\Omega,\\ h(a)=i}}B_aB_{a^*}-\sum_{\substack{a\in\Omega,\\ t(a)=i}}B_{a^*}B_a+\delta_{i,0}l_0k_0,\ i\in I,
    \end{aligned}
\end{equation}
where $\delta_{i,0}$ means the Kronecker delta. Setting \cref{ADEmomentmap} equal to zero, we can further get the ADHM equation
\begin{equation}\label{ADEmomentmap0}
    \sum_{\substack{a\in\Omega,\\
    h(a)=i}}B_aB_{a^*}-\sum_{\substack{a\in\Omega,\\ t(a)=i}}B_{a^*}B_a=l_0k_0=0,\ 0\leq i\leq n.
\end{equation}
by summing $\sum_{i=0}^n\tr(\mu_i)=0$ and using the fact that $V_0$ is one-dimensional. Moreover, the $\chi$-semistability condition in \Cref{stTFAE} implies that if $(B,B^*,l_0,k_0)\in\mu^{-1}(0)^{\chi-ss}$, then $l_0\neq 0,\ k_0=0$. 

\begin{example}\label{D4quiver} For the extended Dynkin quiver $\widetilde{D}_4$, we have $\delta=(1,1,2,1,1),\ w=(1,0,0,0,0)$. The representation space $M(\delta,w)$ is illustrated in \Cref{quiverD4}.

\begin{mtable}[ht]
    \centering
\begin{tabular}{c}
    \begin{tikzpicture}
    \tikzset
        {
        vertex/.style={circle, inner sep=0pt, outer sep=0pt, minimum width=0.8cm},
        frame/.style={rectangle,minimum height=0.7cm,
        minimum width=0.7cm},
        }

        \node at (-2.5,1.4) [vertex,draw] (V0) {$V_0$};
        \node at (-2.5,-1.4) [vertex,draw] (V1) {$V_1$};
        \node at (0,0) [vertex,draw] (V2) {$V_2$};
        \node at (2.5,1.4) [vertex,draw] (V3) {$V_3$};
        \node at (2.5,-1.4) [vertex,draw] (V4) {$V_4$};
        
        \node at (-2.5,3) [frame,draw] (W0) {$W_0$};
        
        \draw[->] (-2.57,1.8) to node [left]{$k_0$} (-2.57,2.6);
        \draw[->] (-2.43,2.65) to node [right]{$l_0$} (-2.43,1.85);
        
        \draw [->] (V0) edge [bend left=10] node [above]{$B_{2\la 0}$} (V2);
        \draw [->] (V2) edge [bend left=10] node [below]{$B^*_{0\la 2}$} (V0);
        \draw [->] (V1) edge [bend left=10] node [above]{$B_{2\la 1}$} (V2);
        \draw [->] (V2) edge [bend left=10] node [below]{$B^*_{1\la 2}$} (V1);
        \draw [->] (V2) edge [bend left=10] node [above]{$B_{3\la 2}$} (V3);
        \draw [->] (V3) edge [bend left=10] node [below]{$B^*_{2\la 3}$} (V2);
        \draw [->] (V2) edge [bend left=10] node [above]{$B_{4\la 2}$} (V4);
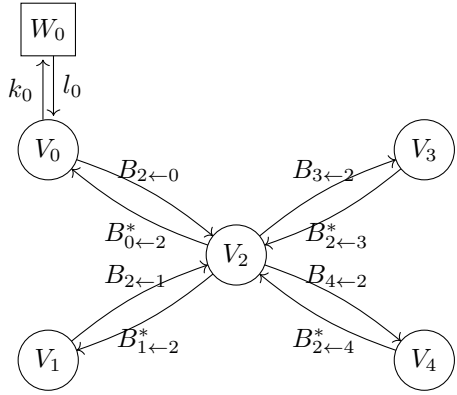
        \draw [->] (V4) edge [bend left=10] node [below]{$B^*_{2\la 4}$} (V2);
    \end{tikzpicture}
\end{tabular}
\caption{Extended Dynkin quiver $\widetilde{D}_4$}
\label{quiverD4}
\end{mtable}
\end{example}

Next, we relate Nakajima quiver varieties of extended Dynkin quivers to Kleinian singularities and their minimal resolutions.

\begin{lemma}\label[lemma]{woframing}
Let $Q=(I,\Omega)$ be an extended Dynkin quiver. Then we have
\begin{equation}\label{isowoframing}
    \M_0(\delta,w)\xrightarrow[]{\sim}\M_0(\delta,0).
\end{equation}
\end{lemma}
\begin{proof}
    We denote the moment map for the space $M(\delta,w)$ (resp., $M(\delta,0)$) by $\mu$ (resp., $\bar{\mu}$). From \cref{ADEmomentmap0}, we see that
    \[
    \C[\mu^{-1}(0)]=\C[\bar{\mu}^{-1}(0)][k^*_0,l^*_0]/(k^*_0l^*_0).
    \]
    The scalar subgroup $\C^*\subset\GL(\delta)$ acts on $M(\delta,w)$ by $\lambda\cdot(B,B^*,l_0,k_0)=(B,B^*,\lambda l_0,\lambda^{-1}k_0)$. Thus, we have 
    \[
    \C[\mu^{-1}(0)]^{\C^*}=\C[\bar{\mu}^{-1}(0)],
    \]
    which implies that
    \[
    \C[\bar{\mu}^{-1}(0)]^{\GL(\delta)}\xrightarrow{\sim}\C[\mu^{-1}(0)]^{\GL(\delta)}.
    \]
\end{proof}

Recall that a basis of the \emph{path algebra} $\C\bar{Q}$ consists of the paths in the doubled quiver $\bar{Q}$, including a trivial path $e_i$ for each vertex $i$. There is a natural grading on $\C\bar{Q}$ given by the number of arrows in a path. Consider the \emph{preprojective algebra}
\[
\Pi^{0}:=\C\bar{Q}/(\sum_{a\in\Omega}[a,a^*]).
\]

\begin{theorem}[\cite{CBH98}, Lemma 1.1, Corollary 3.5]\label{pathalgebra}
Let $\Gamma\subset\SL_2(\C)$ be a finite subgroup and $Q=(I,\Omega)$ the corresponding extended Dynkin quiver. There is a graded algebra isomorphism
\[e_0\Pi^{0}e_0\simeq\C[u,v]^{\Gamma}.
\]
\end{theorem}

There is also a natural grading on $\C[M(\delta,w)]$ given by the degree of the polynomial functions, which is equivalent to the following $\C^*$-action on $M(\delta,w)$:
\begin{equation}\label{torusaction}
    t\cdot(B,B^*,l_0,k_0):=(t^{-1}B,t^{-1}B^*,t^{-1}l_0,t^{-1}k_0),\ t\in\C^*.
\end{equation}
Note that this grading passes to $\C[\M_0(\delta,w)]=\C[\mu^{-1}(0)]^{\GL(\delta)}$ because the ADHM equation \cref{ADEmomentmap0} is homogeneous.

\begin{theorem}[\cite{CBH98}, Lemmas 8.1, 8.5, Theorem 8.10]\label{quiverinv}
There is a graded isomorphism of algebras
\begin{equation}
\begin{aligned}
    e_0\Pi^{0}e_0&\xrightarrow[]{\sim}\C[\M_0(\delta,0)],
\end{aligned} 
\end{equation}
given by sending a cycle $p=a_1\cdots a_m$ to the trace function $\tr(B_{a_1}\cdots B_{a_m})$.
\end{theorem}

Combining \Cref{woframing}, \Cref{pathalgebra}, and \Cref{quiverinv}, we obtain that
\begin{corollary}\label[corollary]{wframinginv}
There is an isomorphism of graded algebras 
\begin{equation}\label{gradedisoframing}
    C[\M_0(\delta,w)]\simeq\C[u,v]^{\Gamma}.
\end{equation}
Moreover, $\C[\M_0(\delta,w)]$ is generated by trace functions of paths that start and end at the vertex $0$.
\end{corollary} 

We know that $\M_{\chi}(\delta,w)=\mu^{-1}(0)^{\chi-ss}\gitquo\GL(\delta)$ is a smooth symplectic variety from \Cref{stcor}. By the fact that symplectic resolutions are minimal in dimension $2$, we obtain the following theorem.
\begin{theorem}[\cite{Nakajima98}, Proposition 3.24]\label{sympres}
  Let $\chi=(-1,\cdots,-1)$. The projective morphism
    \[
    \pi\colon\M_{\chi}(\delta,w)\to\M_0(\delta,w)
    \]
    in \cref{piproj} gives the minimal resolution of Kleinian singularities. 
\end{theorem}

Since $\dim V_0=1$, we may identify $\End(V_0)$ with $\C$ by taking traces. The graded isomorphism in \cref{gradedisoframing} allows us to pick out the generators of $\C[X]\simeq\C[\mu^{-1}(0)]^{\GL(\delta)}$ in terms of trace functions (cf. \Cref{Dngeneratorprop,E6generatorprop,E7generatorprop,E8generatorprop}), which is essential to the construction of the lifts in \Cref{Anquiver,Dnquiver,E6quiver,E7quiver,E8quiver}.

\subsection{Lifting anti-Poisson involutions to the minimal resolutions}\label{liftsection}
Recall that $\pi\colon\wt{X}\to X$ denotes the minimal resolution of a Kleinian singularity. We would like to describe the preimage $\pi^{-1}(X^{\theta})$ by answering the following questions.
\begin{enumerate}[label=(\roman*)]
    \item\label{q1} What are the irreducible components?
    \item\label{q2} How do the irreducible components intersect each other?
    \item\label{q3} Is $\pi^{-1}(X^{\theta})$ (generically) reduced?
\end{enumerate}

Our main approach is through lifting the anti-Poisson involution $\theta$ to the minimal resolution $\wt{X}$. By a \emph{lift} of $\theta$, we mean an anti-symplectic involution $\tilde{\theta}\colon\tilde{X}\to\tilde{X}$ such that 
\begin{equation}\label{liftcondition}
    \pi\circ\ttheta=\theta\circ\pi.
\end{equation}
Note that the action of $\ttheta$ is determined by its action on the open dense subset $\pi^{-1}(X\setminus\{0\})$, and $\pi$ is an isomorphism away from the exceptional fiber. Therefore, any automorphism $\ttheta$ of $\wt{X}$ satisfying \cref{liftcondition} is automatically an anti-symplectic involution.

Note that the lift $\ttheta$ always exists and is unique due to the definition and uniqueness of the minimal resolution $\pi\colon\wt{X}\to X$. However, to analyze the action of $\ttheta$ on $\wt{X}$, especially the action on the exceptional fiber $\pi^{-1}(0)_{\red}=\bigcup_{1\leqslant i\leqslant n}C_i$ with $C_i\simeq\pr^1$, we need an explicit construction of the lift $\ttheta$. We use \Cref{liftconditions} as a guideline to construct the lifts in \Cref{Anquiver,Dnquiver,E6quiver,E7quiver,E8quiver}. Recall the symplectic form $\omega(-,-)$ defined on $M(\delta,w)$ in \cref{quiversymform}.



Let $x,y,z$ be the generators of $\C[X]\simeq\C[\mu^{-1}(0)]^G$. Then they are trace functions of paths that start and end at the vertex $0$ of the underlying extended Dynkin quiver, according to \Cref{wframinginv}. 
We comment that these trace functions $x,y,z$ will be calculated explicitly later in \Cref{Anquiver,Dnquiver,E6quiver,E7quiver,E8quiver}.

\begin{proposition}\label[proposition]{liftconditions} 
    Let $\theta$ be an anti-Poisson involution of $X$. Let $\Theta\colon M(\delta,w)\to M(\delta,w),\ u\mapsto\Theta(u)$ be an anti-symplectic linear automorphism satisfying the following conditions.
    \begin{enumerate}[label=(\arabic*)]
    \item\label{lift2} There is an automorphism $g\mapsto g^{\Theta}$ of $\GL(\delta)$ such that $\Theta g\Theta^{-1}$ equals the action of $g^{\Theta}$ on $M(\delta,w)$.
    \item\label{lift3}  $\Theta(x)=\theta(x),\ \Theta(y)=\theta(y),\ \Theta(z)=\theta(z)$, where $x,y,z$ are generators of $\C[X]$ given by trace functions of paths starting and ending at the vertex $0$. 
    
\end{enumerate}
Then $\Theta$ descends to an anti-symplectic involution $\ttheta\colon\wt{X}\to\wt{X},\ [u]\mapsto[\Theta(u)]$. Moreover, $\ttheta$ is a lift of $\theta$.
\end{proposition}
\begin{proof}

We show that condition \ref{lift2} implies the following three statements.
\begin{enumerate}[label=(1\alph*)]
    \item\label{lifta} $\Theta$ preserves $\mu^{-1}(0)$. The natural $\GL(\delta)$-action on $M(\delta,w)$ induces the $\gl(\delta)$-action on $M(\delta,w)$ and the condition \ref{lift2} implies that the conjugation $A^{\Theta}:=\Theta A\Theta^{-1}$ is obtained by permuting the factors of $A=(A_i)_{i\in I}\in\gl(\delta)$. One can compute that $\mu(\Theta(u))=-\mu(u)^{\Theta}$, and the claim follows.
    \item\label{liftb} $\Theta$ preserves $\mu^{-1}(0)^{\chi-ss}$. Recall that $\chi(g)=\prod_{i\in I}\det(g_i)^{-1}$. One can show that $\chi((\cdot)^{\Theta})=\chi(\cdot)$ by the condition \ref{lift2}, hence $\Theta(f)\in\C[\mu^{-1}(0)]^{\GL(\delta),m\chi}$ iff $f\in\C[\mu^{-1}(0)]^{\GL(\delta),m\chi}$. It follows that $\Theta(u)$ is $\chi$-semistable iff $u$ is (cf. \cref{chisemistable}).
    \item\label{liftc} $\Theta$ normalizes the $\GL(\delta)$-action because the conjugation $g^{\Theta}=\Theta g\Theta^{-1}$ still lies in $\GL(\delta)$. 
\end{enumerate}
Together with the condition \ref{lift3}, we see that $\Theta$ descends to the anti-Poisson involution $\theta$ on $X\simeq\mu^{-1}(0)\gitquo\GL(\delta)$ and the anti-symplectic involution $\ttheta$ on $\wt{X}\simeq\mu^{-1}(0)^{ss}\gitquo\GL(\delta)$. We automatically have that $\pi\circ\ttheta=\theta\circ\pi$, hence $\ttheta$ is a lift of $\theta$.
\end{proof}

Consider the scheme-theoretic fixed point locus $\wt{X}^{\ttheta}$. It is a smooth Lagrangian subvariety of $\wt{X}$ (cf. \Cref{antisymtangent}). It is easy to show that
\begin{equation}\label{fixedlocilift}
    \pi^{-1}(X^{\theta})_{\red}=\pi^{-1}(0)_{\red}\cup\widetilde{X}^{\ttheta}.
\end{equation}
The advantage of the lift $\ttheta$ is that the smooth Lagrangian $\wt{X}^{\ttheta}$ is much easier to analyze than $\pi^{-1}(X^{\theta})_{\red}$. We study the variety $\pi^{-1}(X^{\theta})_{\red}$ in \Cref{liftsectionq1,liftsectionq2}, and the scheme structure of $\pi^{-1}(X^{\theta})$ in \Cref{liftsectionq3}.

Now we address the three questions asked at the beginning of \Cref{liftsection} separately.

\subsubsection{Question \ref{q1}}\label{liftsectionq1}

Recall that the natural grading on $\C[M(\delta,w)]$ given by the degree of polynomial functions induces the following $\C^*$-action on $M(\delta,w)$, as described in \cref{torusaction}:
\begin{equation*}
    t\cdot(B,B^*,l_0,k_0):=(t^{-1}B,t^{-1}B^*,t^{-1}l_0,t^{-1}k_0),\ t\in\C^*.
\end{equation*}
The symplectic form $\omega$ in \cref{quiversymform} is transformed into $t^{-2}\omega$ under the $\C^*$-action. The $\C^*$-action commutes with the $\GL(\delta)$-action, and the subsets $\mu^{-1}(0),\ \mu^{-1}(0)^{\chi-ss}$ are $\C^*$-stable. It follows that the $\C^*$-action descends to $X\simeq\mu^{-1}(0)\gitquo\GL(\delta)$ and $\wt{X}\simeq\mu^{-1}(0)^{\chi-ss}\gitquo\GL(\delta)$. Moreover, the minimal resolution $\pi\colon\wt{X}\to X$ is $\C^*$-equivariant.

There is an easier way to observe this $\C^*$-action on $X$. Recall the grading on $\C[X]=\C[u,v]^{\Gamma}$ given by the degree of the polynomial functions on $u,v$. We know from \Cref{wframinginv} that $\C[\mu^{-1}(0)]^{\GL(\delta)}\simeq\C[u,v]^{\Gamma}$ is a graded algebra isomorphism. Therefore, the $\C^*$-action on $X$ induced from the grading of $\C[u,v]^{\Gamma}$ is the same as the one descending from \cref{torusaction}. Let $x,y,z$ denote the generators of $\C[X]=\C[u,v]^{\Gamma}$ given in \Cref{ADExyz}. We have
\begin{equation}\label{XCstar}
    t\cdot(x_0,y_0,z_0)=(t^{-\deg x}x_0,t^{-\deg y}y_0,t^{-\deg z}z_0),\ (x_0,y_0,z_0)\in X.
\end{equation}
Let $L\subset X^{\theta}$ be a one-dimensional irreducible component. From \Cref{Anfixedloci,Dnevenfixedloci,Dnoddfixedloci,E6fixedloci,E7fixedlocus,E8fixedlocus}, we know that $L$ is either $\A^1$ or a cusp, and it is $\C^*$-stable. Moreover, one can deduce that the $\C^*$-action is transitive on the open part $L\setminus\{0\}$ by looking at the coordinate ring of $L$. Consider the variety 
\begin{equation}\label{tildeLdef}
    \tilde{L}:=\overline{\pi^{-1}(L\setminus\{0\})}\subset\wt{X}^{\ttheta},
\end{equation}
which is smooth because it is an irreducible component of the smooth Lagrangian $\wt{X}^{\ttheta}$. We will further show that $\tilde{L}\simeq\A^1$ in \Cref{LA1}. Pick an arbitrary point $s\in L\setminus\{0\}$. It doesn't matter which $s$ we choose because the $\C^*$-action on $L\setminus\{0\}$ is transitive. Set
\begin{equation}\label{toruss0}
     s_0:=\lim_{t\to 0}t^{-1}\cdot\pi^{-1}(s)\in\tilde{L}\cap\pi^{-1}(0),
\end{equation}
which exists because the composition $\tilde{L}\xhookrightarrow[]{i}\pi^{-1}(L)\xrightarrow[]{\pi} L$ is proper and the limit point $\lim_{t\to 0} t^{-1}\cdot s=0$ exists.

\begin{proposition}\label[proposition]{LA1}
    We have an isomorphism of varieties 
    \[
    \tilde{L}\xrightarrow[]{\sim
    }\A^1
    \]
    with the limit point $s_0$ being mapped to the origin.
\end{proposition}
\begin{proof}
   We first show that the morphism $f:=\pi\circ i\colon\tilde{L}\to L$ is quasi-finite. It suffices to show that the fiber $f^{-1}(0)$ consists of finitely many points as $f$ is an isomorphism outside of $f^{-1}(0)$. Note that the morphism $f$ is of finite type, hence quasi-compact. Therefore, it is enough to show that $\dim f^{-1}(0)=0$. If $\dim f^{-1}(0)>0$, then $\dim f^{-1}(0)=1$. It follows that $f^{-1}(0)=\tilde{L}$ as $\tilde{L}$ is irreducible, a contradiction.
   
   Next, note that the morphism $f$ is also proper, hence finite. It follows that $\tilde{L}$ is affine as an affine scheme over the affine scheme $L$. Moreover, $\tilde{L}$ is smooth as an irreducible component of the smooth Lagrangian $\wt{X}^{\ttheta}$, and rational because it contains $\pi^{-1}(L\setminus\{0\})\simeq\C^*$ as an open dense subset. It follows that $\tilde{L}$ is isomorphic to $\A^1$ with finitely many points removed.

   Let $U:=\pi^{-1}(L\setminus\{0\})\cup\{s_0\}\subset\tilde{L}$. Then it admits an open embedding into $\A^1$. More importantly, $U$ is equipped with a contracting $\C^*$-action to the point $s_0$ from \cref{torusaction,toruss0}. By a version of Luna's slice theorem \cite[Theorem 5.4]{Luna04}, we can identify $U$ with its tangent space, i.e., $U\simeq T_{s_0}U\simeq\A^1$. This implies that
   \[
   U=\tilde{L}\simeq\A^1,
   \]
   because a proper open subset of $\A^1$ cannot be isomorphic to $\A^1$, which can be seen by comparing the invertible functions on $U$ and $\A^1$.
\end{proof}

Thus, the irreducible components of $\pi^{-1}(X^{\theta})$ are $\pr^1$'s that are irreducible components of $\pi^{-1}(0)$ and $\A^1$'s that are normalizations of the irreducible components of $X^{\theta}$. Denote the $\pr^1$'s by $C_1,\cdots,C_n$ and the $\A^1$'s by $\tilde{L}_1,\cdots,\tilde{L}_m$, where $m$ is the number of one-dimensional irreducible components of $X^{\theta}$. We would like to point out that the variety $\wt{X}^{\ttheta}$ may contain not only all of the $\tilde{L}_j$'s, but also some of the $C_i$'s, which is the situation in most of the cases.

\subsubsection{Question \ref{q2}}\label{liftsectionq2}
We study how the irreducible components of $\pi^{-1}(X^{\theta})$ intersect each other. The $\A^1$'s do not intersect each other because they are irreducible components of the smooth variety $\wt{X}^{\ttheta}$. The $\pr^1$'s intersect each other according dually to the Dynkin diagram labeling $\Gamma$. It remains to determine how the $\A^1$'s intersect the $\pr^1$'s. A detailed answer will be given in \Cref{isolatedtransversal}. 

The lift $\ttheta$ preserves $\pi^{-1}(0)$ because $\theta$ fixes the singular point $0$. Note that $\ttheta$ is of order $2$, so $\ttheta$ either preserves a component $C_i$, or swaps it with another component $C_j$. Moreover, the intersection points of the $\A^1$'s and the $\pr^1$'s must lie in $\pi^{-1}(0)^{\ttheta}$ (cf. \cref{toruss0} and \Cref{LA1}). \Cref{liftoncomponent} will be useful in determining $\pi^{-1}(0)^{\ttheta}$.

\begin{lemma}\label[lemma]{liftoncomponent}
    \begin{enumerate}[label=(\arabic*)]
        \item\label{autop1} If $\ttheta$ swaps two components $C_i$ and $C_j$, then $C_i\cap C_j$ is either empty, or a single point contained in $\pi^{-1}(0)^{\ttheta}$.
        \item\label{autop2} If $\ttheta$ preserves a component $C_i$, then $\ttheta$ either acts trivially on $C_i$, or has exactly two fixed points.
        \item\label{autop3} If $\ttheta$ preserves two components $C_i$ and $C_j$ such that $C_i\cap C_j\neq\emptyset$, then $\ttheta$ acts trivially on one of the components, and has exactly two fixed points on the other component.
    \end{enumerate}
\end{lemma}
\begin{proof}
    \begin{enumerate}[label=(\arabic*)]
        \item Note that if $C_i\cap C_j$ is non-empty, then it is a single point. The rest follows.
        \item Any automorphism of $C_i\simeq\pr^1$ is given by a M$\ddot{o}$bius transformation (cf. \cite[Exercise I. 6.6]{Hartshorne}). It follows that an involution of $\pr^1$ either acts trivially or has exactly two fixed points.
        \item We denote by $p$ the unique point in $C_i\cap C_j$. The tangent map $\mathrm{d}_p\ttheta\colon T_p\wt{X}\to T_p\wt{X}$ is an anti-symplectic linear involution of the two-dimensional symplectic vector space $T_p\wt{X}$ with two distinct eigenvalues $1,-1$. By assumption, $T_pC_i,T_pC_j$ are eigenspaces of the linear map $\mathrm{d}_p\ttheta$. It follows that $\mathrm{d}_p\ttheta$ acts on one of $T_pC_i,T_pC_j$ by eigenvalue $1$, and acts on the other by eigenvalue $-1$. The rest follows from \ref{autop2}.
    \end{enumerate}
\end{proof}

We emphasize that when we say that $\ttheta$ preserves a component, we mean that $\ttheta(C_i)\subset C_i$; when we say that $\ttheta$ fixes a component, we mean that $\ttheta$ fixes $C_i$ pointwise, i.e. acts trivially. A direct application of \ref{autop3} in \Cref{liftoncomponent} gives \Cref{liftoncomponentcor}.

\begin{corollary}\label[corollary]{liftoncomponentcor}
    Consider a subset $\{C_{i'}\}_{i'\in I'}\subset\{C_i\}_{i\in I}$. Assume that
    \begin{enumerate}
        \item The lift $\ttheta$ preserves each component $C_{i'},\ i'\in I'$.
        \item The $C_{i'}$'s have pairwise transversal intersection according dually to a Dynkin diagram of type $A$.
    \end{enumerate}
    Then the $\ttheta$-fixed and non-fixed components appear alternately.
\end{corollary}

\Cref{isolatedtransversal} follows from \Cref{LA1}, \Cref{liftoncomponent}, and the fact that $\wt{X}^{\ttheta}$ is a smooth Lagrangian (cf. \Cref{antisymtangent}).

\begin{proposition}\label[proposition]{isolatedtransversal}
\begin{enumerate}[label=(\arabic*)]
     \item\label{it1} $\pi^{-1}(0)^{\ttheta}$ is a disjoint union of $\pr^1$'s and finitely many points. The number of isolated points in $\pi^{-1}(0)^{\ttheta}$ equals the number of one-dimensional irreducible components of $X^{\theta}$. 
     \item\label{it2} For each isolated point $p\in\pi^{-1}(0)^{\ttheta}$, there is a unique $\tilde{L}_j$ such that $\tilde{L}_j\cap\pi^{-1}(0)^{\ttheta}=\{p\}$. 
     \item\label{it3} If $\tilde{L}_j\cap C_i\neq\emptyset$, then $\tilde{L}_j\cap C_i=\tilde{L}_j\cap\pi^{-1}(0)^{\ttheta}=\{p\}$. Moreover, $\tilde{L}_j$ and $C_i$ meet transversely at $p$. 
\end{enumerate}
\end{proposition}

To determine exactly which component $C_i$ intersects $\tilde{L}_j$, we need a description of $C_i$ in terms of quiver data. Recall $B_a\colon V_{t(a)}\to V_{h(a)}$. Set 
\[
\epsilon B_a:=\begin{cases}
    B_a,\ a\in\Omega\\
    -B_a,\ a\in\Omega^* .
\end{cases}
\]
Let $(B_{a})_{\substack{a\in\bar{\Omega} \\ t(a)=i}}$ denote the collection of the entries of $(B,B^*)$ with source $V_i$. Fix the character $\chi=(-1,\cdots,-1)$. Pick a point $(B,B^*,l_0,k_0)\in\mu^{-1}(0)^{\chi-ss}$, and consider the following chain of linear maps for each $i=1,\cdots,n$.
\begin{equation}\label{complex}
  V_i\xrightarrow[]{(B_{a})_{a\in\bar{\Omega},\ t(a)=i}}\bigoplus_{\substack{a\in\bar{\Omega} \\ t(a)=i}}V_{h(a)}\xrightarrow[]{\epsilon (B_{a^*})_{a\in\bar{\Omega},\ t(a)=i}}V_i.
\end{equation}
The composition of the maps in \cref{complex} is $0$ by the ADHM equation (cf. \cref{ADEmomentmap0}). Moreover, the second map in \cref{complex} is surjective (cf. \Cref{stTFAE}). Recall from \cref{McKay2}, we have that
\[
\dim\bigoplus_{\substack{a\in\bar{\Omega}\\ t(a)=i}}V_{h(a)}=2\dim V_i.
\]
Therefore, \cref{complex} is an exact sequence if and only if the first map in \cref{complex} is injective, i.e. 
\[
\rank(B_{a})_{\substack{a\in\bar{\Omega}\\t(a)=i}}=\dim V_i, 
\]
where $\rank(B_{a})_{\substack{a\in\bar{\Omega}\\t(a)=i}}$ denotes the rank of the linear map 
$(B_{a})_{\substack{a\in\bar{\Omega}\\t(a)=i}}\colon V_i\to\bigoplus_{\substack{a\in\bar{\Omega} \\ t(a)=i}}V_{h(a)}$. 

\begin{proposition}\label[proposition]{Ciequation}
    Define
    \[
    C_i=\bigg\{[B,B^*,l_0,0]\in\M_{\chi}(\delta,w)\big| \rank(B_{a})_{\substack{a\in\bar{\Omega}\\ t(a)=i}}<\dim V_i \bigg\},\ 1\leq i\leq n.
    \]
    Then $\pi^{-1}(0)_{\red}=\bigcup_{1\leq i\leq n}C_i$, and $C_i\simeq\pr^1$ corresponds to the vertex $i$ in $\Delta_{\Gamma}$, the Dynkin diagram labeling $\Gamma$.
\end{proposition}

To prove \Cref{Ciequation}, we also need to consider quiver varieties without framing. Let $M(\delta,0)$ be the representation space of an extended Dynkin quiver of types $ADE$ without framing. Consider the character $\chi'=(\sum_{i=1}^n\delta_i,-1\cdots,-1)$. Then $\pi_{\chi'}\colon\M_{\chi'}(\delta,0)\to\M_0(\delta,0)\simeq\C^2/\Gamma$ gives the minimal resolution of a Kleinian singularity (\cite[Corollary 3.2 and Proposition 3.10]{Kronheimer}). Note that we use $\pi_{\chi'}$ to distinguish it from the minimal resolution $\pi\colon\M_{\chi}(\delta,w)\to\M_0(\delta,w)$ given in \Cref{sympres}. Next, define the dual representation space $M(\delta^*,0)$ by replacing each $V_i$ in the definition \cref{quiverrepdef} by $V^*_i$. Consider the character $-\chi'=(-\sum_{i=1}^n\delta_i,1\cdots,1)$. Similarly to the case of $\pi_{\chi'}$, we have that $\pi_{-\chi'}\colon\M_{-\chi'}(\delta^*,0)\to\M_0(\delta^*,0)\simeq\C^2/\Gamma$ is the minimal resolution of a Kleinian singularity. For a linear map $B_a\colon V_{t(a)}\to V_{h(a)}$, we denote by $B'_{a^*}\colon V_{h(a)}^*\to V_{t(a)}^*$ its dual. There is an obvious linear isomorphism of vector spaces $M(\delta,0)\xrightarrow[]{\sim} M(\delta^*,0),\ (B,B^*,0,0)\mapsto ((B^*)',B',0,0)$. It naturally induces the following commutative diagram.
\begin{equation}\label{comdiag}
    \begin{tabular}{c}
    \begin{tikzpicture}
    \def\a{1.5} \def\b{1}
    \path
    (-\a,0) node (A) {$\M_{\chi'}(\delta,0)$}      
    (\a,0) node (B) {$\M_{-\chi'}(\delta^*,0)$}
    (-\a,-\b) node (C) {$\M_0(\delta,0)$}
    (\a,-\b) node (D) {$\M_0(\delta^*,0)$};
    \begin{scope}[nodes={midway,scale=.75}]
    \draw [->] (A)--(B) node[above]{$\phi_1$};
    \draw [->] (C)--(D) node[above] {$\phi_0$};
    \draw [->] (A)--(C) node[left]
    {$\pi_{\chi'}$};
    \draw [->] (B)--(D) node[right]{$\pi_{-\chi'}$};
    \end{scope}
    \end{tikzpicture} 
    \end{tabular}
\end{equation}
Moreover, $\phi_1$ and $\phi_0$ are both isomorphisms of varieties (cf. \cite[Lemma 10.29]{Kirillovjr}). It follows that $\phi_1$ restricts to an isomorphism on the exceptional fibers $\pi_{\chi'}^{-1}(0)\xrightarrow[]{\sim}\pi_{-\chi'}^{-1}(0)$. \Cref{Cilemma} below, due to Nakajima, is essential to the proof of \Cref{Ciequation}. 

\begin{theorem}\cite[Theorem 5.10]{Nakajima96}\label{Cilemma}
Define
\[
  C'_i:=\bigg\{[B',B'^*,0,0]\in\M_{-\chi'}(\delta^*,0)\big| \rank(B'_{a^*})_{\substack{a\in\bar{\Omega}\\ t(a)=i}}<\dim V^*_i \bigg\},\ 1\leq i\leq n.
\]
Then $\pi_{-\chi'}^{-1}(0)_{\red}=\bigcup_{1\leq i\leq n}C'_i$, and $C'_i\simeq\pr^1$ corresponds to the vertex $i$ in $\Delta_{\Gamma}$, the Dynkin diagram labeling $\Gamma$.
\end{theorem}

\begin{proof}[Proof of \Cref{Ciequation}]
Consider the projection map 
\begin{equation*}
    \begin{aligned}
       M(\delta,w)&\to M(\delta,0)\\(B,B^*,l_0,k_0)&\mapsto(B,B^*,0,0). 
    \end{aligned}
\end{equation*}
We show that it induces an isomorphism of varieties
\begin{equation}\label{chichi'}
    \begin{aligned}
        \varphi_1\colon \M_{\chi}(\delta,w)&\to\M_{\chi'}(\delta,0) \\
        [B,B^*,l_0,0]&\mapsto[B,B^*,0,0].
    \end{aligned}
\end{equation}
We denote the moment map for the space $M(\delta,w)$ (resp., $M(\delta,0)$) by $\mu$ (resp., $\bar{\mu}$). According to \cref{ADEmomentmap0}, $(B,B^*,l_0,k_0)\in\mu^{-1}(0)$ if and only if $(B,B^*,0,0)\in\bar{\mu}^{-1}(0)$ and $l_0k_0=0$. Moreover, if $(B,B^*,l_0,k_0)\in\mu^{-1}(0)^{\chi-ss}$, then $l_0\neq 0,\ k_0=0$ (cf. \Cref{stTFAE}). To prove that \cref{chichi'} is an isomorphism, it suffices to show that $(B,B^*,l_0,0)\in\mu^{-1}(0)^{\chi-ss}$ if and only if $(B,B^*,0,0)\in\bar{\mu}^{-1}(0)^{\chi'-ss}$. We have that 
\begin{enumerate}[label=(\arabic*)]
    \item\label{ssequiv1} $(B,B^*,l_0,0)$ is $\chi$-semistable if and only if $\im(l_0)=V_0$ and it generates the entire collection of vector spaces $V=(V_i)_{1\leqslant i\leqslant n}$ under the action of $(B,B^*)$, according to \Cref{stTFAE}.
    \item\label{ssequiv2} $(B,B^*,0,0)$ is $\chi'$-semistable if and only if any $(B,B^*)$-stable subspace $S=(S_i)_{i\in I}$ satisfies 
    \[
    \chi'(\dim S):=\sum_{i=0}^n\chi'_i\dim S_i\leqslant 0,
    \] 
    according to \cite[Lemma 2.4]{Nakajima96}.
\end{enumerate}
Note that $\dim V_i=\delta_i$ and especially $\dim V_0=\delta_0=1$. For a subspace $S=(S_i)_{i\in I}\subset(V_i)_{i\in I}$, if $S_0=\{0\}$, then we automatically have that $\chi'(\dim S)\leqslant 0$. If $S_0=V_0$, then $\chi'(\dim S)\leqslant 0$ implies that $S_i=V_i$. Therefore, we can restate \ref{ssequiv2} into the following.
\begin{enumerate}[label=(\arabic*)$'$]
    \setcounter{enumi}{1}
    \item\label{ssequiv2'} $(B,B^*,0,0)$ is $\chi'$-semistable if and only if $V_0$ generates the entire $V$ under the action of $(B,B^*)$.
\end{enumerate}
From the semistability conditions \ref{ssequiv1} and \ref{ssequiv2'}, it is obvious that $(B,B^*,l_0,0)\in\mu^{-1}(0)^{\chi-ss}$ if and only if $(B,B^*,0,0)\in\bar{\mu}^{-1}(0)^{\chi'-ss}$. The isomorphism $\varphi_1$ in \cref{chichi'} then follows. Similar to showing that $\phi_1$ restricts to an isomorphism $\pi_{\chi'}^{-1}(0)\xrightarrow{\sim}\pi_{-\chi'}^{-1}(0)$ (cf. the commutative diagram in \cref{comdiag}), one can show that $\varphi_1$ also restricts to an isomorphism $\pi^{-1}(0)\xrightarrow{\sim}\pi_{\chi'}^{-1}(0)$. The rest follows from \Cref{Cilemma} by computing that $(\phi_1\circ\varphi_1)^{-1}(C'_i)=C_i$.
\end{proof}

\subsubsection{Question \ref{q3}}\label{liftsectionq3}
We can write the corresponding divisor of $\pi^{-1}(X^{\theta})$ as a linear combination of the irreducible components $\tilde{L}_j,\ 1\leq j\leq m,\ C_i,\ 1\leq i\leq n$. The coefficient of $\tilde{L}_j$ must be $1$ because we know that $X^{\theta}$ is reduced from \Cref{Anfixedloci,Dnevenfixedloci,Dnoddfixedloci,E6fixedloci,E7fixedlocus,E8fixedlocus}, and that $\pi$ is an isomorphism away from the exceptional fiber. Therefore, we have 
\[
\pi^{-1}(X^{\theta})=\sum_{j=1}^m\tilde{L}_j+\sum_{i=1}^n a_iC_i,\ a_i\geqslant 1
\]
as a divisor. Our goal is to determine the positive integers $a_i$. Recall the intersection pairing defined in \Cref{ZSpairing}.
\[
   \Div_{\pi^{-1}(0)}(\wt{X})\times\Div(\wt{X})\to\Z,\ (C,D)\mapsto C\cdot D
\]
Recall that $\mathcal{C}$ denotes the Cartan matrix of the simply-laced Dynkin diagram labeling $\Gamma$, and we know that the intersection matrix $(C_i\cdot C_j)_{1\leq i,j\leq n}$ equals $\minus\mathcal{C}$. Define 
\begin{equation}\label{bidef}
    b_i:=\#\{1\leqslant j\leqslant m|\tilde{L}_j\cap C_i\neq\emptyset\},\ \ 1\leq i\leq n.
\end{equation}
Note that if $\tilde{L}_j\cap C_i\neq\emptyset$, then the two curves meet transversely at a single point by \ref{it3} of \Cref{isolatedtransversal}. Applying property \ref{pair1} of the intersection pairing in \Cref{ZSpairing}, we can deduce that
\begin{equation}\label{biequiv}
    b_i=\sum_{j=1}^mC_i\cdot\tilde{L}_j.
\end{equation}
\begin{proposition}\label[proposition]{multofcomponent}
   If $X^{\theta}=\ddiv(f)\subset X$ is a principal divisor, then 
   \begin{equation}\label{mult}
       (a_1,\cdots,a_n)^t=\mathcal{C}^{-1}(b_1,\cdots,b_n)^t
   \end{equation}
\end{proposition}
\begin{proof}
    If $X^{\theta}=\ddiv(f)$, then $\pi^{-1}(X^{\theta})=\ddiv(\pi^*f)=\sum_{j=1}^m\tilde{L}_j+\sum_{i=1}^na_iC_i$ is also a principal divisor. Applying properties \ref{pair3} and \ref{pair4} in \Cref{ZSpairing}, we have that
    \begin{equation}\label{intertemp}
    \sum_{j=1}^mC_i\cdot\tilde{L}_j+\sum_{k=1}^na_k C_i\cdot C_k= C_i\cdot\ddiv(\pi^*f)=0,\ 1\leq i\leq n.
    \end{equation}
    Rearranging the terms and using \cref{biequiv} together with the fact that $(C_i\cdot C_j)_{1\leq i,j\leq n}=\minus\mathcal{C}$, we get
    \[
    \mathcal{C}(a_1,\cdots,a_n)^t=(\sum_{j=1}^mC_1\cdot\tilde{L}_j,\cdots,\sum_{j=1}^mC_n\cdot\tilde{L}_j)^t=(b_1,\cdots,b_n)^t.
    \]
    Then \cref{mult} follows from the fact that the Cartan matrix $\mathcal{C}$ is invertible.
\end{proof}

\begin{remark}\label[remark]{multofcomponentrem}
   \begin{enumerate}[label=(\arabic*)]
       \item\label{gred} Note that $a_i=1,\ 1\leq i\leq n$ does not imply that $\pi^{-1}(X^{\theta})$ is reduced, but rather generically reduced. However, for a principal divisor, being generically reduced is equivalent to being reduced.
       \item Assume $X^{\theta}$ is a principal divisor. If $C_i$ intersects more than two irreducible components in $\{\tilde{L}_j|1\leq j\leq m\}\cup\{C_k|1\leq k\leq n\}$, then $a_i>1$. This follows from \cref{intertemp}.
       \item Note that $\pi^{-1}(X^{\theta})$ is definitely not reduced in types $D,E$ because $\pi^{-1}(0)$ is not reduced in these types (cf. \Cref{exceptionalnonreduced}). However, $\pi^{-1}(X^{\theta})$ can be non-reduced in type $A$ (cf. \Cref{Anoddcase1preimage,Anevencase1preimage}), though $\pi^{-1}(0)$ is reduced in type $A$. 
       \item\label{generaldivisor} Let $\ddiv(f)\subset X$ be a reduced principal divisor. We write $\ddiv(\pi^*f)=D+\sum_{i=1}^na_iC_i$ such that $D$ does not involve any $C_i$. Assume further that each irreducible component of $D$ intersects at most one $C_i$, and the intersection is transversal. Define $b_i$ similarly as in \cref{bidef}. Then \cref{mult} still holds.
   \end{enumerate}
\end{remark}

From \Cref{Anfixedloci,Dnevenfixedloci,Dnoddfixedloci,E6fixedloci,E7fixedlocus,E8fixedlocus}, we see that $X^{\theta}$ is a principal divisor except for the following two cases.

\begin{enumerate}
    \item Type $A_n$, with $n$ odd, case $\RMN{3}$, where $X^{\theta}=\{0\}$. We have $\pi^{-1}(X^{\theta})=\pi^{-1}(0)=\sum_{i=1}^nC_i$ is reduced, according to \Cref{exceptionalnonreduced}.
    \item Type $A_n$, with $n$ even, case $\RMN{2}$, where $X^{\theta}=\Spec\C[X]/(x,z)$. The preimage $\pi^{-1}(X^{\theta})$ is generically reduced (cf. \ref{Anliftxztemp3} of \Cref{Anevencase2preimage}).
\end{enumerate}

\section{Preimages of fixed point loci for type $A_n$ Kleinian singularity}\label{Anquiver} 

\begin{mtable}[ht]
    \centering
\begin{tabular}{c}
    \scalebox{1.3}{\begin{tikzpicture}
        \filldraw[black] (3,3.5) circle (2pt) node[anchor=south]{$W_0$};
        \filldraw[black] (3,2) circle (2pt) node[anchor=north]{$V_0$};
        \node at (-1.2,1.2) (Ar) {$\widetilde{A}_n$};
        \filldraw[black] (-0.05,0) circle (2pt) node[anchor=north]{$V_1$};
        \filldraw[black] (1.5,0) circle (2pt) node[anchor=north]{$V_2$};
        \filldraw[black] (3,0) circle (2pt) node[anchor=north]{$V_3$};
        \node at (3.75,0) (cdots) {$\cdots$};
        \filldraw[black] (4.5,0) circle (2pt) node[anchor=north]{$V_{n\sminus1}$};
        \filldraw[black] (6.05,0) circle (2pt) node[anchor=north]{$V_n$};

        \draw[<-,thick] (2.95,3.45) -- (2.95,2.1);
        \node at (2.7,2.75) (q) {$k_0$};
        \draw[->,thick] (3.05,3.4) -- (3.05,2.05);
        \node[] at (3.3,2.75) (p) {$l_0$};
        \draw[->,thick] (2.9,1.95) -- (0,0.1);
        \node at (2.1,1.1) (a0) {\scalebox{0.8}{$B_{1\la 0}$}};
        \draw[<-,thick] (2.83,2) -- (-0.1,0.15);
        \node[] at (1.5,1.55) (b0) {\scalebox{0.8}{$B^*_{0\la 1}$}};
        \draw[<-,thick] (3.1,1.95) -- (6,0.1);
        \node at (4.1,1.1) (ar) {\scalebox{0.8}{$B_{0\la n}$}};
        \draw[->,thick] (3.17,2) -- (6.1,0.15);
        \node[] at (4.35,1.55) (br) {\scalebox{0.8}{$B^*_{n\la 0}$}};
        \draw[->,thick] (0.1,0.05) -- (1.4,0.05);
        \node at (0.7,0.2) (a1) {\scalebox{0.8}{$B_{2\la 1}$}};
        \draw[<-,thick] (0.1,-0.05) -- (1.4,-0.05);
        \node[] at (0.7,-0.25) (b1) {\scalebox{0.8}{$B^*_{1\la 2}$}};
        \draw[->,thick] (1.6,0.05) -- (2.9,0.05);
        \node at (2.2,0.2) (a2) {\scalebox{0.8}{$B_{3\la 2}$}};
        \draw[<-,thick,] (1.6,-0.05) -- (2.9,-0.05);
        \node[] at (2.2,-0.25) (b2) {\scalebox{0.8}{$B^*_{2\la 3}$}};
        \draw[->,thick] (4.6,0.05) -- (5.9,0.05);
        \node at (5.2,0.2) (an-1) {\scalebox{0.8}{$B_{n\la n\sminus1}$}};
        \draw[<-,thick,] (4.6,-0.05) -- (5.9,-0.05);
        \node[] at (5.3,-0.25) (bn-1) {\scalebox{0.8}{$B^*_{n\sminus1\la n}$}};
    \end{tikzpicture}}
\end{tabular}
\caption{Extended Dynkin quiver $\widetilde{A}_n$}
\label{quiverAn}
\end{mtable}

We realize the Kleinian singularity of type $A_n: X=\Spec\C[x,y,z]/(xy-z^{n+1})$ as a Nakajima quiver variety explicitly (\Cref{quiverAn}). Recall that $\delta=(1,\cdots,1)$. Using the notation in \Cref{quiverapp}, we have $\dim W_0=\dim V_i=1,\ 1\leq i\leq n$, and $B_{i+1\la i}\in\Hom(V_i,V_{i+1}),\ B^*_{i\la i+1}\in\Hom(V_i,V_{i+1})^*\simeq\Hom(V_{i+1},V_i)$. The representation space is
\[
M(\delta,w)=\bigoplus_{i=0}^n\Hom(V_i,V_{i+1})\oplus\bigoplus_{i=0}^n\Hom(V_{i+1},V_i)\oplus\Hom(W_0,V_0)\oplus\Hom(V_0,W_0).
\]
The group $\GL(\delta)=\prod_{i=0}^n\GL(V_i)$ acts on $M(\delta,w)$ naturally with the moment map 
\[
\mu=(\mu_i)_{0\leq i\leq n}\colon M(\delta,w)\to\gl(\delta),
\]
where 
\begin{equation}\label{Anmomentmap}
\begin{aligned}
    \mu_0&=B_{0\la n}B^*_{n\la 0}-B^*_{0\la 1}B_{1\la 0}+l_0k_0, \\
    \mu_i&=B_{i\la i-1}B^*_{i-1\la i}-B^*_{i\la i+1}B_{i+1\la i},\ 1\leq i\leq n.
\end{aligned}
\end{equation}
According to \Cref{wframinginv} and \Cref{ADExyz}, the algebra of invariant functions $\C[\mu^{-1}(0)]^{\GL(\delta)}$ is generated by the trace functions of cycles that start and end at the vertex $0$ of degrees $2$ or $n+1$, which are
\begin{equation}\label{Angenerator}
\begin{aligned}
    &x:=B_{0\la n}\cdots B_{2\la 1}B_{1\la 0},\ \ \ y:=B^*_{0\la 1}\cdots B^*_{n-1\la n}B^*_{n\la 0}, \\
    z:=&B^*_{0\la 1}B_{1\la 0}\stackrel{(\ref{Anmomentmap})}{=}B^*_{i\la i+1}B_{i+1\la i},\ 1\leq i\leq n\ (\text{assume}\ V_{n+1}=V_0).
\end{aligned}
\end{equation} 
It follows that $xy=z^{n+1}$. Thus we can construct a surjective homomorphism 
\begin{equation}\label{Ansurj}
   \C[x,y,z]/(xy-z^{n+1})\twoheadrightarrow\C[\mu^{-1}(0)]^{\GL(\delta)},
\end{equation}
which turns out to be an isomorphism because the left-hand side is a domain of dimension $2$ and the right-hand side is also of dimension $2$ from \Cref{wframinginv}. The explicit isomorphism in \cref{Ansurj} is known before; see, for example, \cite[eq.~(5.3)]{Nakajima96}. By \Cref{Ciequation}, we can write down the equation of each irreducible component of $\pi^{-1}(0)$.
\begin{equation}\label{AnCiequation}
\begin{aligned}
    C_i=\big\{[B,B^*,l_0,0]\in\M_{\chi}(\delta,w)\big| B_{i+1\la i}=B^*_{i-1\la i}=0 \big\},\ 1\leq i\leq n.
\end{aligned}
\end{equation}
Consider the following automorphism $\tau$ of the extended Dynkin diagram of type $\widetilde{A}_n$.
\begin{equation}\label{AnDynkinauto}
\tau(0)=0;\ \tau(i)=n+1-i,\ 1\leq i\leq n.
\end{equation}
It induces an automorphism of the extended Dynkin quiver $\widetilde{A}_n$ in \Cref{quiverAn} by sending an arrow $a:i\to j$ to the arrow $\tau(a):\tau(i)\to\tau(j)$. The automorphism $\tau$ will be used to construct lifts of the anti-Poisson involutions in \Cref{Anoddcase1sub,Anevencase1sub}. Note that $\tau$ coincides with the nontrivial automorphism of the McKay graph of type $\widetilde{A}_n$ induced by an element $g\in N_{\GL_2(\C)}(\Gamma)^-$ (cf. \Cref{AnAPIrem}). \Cref{Anoddcase1sub,Anoddcase2sub,Anoddcase3sub} study the preimages of the fixed point loci of anti-Poisson involutions for Kleinian singularities of type $A_n$, with $n$ odd. \Cref{Anevencase1sub,Anevencase2sub} deal with the cases with $n$ even. 

\subsection{Type $A_n,\ n$ odd, case \RNum{1}}\label{Anoddcase1sub} 
In this section, we consider the anti-Poisson involution
\[
  \theta(x)=y,\ \theta(y)=x,\  \theta(z)=z.
\]
The fixed point locus $X^{\theta}=\Spec\C[X]/(x-y)=\Spec\C[x,z]/(x^2-z^{n+1})$ has two irreducible components $L_j,\ j=1,2$. Following the notation in \Cref{liftsectionq1}, we have $\pi^{-1}(X^{\theta})_{\red}=\pi^{-1}(0)_{\red}\cup\tilde{L}_1\cup\tilde{L}_2$.

\begin{proposition}\label[proposition]{Anliftx-y}
The involution $\ttheta\colon\M_{\chi}(\delta,w)\to\M_{\chi}(\delta,w)$ defined by
\begin{equation}\label{Anliftx-yequation}
    \ttheta([B_{i+1\la i}, B^*_{j-1\la j},l_0,k_0])=([B^*_{\tau(i+1)\la\tau(i)}, B_{\tau(j-1)\la\tau(j)},l_0,-k_0])
\end{equation}
is anti-symplectic. Moreover, it is a lift of $\theta\colon X\to X$.
\end{proposition}
\begin{proof}
We consider the following linear anti-symplectic involution of the representation space.
\begin{equation*}
    \Theta\colon M(\delta,w)\to M(\delta,w),\ \ (B_{i+1\la i}, B^*_{j-1\la j},l_0,k_0)\mapsto (B^*_{\tau(i+1)\la\tau(i)}, B_{\tau(j-1)\la\tau(j)},l_0,-k_0).
\end{equation*}
It satisfies the conditions listed in \Cref{liftconditions}.
    \begin{enumerate}
    \item $(g_i)_{i\in I}^{\Theta}=(g_{\tau(i)})_{i\in I}$.
    \item $\Theta(x)=y,\ \Theta(y)=x,\ \Theta(z)=z$.
\end{enumerate}
The involution $\ttheta$ defined in \cref{Anliftx-yequation} is the descent of $\Theta$ to $\M_{\chi}(\delta,w)$. The rest follows from \Cref{liftconditions}.
\end{proof}
The action of $\ttheta$ on the exceptional fiber $\pi^{-1}(0)$ is described as follows.
\begin{proposition}\label[proposition]{Anoddcase1preimage}
\begin{enumerate}[label=(\arabic*)]
   \item The involution $\ttheta$ preserves only the component $C_{\frac{n+1}{2}}$, and swaps the component $C_i$ with $C_{n+1\minus i}$ for $i\neq\frac{n+1}{2}$. We have $\pi^{-1}(0)^{\ttheta}=\{p_1,p_2\}\subset C_{\frac{n+1}{2}}\setminus(C_{\frac{n-1}{2}}\cup C_{\frac{n+3}{2}})$.
   \item\label{Anoddcase1temp2} With a suitable choice of labeling, we have $\tilde{L}_{i}\cap\pi^{-1}(0)=\tilde{L}_{i}\cap C_{\frac{n+1}{2}}=\{p_i\},\ i=1,2$.
   \item The scheme $\pi^{-1}(X^{\theta})$ is not reduced (except in type $A_1$). As a divisor, we have $\pi^{-1}(X^{\theta})=\tilde{L}_{1}+\tilde{L}_{2}+\sum_{i=1}^na_iC_i$, where $a_i=\min\{i,n+1\minus i\}$.
\end{enumerate}
\end{proposition}
\begin{proof}
\begin{enumerate}[label=(\arabic*)]
    \item\label{temp1} It is easy to see that $\ttheta(C_i)=C_{\tau(i)}=C_{n+1\minus i}$ thanks to \cref{AnCiequation}. It follows that $\pi^{-1}(0)^{\ttheta}\subset C_{\frac{n+1}{2}}$. The intersection point $C_{\frac{n-1}{2}}\cap C_{\frac{n+1}{2}}$ is not $\ttheta$-stable because $\ttheta(C_{\frac{n-1}{2}})=C_{\frac{n+3}{2}}$. Similarly, the intersection point $C_{\frac{n+1}{2}}\cap C_{\frac{n+3}{2}}$ is not $\ttheta$-stable, either. 

    \item This follows from \Cref{isolatedtransversal}.
    
    \item Recall that we defined $b_i:=\#\{1\leqslant j\leqslant m|\tilde{L}_j\cap C_i\neq\emptyset\}$ in \cref{bidef}. We know that $b_{\frac{n+1}{2}}=2$ and $b_i=0,\ i\neq\frac{n+1}{2}$ from \ref{Anoddcase1temp2}. The divisor $X^{\theta}=\ddiv(x-y)$ is principal. Then applying \Cref{multofcomponent}, we can determine that $a_i=\min\{i,n+1\minus i\}$.
\end{enumerate}
\end{proof}

We depict the fixed point locus $X^{\theta}$ and the preimage $\pi^{-1}(X^{\theta})$ in \Cref{Anoddcase1pic}. The curved lines are $\pr^1$'s and the straight lines are $\A^1$'s. The number below each component indicates its multiplicity. Thickened lines indicate that the multiplicities are larger than $1$. The $\ttheta$-fixed components of $\pi^{-1}(X^{\theta})$ are colored in red.

\begin{mtable}[ht]
\begin{tabular}{ccc} 
    \begin{tikzpicture}
    \draw[line width=0.25mm] (-0.5,1) to (0.5,-1);
    \draw[line width=0.25mm] (0.5,1) to (-0.5,-1);
    \end{tikzpicture} & $\qquad$ &
    \begin{tikzpicture}
    \draw[line width=0.25mm] (-4,0) edge [bend left=35] node [below]{} (-1.5,0);
    \node at (-2.75,0) () {$1$};
    \draw[line width=0.4mm] (-2,0) edge [bend left=25] node [below]{} (-1,0.3);
    \node at (-1,0) () {$2$};
    \filldraw[black] (-0.65,0.3) circle (0.5pt);
    \filldraw[black] (-0.5,0.3) circle (0.5pt);
    \filldraw[black] (-0.35,0.3) circle (0.5pt);
    \draw[line width=0.4mm] (1,0) edge [bend right=25] node [below]{} (0,0.3);
    \node at (0,0) () {$\frac{n-1}{2}$};
    \draw[line width=0.4mm] (0.5,0) edge [bend left=35] node [below]{} (3,0);
    \node at (1.75,0) () {$\frac{n+1}{2}$};
    \draw[red, line width=0.25mm] (1.35,-0.8) to (1.35,1.2);
    \node at (1.35,-1) () {\textcolor{red}{$1$}};
    \draw[red, line width=0.25mm] (2.15,-0.8) to (2.15,1.2);
    \node at (2.15,-1) () {\textcolor{red}{$1$}};
    \draw[line width=0.4mm] (2.5,0) edge [bend left=25] node [below]{} (3.5,0.3);
    \node at (3.5,0) () {$\frac{n-1}{2}$};
    \filldraw[black] (3.8,0.3) circle (0.5pt);
    \filldraw[black] (4,0.3) circle (0.5pt);
    \filldraw[black] (4.2,0.3) circle (0.5pt);
    \draw[line width=0.4mm] (5.5,0) edge [bend right=25] node [below]{} (4.5,0.3);
    \node at (4.5,0) () {$2$};
    \draw[line width=0.25mm] (5,0) edge [bend left=35] node [below]{} (7.5,0);
    \node at (6.25,0) () {$1$};
    \end{tikzpicture} \\
     & & \\
    $X^{\theta}$ & $\qquad$ & $\pi^{-1}(X^{\theta})$
\end{tabular}
\caption{$A_n,\ n$ odd, Type \RMN{1}}
\label{Anoddcase1pic}
\end{mtable}

\subsection{Type $A_n,\ n$ odd, case \RNum{2}} \label{Anoddcase2sub}
In this section, we consider the anti-Poisson involution
\[
  \theta(x)=x,\ \theta(y)=y,\  \theta(z)=-z.
\]
The fixed point locus $X^{\theta}=\Spec\C[X]/(z)=\Spec\C[x,y]/(xy)$ has two irreducible components $L_j,\ j=1,2$. Following the notation in \Cref{liftsectionq1}, we have $\pi^{-1}(X^{\theta})_{\red}=\pi^{-1}(0)_{\red}\cup\tilde{L}_1\cup\tilde{L}_2$.

\begin{proposition}\label[proposition]{Anliftz}
The involution $\ttheta\colon \M_{\chi}(\delta,w)\to \M_{\chi}(\delta,w)$ defined by
\begin{equation}\label{Anliftzequation}
    \ttheta([B,B^*,l_0,k_0])=[-B,B^*,l_0,-k_0]
\end{equation}
is anti-symplectic. Moreover, it is a lift of $\theta\colon X\to X$.
\end{proposition}
\begin{proof} 
We consider the following linear anti-symplectic involution on the representation space.
\begin{equation*}
    \begin{aligned}
        \Theta\colon M(\delta,w)\to M(\delta,w),\ \ (B,B^*,l_0,k_0)\mapsto (-B,B^*,l_0,-k_0).
    \end{aligned}
\end{equation*}
It satisfies the conditions listed in \Cref{liftconditions}.
    \begin{enumerate}
    \item $g^{\Theta}=g$.
    \item $\Theta(x)=(-1)^{n+1}x=x,\ \Theta(y)=y,\ \Theta(z)=-z$.
\end{enumerate}
The involution $\ttheta$ defined in \cref{Anliftzequation} is the descent of $\Theta$ to $\M_{\chi}(\delta,w)$. The rest follows from \Cref{liftconditions}.
\end{proof}

\begin{proposition}\label[proposition]{Anoddcase2preimage}
   \begin{enumerate}[label=(\arabic*)]
       \item The involution $\ttheta$ preserves all the components $C_i,\ 1\leq i\leq n$. We have $\pi^{-1}(0)^{\ttheta}=\bigcup_{i\ \text{even}}C_i\cup\{p_1,p_2\}$, where $p_1\in C_1\setminus C_2$ and $p_2\in C_n\setminus C_{n-1}$.
       \item\label{Anoddcase2temp2} With a suitable choice of labeling, $\tilde{L}_1\cap\pi^{-1}(0)=\tilde{L}_1\cap C_1=\{p_1\},\ \tilde{L}_2\cap\pi^{-1}(0)=\tilde{L}_2\cap C_n=\{p_2\}$.
       \item The scheme $\pi^{-1}(X^{\theta})$ is reduced.
   \end{enumerate}
\end{proposition}
\begin{proof}
   \begin{enumerate}[label=(\arabic*)]
       \item First, the lift $\ttheta$ preserves all the components $C_i$ thanks to \cref{AnCiequation}. It follows from \Cref{liftoncomponentcor} that the $\ttheta$-fixed components and non-fixed components appear alternately. If $\ttheta$ fixes all the odd components, then there are no isolated points in $\pi^{-1}(0)^{\ttheta}$. However, \ref{it1} of \Cref{isolatedtransversal} tells us that there are two isolated points in $\pi^{-1}(0)^{\ttheta}$. Therefore, $\ttheta$ fixes each $C_i$ with $i$ even. Then it follows from \ref{autop3} of \Cref{liftoncomponent} that $\ttheta$ has exactly two fixed points on each $C_j$ with $j$ odd. The intersection points $C_j\cap C_{j-1}$ and $C_j\cap C_{j+1}$ exhaust both $\ttheta$-fixed points of $C_j$ when $j\neq 1,n,\ j$ odd. However, when $j=1$ (resp. $j=n$), the intersection point $C_1\cap C_2$ (resp. $C_n\cap C_{n-1}$) only accounts for one of the $\ttheta$-fixed points of $C_1$ (resp. $C_n$), resulting in an extra $\ttheta$-fixed point on $C_1$ (resp. $C_n$) that does not lie in $C_2$ (resp. $C_{n-1}$).
        \item This follows from \Cref{isolatedtransversal}.
        \item Recall the definition of  $b_i$ in \cref{bidef}. We have $(b_1,b_2,\cdots,b_{n-1},b_n)=(1,0,\cdots,0,1)$ by \ref{Anoddcase2temp2}. The divisor $X^{\theta}=\ddiv(z)$ is principal. Then applying \Cref{multofcomponent}, we can determine that
   \begin{equation*}
        (a_1,\cdots,a_n)=(1,\cdots,1).
   \end{equation*}
    Therefore, $\pi^{-1}(X^{\theta})$ is generically reduced. Note that $\pi^{-1}(X^{\theta})=\ddiv(\pi^*z)$ is a principal divisor, so being generically reduced is equivalent to being reduced (cf. \ref{gred} of \Cref{multofcomponentrem}).
    \end{enumerate}
\end{proof}

We depict the fixed point locus $X^{\theta}$ and the preimage $\pi^{-1}(X^{\theta})$ in \Cref{Anoddcase2pic}. All the components have multiplicity $1$. The $\ttheta$-fixed components of $\pi^{-1}(X^{\theta})$ are colored in red.

\begin{mtable}[ht]
\centering
\begin{tabular}{ccc}
    \begin{tikzpicture}
    \draw[thick] (-0.5,1) to (0.5,-1);
    \draw[thick] (0.5,1) to (-0.5,-1);
    \end{tikzpicture} & $\qquad$ &
    \begin{tikzpicture}
    \draw[red, thick] (-5,1.2) to (-4,-0.8);
    \draw[thick] (-5,0) edge [bend left=35] node [below]{} (-2.5,0);
    \draw[red, thick] (-3.5,0) edge [bend left=35] node [below]{} (-1,0);
    \draw[thick] (-2,0) edge [bend left=25] node [below]{} (-0.75,0.3);
    \node at (0,0.2) (cdots) {$\cdots$};
    \draw[thick] (2,0) edge [bend right=25] node [below]{} (0.75,0.3);
    \draw[red, thick] (3.5,0) edge [bend right=35] node [below]{} (1,0);
    \draw[thick] (5,0) edge [bend right=35] node [below]{} (2.5,0);
    \draw[red, thick] (5,1.2) to (4,-0.8);
    \end{tikzpicture} \\
    &  & \\
    $X^{\theta}$ & $\qquad$ & $\pi^{-1}(X^{\theta})$
\end{tabular}
\caption{$A_n,\ n$ odd, Type \RMN{2}}
\label{Anoddcase2pic}
\end{mtable}

\subsection{Type $A_n,\ n$ odd, case \RNum{3}}\label{Anoddcase3sub} In this section, we consider the anti-Poisson involution
\[
  \theta(x)=-x,\ \theta(y)=-y,\  \theta(z)=-z.
\]
The fixed point locus $X^{\theta}$ is the singular point $0$. The preimage $\pi^{-1}(X^{\theta})$ is the exceptional fiber $\pi^{-1}(0)$ whose irreducible components are $n$ copies of $\pr^1$'s, with dual graph the Dynkin diagram $A_n$. Moreover, $\pi^{-1}(0)$ is reduced (according to \Cref{exceptionalnonreduced}). This case can therefore be treated without constructing a lift. 
\begin{remark}\label[remark]{Anliftxyz}
However, a lift of $\theta$ still exists, and it is given by  \begin{equation}\label{Anliftxyzequation}
    \ttheta\colon \M_{\chi}(\delta,w)\to \M_{\chi}(\delta,w),\ \ [B,B^*,l_0,k_0]\mapsto [-\epsilon B,\epsilon^{-1}B^*,l_0,-k_0],
\end{equation}
where $\epsilon$ is a primitive $(2n+2)$-th root of unity. Moreover, $\ttheta$ fixes each $C_i$ with $i$ odd. The proofs are similar to those of \Cref{Anliftz,Anoddcase2preimage}. 
\end{remark}

\subsection{Type $A_n,\ n$ even, case \RNum{1}}\label{Anevencase1sub} In this section, we consider the anti-Poisson involution
\[
  \theta(x)=y,\ \theta(y)=x,\  \theta(z)=z.
\]
The fixed point locus $X^{\theta}=\Spec\C[X]/(x-y)=\Spec\C[x,z]/(x^2-z^{n+1})$ is a cusp (a single irreducible component $L$). Following the notation in \Cref{liftsectionq1}, we have $\pi^{-1}(X^{\theta})_{\red}=\pi^{-1}(0)_{\red}\cup\tilde{L}$. The argument in this section is very similar to that of \Cref{Anoddcase1sub}, with minor changes due to parity. We state the results, but omit the proofs. Recall the automorphism $\tau$ of the extended Dynkin diagram of type $\wt{A}_n$ defined in \cref{AnDynkinauto}.

\begin{proposition}\label[proposition]{Anevenliftx-y}
The involution $\ttheta\colon\M_{\chi}(\delta,w)\to\M_{\chi}(\delta,w)$ defined by 
\begin{equation}\label{Anevenliftx-yequation}
    \ttheta([B_{i+1\la i}, B^*_{j-1\la j},l_0,k_0])=[B^*_{\tau(i+1)\la\tau(i)}, B_{\tau(j-1)\la\tau(j)},l_0,-k_0]
\end{equation}
is anti-symplectic. Moreover, it is a lift of $\theta\colon X\to X$.
\end{proposition}

\begin{proposition}\label[proposition]{Anevencase1preimage}
\begin{enumerate}
   \item The involution $\ttheta$ swaps the component $C_i$ with $C_{n+1\minus i},\ \forall i$. We have $\pi^{-1}(0)^{\ttheta}=C_{\frac{n}{2}}\cap C_{\frac{n+2}{2}}=\{p\}$. 
   \item $\tilde{L}\cap \pi^{-1}(0)=\tilde{L}\cap C_{\frac{n}{2}}\cap C_{\frac{n+2}{2}}=\{p\}$.
   \item The scheme $\pi^{-1}(X^{\theta})$ is not reduced. As a divisor, we have $\pi^{-1}(X^{\theta})=\tilde{L}+\sum_{i=1}^na_iC_i$, where $a_i=\min\{i,n+1\minus i\}$.
\end{enumerate}
\end{proposition}

The fixed point locus $X^{\theta}$ and the preimage $\pi^{-1}(X^{\theta})$ are depicted in \Cref{Anevencase1pic}. The number below each component indicates its multiplicity. Thickened lines indicate that the multiplicities are larger than $1$. The unique $\ttheta$-fixed component of $\pi^{-1}(X^{\theta})$ is $\tilde{L}$, and it is colored in red.

\begin{mtable}[ht]
\begin{tabular}{ccc} 
    \begin{tikzpicture}
    \draw[line width=0.25mm] (0.5,-1.2) edge [bend right=20] node [below]{} (-0.3,0);
    \draw[line width=0.25mm] (0.5,1.2) edge [bend left=20] node [below]{} (-0.3,0);
    \end{tikzpicture} & $\ $ &
    \begin{tikzpicture}
    \draw[line width=0.25mm] (-4,0) edge [bend left=35] node [below]{} (-1.5,0);
    \node at (-2.75,0) () {$1$};
    \draw[line width=0.4mm] (-2,0) edge [bend left=25] node [below]{} (-1,0.3);
    \node at (-1,0) () {$2$};
    \filldraw[black] (-0.65,0.3) circle (0.5pt);
    \filldraw[black] (-0.5,0.3) circle (0.5pt);
    \filldraw[black] (-0.35,0.3) circle (0.5pt);
    \draw[line width=0.4mm] (1,0) edge [bend right=25] node [below]{} (0,0.3);
    \node at (0,0) () {$\frac{n-2}{2}$};
    \draw[line width=0.4mm] (0.5,0) edge [bend left=35] node [below]{} (3,0);
    \node at (1.75,0) () {$\frac{n}{2}$};
    \draw[line width=0.4mm] (2.5,0) edge [bend left=35] node [below]{} (5,0);
    \node at (3.75,0) () {$\frac{n}{2}$};
    \draw[red, line width=0.25mm] (2.75,-1.2) to (2.75,1.2);
    \node at (2.75,-1.4) () {\textcolor{red}{$1$}};
    \draw[line width=0.4mm] (4.5,0) edge [bend left=25] node [below]{} (5.5,0.3);
    \node at (5.5,0) () {$\frac{n-2}{2}$};
    \filldraw[black] (5.8,0.3) circle (0.5pt);
    \filldraw[black] (6,0.3) circle (0.5pt);
    \filldraw[black] (6.2,0.3) circle (0.5pt);
    \draw[line width=0.4mm] (7.5,0) edge [bend right=25] node [below]{} (6.5,0.3);
    \node at (6.5,0) () {$2$};
    \draw[line width=0.25mm] (7,0) edge [bend left=35] node [below]{} (9.5,0);
    \node at (8.25,0) () {$1$};
    \end{tikzpicture} \\
     & \\
    $X^{\theta}$ & & $\pi^{-1}(X^{\theta})$
\end{tabular}
\caption{$A_n,\ n$ even, Type \RMN{1}}
\label{Anevencase1pic}
\end{mtable}

\subsection{Type $A_n,\ n$ even, case \RNum{2}}\label{Anevencase2sub} In this section, we consider the anti-Poisson involution
\[
  \theta(x)=-x,\ \theta(y)=y,\  \theta(z)=-z.
\]
The fixed point locus $X^{\theta}=\Spec\C[X]/(x,z)\simeq \A^1$ has a single irreducible component $L$. Following the notation in \Cref{liftsectionq1}, we have $\pi^{-1}(X^{\theta})_{\red}=\pi^{-1}(0)_{\red}\cup\tilde{L}$. Similar to \Cref{Anliftz}, one can construct the following lift of $\theta$.

\begin{proposition}\label[proposition]{Anliftxz}
The involution $\ttheta\colon \M_{\chi}(\delta,w)\to \M_{\chi}(\delta,w)$ defined by
\begin{equation}\label{Anliftxzequation}
    \ttheta([B,B^*,l_0,k_0])=[-B,B^*,l_0,-k_0]
\end{equation}
is anti-symplectic. Moreover, it is a lift of $\theta\colon X\to X$.
\end{proposition}

\begin{proposition}\label[proposition]{Anevencase2preimage}
\begin{enumerate}[label=(\arabic*)]
       \item\label{Anliftxztemp1} $\pi^{-1}(0)^{\ttheta}=\bigcup_{i\ \text{even}}C_i\cup\{p\}$ where $p\in C_1\setminus C_2$.
       \item\label{Anliftxztemp2} We have $\tilde{L}\cap\pi^{-1}(0)=\tilde{L}\cap C_1=\{p\}$.
       \item\label{Anliftxztemp3} The scheme $\pi^{-1}(X^{\theta})$ is generically reduced.
   \end{enumerate}
\end{proposition}
\begin{proof}
The proofs of \ref{Anliftxztemp1} and \ref{Anliftxztemp2} are almost identical to those of \Cref{Anoddcase2preimage}, so we leave it to the reader. The proof of \ref{Anliftxztemp3} differs slightly from that of \Cref{Anoddcase2preimage} because $X^{\theta}$ is not a principal divisor. Set 
\begin{equation}\label{AnevenYdef}
Y:=\Spec\C[X]/(z)=\Spec\C[x,y,z]/(z,xy)\supset X^{\theta}.
\end{equation}
The closed embedding $\pi^{-1}(X^{\theta})\hookrightarrow\pi^{-1}(Y)$ implies that $\pi^{-1}(X^{\theta})\leq\pi^{-1}(Y)$ as divisors. We prove that 
$\pi^{-1}(Y)$ is reduced in \Cref{AnevenYreduced}. It follows that $\pi^{-1}(X^{\theta})$ is generically reduced as well.
\end{proof}

\begin{lemma}\label[lemma]{AnevenYreduced}
    The scheme $\pi^{-1}(Y)$ is reduced.
\end{lemma}
\begin{proof}
    Note that $\pi^{-1}(Y)=\ddiv(\pi^*z)$ is a principal divisor. We want to apply \Cref{multofcomponent} to show that $\pi^{-1}(Y)$ is reduced (cf. \ref{gred} and \ref{generaldivisor} of \Cref{multofcomponentrem}). To determine the $b_i$'s as defined in \cref{bidef}, we need to determine the irreducible components of $\pi^{-1}(Y)$ and study how different components intersect each other. The scheme $Y=\Spec\C[X]/(z)=\Spec\C[x,y]/(xy)$ has two irreducible components $L=\Spec\C[y]$ (coinciding with $X^{\theta}$) and $L':=\Spec\C[x]$. Similar to $\tilde{L}$, one can define $\tilde{L}':=\overline{\pi^{-1}(L'\setminus\{0\})}$. Then the irreducible components of $\pi^{-1}(Y)$ are $C_1,\cdots,C_n,\tilde{L},\tilde{L}'$.

    The recipe in \Cref{liftsection} allows us to understand how different components of $\pi^{-1}(X^{\eta})$ intersect each other when $\eta$ is an anti-Poisson. However, $Y$ is not the fixed point locus of an anti-Poisson involution when $n$ is even (the setup of \Cref{Anevencase2sub}). Nevertheless, we can write $Y$ as the union of the fixed point loci of two different anti-Poisson involutions $\theta$ and $\theta'$ of $X$, where
\begin{equation*}
\theta'\colon\C[X]\to\C[X],\ \ x\mapsto x,\ y\mapsto -y,\ z\mapsto -z.
\end{equation*}
We have $X^{\theta'}=L'$. It is easy to see that $\theta'$ is conjugate to $\theta$ by the classification in \Cref{AnAPIeven}. Recall that we have $\tilde{L}\cap\pi^{-1}(0)=\tilde{L}\cap C_1=\{p\}\subset C_1\setminus C_2$ from \ref{Anliftxztemp2} of \Cref{Anevencase2preimage}. Similarly, one can show that $\tilde{L}'\cap\pi^{-1}(0)=\tilde{L}'\cap C_n=\{p'\}\subset C_n\setminus C_{n-1}$. Therefore, $Y$ and $\pi^{-1}(Y)_{\red}$ look like the graphs in \Cref{Anoddcase2pic} (forgetting about the red coloring). 

Finally, we compute the multiplicities of the components of $\pi^{-1}(Y)$. We write $\pi^{-1}(Y)=\tilde{L}+\tilde{L}'+\sum_{i=1}^na_iC_i$ as a divisor of $\wt{X}$. Recall the definition of $b_i$ in \cref{bidef}. We have $(b_1,b_2,\cdots,b_{n-1},b_n)=(1,0,\cdots,0,1)$ from the previous paragraph. We can determine that $a_i=1$ by \cref{mult}. Therefore, $\pi^{-1}(Y)$ is generically reduced, hence it is also reduced as $\pi^{-1}(Y)=\ddiv(\pi^*z)$ is a principal divisor; cf. \ref{gred} of \Cref{multofcomponentrem}.
\end{proof}

We depict the fixed point locus $X^{\theta}$ and the preimage $\pi^{-1}(X^{\theta})$ in \Cref{Anevencase2pic}. All the components have multiplicity $1$. The $\ttheta$-fixed components of $\pi^{-1}(X^{\theta})$ are colored in red.

\begin{mtable}[ht]
\centering
\begin{tabular}{ccc}
    \begin{tikzpicture}
    \draw[thick] (0,1) to (0,-1);
    \end{tikzpicture} & $\qquad$ &
    \begin{tikzpicture}
    \draw[red, thick] (-5,1.2) to (-4,-0.8);
    \draw[thick] (-5,0) edge [bend left=35] node 
    [below]{} (-2.5,0);
    \draw[red, thick] (-3.5,0) edge [bend left=35] node [below]{} (-1,0);
    \draw[thick] (-2,0) edge [bend left=25] node [below]{} (-0.75,0.3);
    \node at (0,0.2) (cdots) {$\cdots$};
    \draw[thick] (2,0) edge [bend right=25] node [below]{} (0.75,0.3);
    \draw[red, thick] (3.5,0) edge [bend right=35] node [below]{} (1,0);
    \end{tikzpicture} \\
     &  & \\
    $X^{\theta}$ & $\qquad$ & $\pi^{-1}(X^{\theta})$
\end{tabular}
\caption{$A_n,\ n$ even, Type \RMN{2}}
\label{Anevencase2pic}
\end{mtable}

\begin{remark}\label[remark]{Ancomputation}
    There is also a computational way to study the scheme-theoretic preimage $\pi^{-1}(X^{\theta})$. We explain the method below, but omit some details. 
    
    Recall that we have $X\simeq\mu^{-1}(0)\gitquo\GL(\delta)\simeq\Spec\C[x,y,z]/(xy-z^{n+1})$, where the generators $x,y,z$ are defined in \cref{Angenerator}. Moreover, $\wt{X}=\mu^{-1}(0)\gitquo^{\chi}\GL(\delta)$, and the natural projective morphism $\pi\colon\wt{X}\to X$ gives the minimal resolution (cf. \Cref{sympres}). Set $B_{j\leftarrow 0}:=B_{j\leftarrow j-1}\cdots B_{1 \leftarrow 0},\ B^*_{j\leftarrow 0}:=B^*_{j\leftarrow j+1}\cdots B^*_{n\leftarrow 0}$. We identify $B_{0\la 0}=B^*_{n+1\la 0}=\id_{V_0}$ and $B_{n+1\la 0}=x,\ B^*_{0\la 0}=y$. Consider $f_i:=(\prod_{j=0}^i B_{j\la 0}l_0)(\prod_{j=i+1}^n B^*_{j\la 0}l_0)\in\C[\mu^{-1}(0)]^{\GL(\delta),\chi}$. The principal affine open subset $U_i:=\mu^{-1}(0)_{f_i}\subset\mu^{-1}(0)^{\chi-ss}$ is $\GL(\delta)$-invariant, and one can show that $\{U_i\}_{i\in I}$ covers $\mu^{-1}(0)^{\chi-ss}$ by \Cref{stTFAE}. It follows that $\{U_i\gitquo\GL(\delta)\}_{i\in I}$ forms an affine open cover of $\wt{X}$. Moreover, for each $i\in I$, set $u_i:=\frac{B^*_{i\la 0}}{B_{i\la 0}},\ v_i:=\frac{B_{i+1\la 0}}{B^*_{i+1\la 0}}$. One can prove that 
\begin{equation}\label{UiC2}
U_i\gitquo\GL(\delta)=\Spec\C[U_i]^{\GL(\delta)}=\Spec\C[u_i,v_i]\simeq\C^2,
\end{equation}
using a version of Zariski's Main Theorem \cite[Corollary 18.12.13]{EGAIV}. Moreover, two affine open subsets $U_i\gitquo\GL(\delta),\ U_j\gitquo\GL(\delta)$ with $i<j$ are glued by the relation $v_iu_jz^{j-i-1}=1$. The composition $\pi_i\colon U_i\gitquo\GL(\delta)\xhookrightarrow{}\wt{X}\xrightarrow[]{\pi}X$, in terms of coordinates, is given by
\begin{equation}\label{AnUimorphism}
    (u_i,v_i)\mapsto(x=u_i^{n-i}v_i^{n-i+1},\ y=u_i^{i+1}v_i^i,\ z=u_iv_i).
\end{equation}
We can determine the preimage $\pi^{-1}(Y)$ of any subscheme $Y\subset X$ by computing the preimages $\pi_i^{-1}(Y)=\pi^{-1}(Y)\cap(U_i\gitquo\GL(\delta))$ using \cref{AnUimorphism} and then gluing them together.

The idea of the computational method can be carried beyond type $A$ as well, but the computations are not manageable. 
\end{remark}

\section{Preimages of fixed point loci for type $D_n$ Kleinian singularity}\label{Dnquiver} 

\begin{mtable}[ht]
    \centering
\begin{tabular}{c}
    \scalebox{1.3}{\begin{tikzpicture}
        \node at (-3,1) (Dn) {$\widetilde{D}_n$};
        \filldraw[black] (-1.4,2.4) circle (2pt) node[anchor=east]{$W_0$};
        \filldraw[black] (-1.4,0.9) circle (2pt) node[anchor=east]{$V_0$};
        \filldraw[black] (-1.4,-0.9) circle (2pt) node[anchor=east]{$V_1$};
        
        \filldraw[black] (0,0) circle (2pt) node[anchor=north]{$V_2$};
        \filldraw[black] (1.5,0) circle (2pt) node[anchor=north]{$V_3$};
        \filldraw[black] (3,0) circle (2pt) node[anchor=north]{$V_4$};
        \node at (3.75,0) (cdots) {$\cdots$};
        \filldraw[black] (4.5,0) circle (2pt) node[anchor=north]{$V_{n\sminus3}$};
        \filldraw[black] (6.35,0) circle (2pt) node[anchor=north]{$V_{n\sminus2}$};
        
        \filldraw[black] (7.65,0.9) circle (2pt) node[anchor=west]{$V_{n\sminus1}$};
        \filldraw[black] (7.65,-0.9) circle (2pt) node[anchor=west]{$V_n$};

        \draw[<-,thick] (-1.45,2.35) -- (-1.45,1.05);
        \node at (-1.7,1.675) (q) {$k_0$};
        \draw[->,thick] (-1.35,2.25) -- (-1.35,0.95);
        \node[] at (-1.1,1.625) (p) {$l_0$};
        
        \draw[->,thick] (-1.4,0.9) -- (-0.1,0.1);
        \node at (-0.4,0.62) (a0) {$b_{2\la 0}$};
        \draw[<-,thick] (-1.4,0.8) -- (-0.15,0.02);
        \node[] at (-0.95,0.3) (b0) {$b^*_{0\la 2}$};
        \draw[<-,thick] (-1.3,-0.92) -- (-0.1,-0.12);
        \node at (-1.1,-0.3) (a1) {$b_{2\la 1}$};
        \draw[->,thick] (-1.35,-0.85) -- (-0.15,-0.05);
        \node[] at (-0.35,-0.72) (b1) {$b^*_{1\la 2}$};

        \draw[<-,thick] (7.6,0.85) -- (6.35,0.1);
        \node at (6.4,0.72) (a0) {$b_{n\sminus1\la n\sminus 2}$};
        \draw[->,thick] (7.6,0.75) -- (6.35,0);
        \node[] at (7.8,0.2) (b0) {$b^*_{n\sminus2\la n\sminus1}$};
        \draw[->,thick] (7.55,-0.92) -- (6.35,-0.12);
        \node at (7.65,-0.4) (ar) {$b_{n\la n\sminus2}$};
        \draw[<-,thick] (7.6,-0.85) -- (6.4,-0.05);
        \node[] at (6.65,-0.72) (br) {$b^*_{n\sminus2\la n}$};
        
        \draw[->,thick] (0.1,0.05) -- (1.4,0.05);
        \node at (0.7,0.25) (a2) {\scalebox{0.9}{$B_{3\leftarrow2}$}};
        \draw[<-,thick] (0.1,-0.05) -- (1.4,-0.05);
        \node[] at (0.7,-0.25) (b2) {\scalebox{0.9}{$B^*_{2\leftarrow 3}$}};
        \draw[->,thick] (1.6,0.05) -- (2.9,0.05);
        \node at (2.2,0.25) (a3) {\scalebox{0.9}{$B_{4\leftarrow3}$}};
        \draw[<-,thick,] (1.6,-0.05) -- (2.9,-0.05);
        \node[] at (2.2,-0.25) (b3) {\scalebox{0.9}{$B^*_{3\leftarrow 4}$}};
        \draw[->,thick] (4.6,0.05) -- (6.2,0.05);
        \node at (5.4,0.25) (ar-3) {{\scalebox{0.7}{$B_{n\sminus2\leftarrow n\sminus3}$}}};
        \draw[<-,thick,] (4.6,-0.05) -- (6.2,-0.05);
        \node[] at (5.4,-0.25) (br-3) {{\scalebox{0.7}{$B^*_{n\sminus3\leftarrow n\sminus2}$}}};
    \end{tikzpicture}}
\end{tabular}
\caption{Extended Dynkin quiver $\widetilde{D}_n$}
\label{quiverDn}
\end{mtable}

We realize the Kleinian singularity of type $D_n$ as a Nakajima quiver variety explicitly (\Cref{quiverDn}). Recall that $\delta=(1,1,2,\cdots,2,1,1)$. 
Using the notation in \Cref{quiverapp}, we have that $W_0, V_0, V_1, V_{n-1}, V_n$ are one-dimensional vector spaces, and $V_i,\ 2\leq i\leq n-2$ are two-dimensional vector spaces. We use $b_{j\la i},\ b^*_{i\la j}$ to denote the linear maps between a one-dimensional vector space and a two-dimensional vector space, and $B_{i+1\la i},\ B^*_{i\la i+1}$ to denote the linear maps between two-dimensional vector spaces. The representation space is
\[
M(\delta,w)=T^*\big(\Hom(V_0,V_2)\oplus\Hom(V_1,V_2)\bigoplus_{i=2}^{n\sminus3}\Hom(V_i,V_{i+1})\oplus\Hom(V_{n\sminus2},V_{n\sminus1})\oplus\Hom(V_{n\sminus2},V_n)\oplus\Hom(W_0,V_0)\big).
\]
The group $\GL(\delta)=\prod_{i=0}^n\GL(V_i)$ acts on $M(\delta,w)$ naturally with the moment map 
\[
\mu=(\mu_i)_{0\leq i\leq n}\colon M(\delta,w)\to\gl(\delta),
\]
where, for $n\geqslant 5$, we have
\begin{equation}\label{Dnmomentmap}
\begin{aligned}
    \mu_0&=-b^*_{0\la 2}b_{2\la 0}+l_0k_0,\\
    \mu_1&=-b^*_{1\la 2}b_{2\la 1}, \\
    \mu_2&=b_{2\la 0}b^*_{0\la 2}+b_{2\la 1}b^*_{1\la 2}-B^*_{2\leftarrow 3}B_{3\leftarrow 2}, \\
    \mu_i&=B_{i\leftarrow i-1}B^*_{i-1\leftarrow i}-B^*_{i\leftarrow i+1}B_{i+1\leftarrow i},\ 3\leq i\leq n-3,\\
    \mu_{n-2}&=B_{n-2\leftarrow n-3}B^*_{n-3\leftarrow n-2}-b^*_{n-2\la n-1}b_{n-1\la n-2}-b^*_{n-2\la n}b_{n\la n-2}, \\
    \mu_{n-1}&=b_{n-1\la n-2}b^*_{n-2\la n-1}, \\
    \mu_n&=b_{n\la n-2}b^*_{n-2\la n}. \\
\end{aligned}
\end{equation}
When $n=4$, we have
\[
 \mu_2=\mu_{n-2}=b_{2\la 0}b^*_{0\la 2}+b_{2\la 1}b^*_{1\la 2}-b^*_{2\la 3}b_{3\la 2}-b^*_{2\la 4}b_{4\la 2},
\]
and $\mu_0,\ \mu_1,\ \mu_3=\mu_{n-1},\ \mu_4=\mu_n$ are determined by the same equations as in \cref{Dnmomentmap}. \Cref{quiverD4} illustrates the extended Dynkin quiver $\widetilde{D}_4$ more clearly. By \Cref{Ciequation}, we have that
\[
    C_i=\big\{[B,B^*,l_0,0]\in\M_{\chi}(\delta,w)\big| \rank(B_{a})_{\substack{a\in\bar{\Omega}\\ t(a)=i}}<\dim V_i \big\},\ 1\leq i\leq n.
\]
For brevity, we introduce the following notations.
\begin{equation*}
\begin{aligned}
    B_{j\leftarrow i}:&=B_{j\leftarrow j-1}\cdots B_{i+1 \leftarrow i}\colon V_i\to V_j,\ 2\leq i\leq j\leq n-2, \\ 
    B^*_{i\leftarrow j}:&=B^*_{i\leftarrow i+1}\cdots B^*_{j-1\leftarrow j}\colon V_j\to V_i,\ 2\leq i\leq j\leq n-2.
\end{aligned}
\end{equation*}
According to \Cref{wframinginv} and \Cref{ADExyz}, the algebra of invariant functions $\C[\mu^{-1}(0)]^{\GL(\delta)}$ is generated by trace functions of cycles that start and end at the vertex $0$ of degrees $4,\ 2n-4$ and $2n-2$. The explicit generators are given in \Cref{Dngeneratorprop}. 

\begin{proposition}\label[proposition]{Dngeneratorprop}
The trace functions
\begin{equation}\label{Dnxyz}
    \begin{aligned}
    x:=&b^*_{0\la 2}b_{2\la 1}b^*_{1\la 2}b_{2\la 0},\ \ \ y:=b^*_{0\la 2}B^*_{2\la n-2}b^*_{n-2\la n}b_{n\la n-2}B_{n-2\la 2}b_{2\la 0},\\ &z:=b^*_{0\la 2}B^*_{2\la n-2}b^*_{n-2\la n}b_{n\la n-2}b^*_{n-2\la n-1}b_{n-1\la n-2}B_{n-2\la 2}b_{2\la 0}.
\end{aligned}
\end{equation}
generate $\C[\mu^{-1}(0)]^{\GL(\delta)}$. Moreover, we have
\begin{equation}\label{Dnxyzrelation}
    \begin{aligned}
        \C[\mu^{-1}(0)]^{\GL(\delta)}&\simeq\begin{cases}
            \C[x,y,z]/(xy(y-x^{\frac{n-2}{2}})-z^2),\ n\ \text{even}, \\
            \C[x,y,z]/(xy^2-z(z-x^{\frac{n-1}{2}})),\ n\ \text{odd}.
        \end{cases}
    \end{aligned}
\end{equation}
\end{proposition}

\Cref{Dnind,Dn0234} will be needed for the proof of \Cref{Dngeneratorprop}.

\begin{lemma}\label[lemma]{Dnind}
Consider the extended Dynkin quiver $\wt{D}_n,\ n\geqslant 5$. The following relations hold in $\C[\mu^{-1}(0)]^{\GL(\delta)}$.
\begin{equation}\label{Dnindeqn}
     \begin{aligned}
         &b^*_{1\la 2}(B^*_{2\la 3}B_{3\la2})^rb_{2\la1}=b_{n-1\la n-2}(B_{n-2\la n-3}B^*_{n-3\la n-2})^rb^*_{n-2\la n-1}\\
         =&b^*_{0\la 2}(B^*_{2\la 3}B_{3\la2})^rb_{2\la0}=b_{n\la n-2}(B_{n-2\la n-3}B^*_{n-3\la n-2})^rb^*_{n-2\la n}=\begin{cases}
           x^{\frac{r+1}{2}},\ r\ \text{odd}, \\
           0,\ r\ \text{even}.
    \end{cases}
    \end{aligned}
\end{equation}
\end{lemma}
\begin{proof}
We only prove the equalities in the second line of \cref{Dnindeqn}. The remaining equalities can be shown similarly. We begin with $r=1$, i.e., we show that
\begin{equation}\label{Dnxcopy1}
    b^*_{0\la 2}B^*_{2\la 3}B_{3\la2}b_{2\la0}=b_{n\la n-2}B_{n-2\la n-3}B^*_{n-3\la n-2}b^*_{n-2\la n}=x.
\end{equation}
Note that by \cref{ADEmomentmap0}, we have $l_0k_0=0$. Then together with $\mu_0=0$, where $\mu_0$ is given in \cref{Dnmomentmap}, we obtain that
\begin{equation}\label{Dnmu0}
    -b^*_{0\la 2}b_{2\la 0}=0.
\end{equation}
Consider the first term in \cref{Dnxcopy1}, we have 
\[
b^*_{0\la 2}B^*_{2\la 3}B_{3\la2}b_{2\la0}\xlongequal{\mu_2}b^*_{0\la 2}(b_{2\la1}b^*_{1\la2}+b_{2\la0}b^*_{0\la2})b_{2\la0}\xlongequal{\cref{Dnmu0}}b^*_{0\la 2}b_{2\la 1}b^*_{1\la 2}b_{2\la 0}\xlongequal[]{\cref{Dnxyz}}x.
\]
For the second term in \cref{Dnxcopy1}, using $\mu_{n-2}=\mu_n=0$, we obtain that
\[
b_{n\la n-2}B_{n-2\la n-3}B^*_{n-3\la n-2}b^*_{n-2 \la n}=b_{n\la n-2}b^*_{n-2\la n-1}b_{n-1\la n-2}b^*_{n-2 \la n}.
\]
Therefore, to prove \cref{Dnxcopy1}, it remains to prove \cref{Dnxcopy} below.
\begin{equation}\label{Dnxcopy}
    b_{n\la n-2}b^*_{n-2\la n-1}b_{n-1\la n-2}b^*_{n-2 \la n}=b^*_{0\la 2}b_{2\la 1}b^*_{1\la 2}b_{2\la 0}=x,
\end{equation}
which follows from the following computation.
\begin{equation}
\begin{aligned}
    2b_{n\la n-2}b^*_{n-2\la n-1}&b_{n-1\la n-2}b^*_{n-2 \la n}\xlongequal{\mu_{n\sminus1},\ \mu_n}\tr((b^*_{n-2\la n-1}b_{n-1\la n-2}+b^*_{n-2\la n}b_{n\la n-2})^2)\\\xlongequal{\mu_{n\sminus2}}&\tr((B_{n-2\la n-3}B^*_{n-3\la n-2})^2)\xlongequal{\mu_i}\tr((B^*_{2\la 3}B_{3\la 2})^2)\\\xlongequal{\mu_2}\tr((&b_{2\la 0}b^*_{0\la 2}+b_{2\la 1}b^*_{1\la 2})^2)
    =2b^*_{0\la 2}b_{2\la 1}b^*_{1\la 2}b_{2\la 0}=2x.
\end{aligned}
\end{equation}
Next, we prove the equalities in the second line of \cref{Dnindeqn} for any $r$.
\begin{equation}\label{Dnxr}
\begin{aligned}
    b^*_{0\la 2}(B^*_{2\la 3}B_{3\la2})^rb_{2\la 0}\xlongequal{\mu_2}
    &\begin{cases}
    \underline{b^*_{0\la 2}b_{2\la 1}b^*_{1\la 2}b_{2\la 0}}b^*_{0\la 2} \cdots b_{2\la 0}b^*_{0\la 2}b_{2\la 1}b^*_{1\la 2}b_{2\la 0},\ r\ \textrm{odd},   \\
    \underline{b^*_{0\la 2}b_{2\la 0}}b^*_{0\la 2}b_{2\la 1}b^*_{1\la 2} \cdots b_{2\la 0}b^*_{0\la 2}b_{2\la 1}b^*_{1\la 2}b_{2\la 0},\ r\ \textrm{even}.
    \end{cases}  \\
    =&\begin{cases}(b^*_{0\la 2}b_{2\la 1}b^*_{1\la 2}b_{2\la 0})^{\frac{r+1}{2}}=x^{\frac{r+1}{2}},\ r\ \text{odd}, \\
        0,\ r\ \text{even}.\end{cases}
\end{aligned}
\end{equation}
The computation to show that 
\begin{equation}
b_{n\la n-2}(B_{n-2\la n-3}B^*_{n-3\la n-2})^rb^*_{n-2\la n}=\begin{cases}x^{\frac{r+1}{2}},\ r\ \text{odd}, \\
        0,\ r\ \text{even},\end{cases}
\end{equation}
is similar to that of \cref{Dnxr} by using the $r=1$ case:
\[
b_{n\la n-2}B_{n-2\la n-3}B^*_{n-3\la n-2}b^*_{n-2\la n}=b_{n\la n-2}b^*_{n-2\la n-1}b_{n-1\la n-2}b^*_{n-2\la n}=x,
\]
and we leave it to the reader.
\end{proof}

\begin{lemma}\label[lemma]{Dn0234}
Let $q$ be a cycle on the extended Dynkin quiver $\wt{D}_n,\ n\geqslant5$ that starts and ends at the vertex $0$. Assume that $q$ does not contain the vertices $n-1,n$. Then the following holds in $\C[\mu^{-1}(0)]^{\GL(\delta)}$.
\begin{equation}\label{Dn0234eq}
     \tr(B_q)\begin{cases} \in\C x^{\frac{\deg q}{4}},\ 4|\deg q,\\
     =0,\ 4\nmid\deg q.
\end{cases} 
\end{equation}
\end{lemma}
\begin{proof}
We may assume that the vertex $0$ appears only once in $q$. Otherwise, we can write $q=q_1q_2$ where $q_1,q_2$ are cycles that both start and end at $0$ but with lower degrees. We have $\tr(B_{q})=\tr(B_{q_1})\tr(B_{q_2})$. It suffices to prove the corresponding statements for $\tr(B_{q_1})$ and $\tr(B_{q_2})$. 


Note that $q$ does not contain the vertices $n-1,n$ by assumption. We may further assume that $q$ does not contain the vertex $1$. Otherwise, we can replace each occurrence of $b_{2\la 1}b^*_{1\la 2}$ in $\tr(B_q)$ with $B^*_{2\la 3}B_{3\la 2}-b_{2\la 0}b^*_{0\la 2}$ using the moment map $\mu_2=0$. 

It remains to prove \cref{Dn0234eq} when $q$ is a cycle that starts and ends at the vertex $0$ on the subquiver $0\rightleftarrows2\rightleftarrows3\rightleftarrows\cdots\rightleftarrows n-2$ of type $A$. Using the moment maps $\mu_i=0,\ 3\leq i\leq n-3$ in \cref{Dnmomentmap} iteratively, we can deduce that $\tr(B_q)=b^*_{0\la 2}(B^*_{2\la 3}B_{3\la 2})^{\frac{\deg q-2}{2}}b_{2\la 0}$. Then \cref{Dn0234eq} follows from \Cref{Dnind}.
\end{proof}

\begin{proof}[Proof of \Cref{Dngeneratorprop}]
We only prove the statement for $n\geqslant 5$. The case $n=4$ is similar, but easier. 

We first show that $x,y,z$ generate $\C[\mu^{-1}(0)]^{\GL(\delta)}$. It suffices to show that the degree-$d$ component $\C[\mu^{-1}(0)]^{\GL(\delta)}_d$ with $d=4,2n-4,2n-2$, is contained in the subalgebra generated by $x,y,z$.
    
The only cycles of degree $4$ that start and end at the vertex $0$ are $0\to 2\to 0\to 2 \to 0,\ 0\to 2 \to 1 \to 2 \to 0,\ 0\to 2\to 3\to 2\to 0$. The corresponding trace functions are 
\begin{equation*}
\begin{aligned}
    b^*_{0\la 2}b_{2\la 0}b^*_{0\la 2}b_{2\la 0}&\stackrel{\mu_0}{=}0, \\
    b^*_{0\la 2}b_{2\la 1}b^*_{1\la 2}b_{2\la 0}&\stackrel{\text{def}}{=}x, \\
    b^*_{0\la 2}B^*_{2\la 3}B_{3\la 2}b_{2\la 0}&\stackrel{\mu_2}{=}b^*_{0\la 2}(b_{2\la 0}b^*_{0\la 2}+b_{2\la 1}b^*_{1\la 2})b_{2\la 0}=0+x=x.
\end{aligned}
\end{equation*}
Thus, $\C[\mu^{-1}(0)]^{\GL(\delta)}_4=\C x$.

A cycle of degree $2n-4$ that starts and ends at $0$ can contain at most one of $n-1,n$. If it contains neither $n-1$ nor $n$, \Cref{Dn0234} tells us that the corresponding trace function is a scalar multiple of $x^{\frac{n-2}{2}}$ (resp. $0$) when $n$ is even (resp. odd). If the cycle contains exactly one of $n-1, n$, then it must be of the form $0\to 2\to \cdots \to n-2\to n \to n-2 \to \cdots \to 2 \to 0$ or $0\to 2\to \cdots \to n-2\to n-1 \to n-2 \to \cdots \to 2 \to 0$, and the corresponding trace functions are
\begin{equation}\label{Dnyy'}
\begin{aligned}
    y\stackrel{\text{def}}{=}&b^*_{0\la 2}B^*_{2\la n-2}b^*_{n-2\la n}b_{n\la n-2}B_{n-2\la 2}b_{2\la 0},\\
    y':=&b^*_{0\la 2}B^*_{2\la n-2}b^*_{n-2\la n-1}b_{n-1\la n-2}B_{n-2\la 2}b_{2\la 0}. 
\end{aligned}
\end{equation}
There is a relation between $y$ and $y'$.
\begin{equation}\label{Dny+y'}
\begin{aligned}
    y+&y'=b^*_{0\la 2}B^*_{2\la n-2}(b^*_{n-2\la n}b_{n\la n-2}+b^*_{n-2\la n-1}b_{n-1\la n-2})B_{n-2\la 2}b_{2\la 0} \\
    \stackrel{\mu_{n-2}}{=}b^*_{0\la 2}B^*_{2\la n-2}&B_{n-2\la n-3}B^*_{n-3\la n-2}B_{n-2\la 2}b_{2\la 0}{=}b^*_{0\la 2}(B^*_{2\la 3}B_{3\la 2})^{n-3}b_{2\la 0}\xlongequal{\Cref{Dnind}}\begin{cases}x^{\frac{n-2}{2}}, n\ \text{even}, \\
    0,\ n\ \text{odd}.\end{cases}
\end{aligned}
\end{equation}
It follows that $\C[\mu^{-1}(0)]^{\GL(\delta)}_{2n-4}=\C y+\C x^{\frac{n-2}{2}}$ when $n$ is even, and $\C[\mu^{-1}(0)]^{\GL(\delta)}_{2n-4}=\C y$ when $n$ is odd.

Next, we consider the following two cycles of degree $2n-2$ that start and end at $0$.
\begin{equation}\label{Dnzz'}
\begin{aligned}
    z=&b^*_{0\la 2}B^*_{2\la n-2}b^*_{n-2\la n}b_{n\la n-2}b^*_{n-2\la n-1}b_{n-1\la n-2}B_{n-2\la 2}b_{2\la 0}, \\
    z':=&b^*_{0\la 2}B^*_{2\la n-2}b^*_{n-2\la n-1}b_{n-1\la n-2}b^*_{n-2\la n}b_{n\la n-2}B_{n-2\la 2}b_{2\la 0}.
    \end{aligned}
\end{equation}
One can compute that
\begin{equation}\label{Dnz+z'}
    z+z'=\begin{cases} 0,\ n\ \text{even}, \\ x^{\frac{n-1}{2}},\ n\ \text{odd}.\end{cases}
\end{equation}
Similarly to the case of degree $2n-4$, one can show that $\C[\mu^{-1}(0)]^{\GL(\delta)}_{2n-2}=\C z$ when $n$ is even, and $\C[\mu^{-1}(0)]^{\GL(\delta)}_{2n-2}=\C z+\C x^{\frac{n-1}{2}}$ when $n$ is odd. Therefore, $\C[\mu^{-1}(0)]^{\GL(\delta)}$ is generated by the functions $x,y,z$ defined in \cref{Dnxyz} for both $n$ even and odd. Lastly, we compute that
\begin{equation}\label{Dnrelation}
\begin{aligned}
    zz'=&b^*_{0\la 2}B^*_{2\la n-2}b^*_{n-2\la n}b_{n\la n-2}b^*_{n-2\la n-1}\underline{b_{n-1\la n-2}B_{n-2\la 2}b_{2\la 0}}\\
    &\underline{\cdot b^*_{0\la 2}B^*_{2\la n-2}b^*_{n-2\la n-1}}b_{n-1\la n-2}b^*_{n-2\la n}b_{n\la n-2}B_{n-2\la 2}b_{2\la 0} \\
    =&y'b^*_{0\la 2}B^*_{2\la n-2}b^*_{n-2\la n}\underline{b_{n\la n-2}b^*_{n-2\la n-1}b_{n-1\la n-2}b^*_{n-2\la n}}b_{n\la n-2}B_{n-2\la 2}b_{2\la 0}\\
    =&xy' b^*_{0\la 2}B^*_{2\la n-2}b^*_{n-2\la n}b_{n\la n-2}B_{n-2\la 2}b_{2\la 0}=xyy',
\end{aligned}
\end{equation}
where the first and the last equalities only use the definitions of $y,y'$ and $z,z'$ in \cref{Dnyy',Dnzz'}, and the second equality follows from \Cref{Dnind} (taking $r=1$). Then combining \cref{Dny+y',Dnz+z',Dnrelation}, we obtain that
\begin{equation}\label{Dnrelationevenodd}
0=xyy'-zz'=\begin{cases}xy(y-x^{\frac{n-2}{2}})-z^2,\ n\ \text{even}, \\
xy^2-z(z-x^{\frac{n-1}{2}}),\ n\ \text{odd}. \\
\end{cases}
\end{equation}
The isomorphism in \cref{Dnxyzrelation} follows from the same argument used to prove that \cref{Ansurj} is an isomorphism in type $A_n$.
\end{proof}

Consider the following automorphism of the extended Dynkin diagram of type $\widetilde{D}_n$. 
\begin{equation}\label{DnDynkinauto}
\tau(i)=i,\ 0\leq i\leq n-2;\ \tau(n-1)=n;\ \tau(n)=n-1.
\end{equation}
It induces an automorphism of the extended Dynkin quiver $\widetilde{D}_n$ in \Cref{quiverDn} by sending an arrow $a:i\to j$ to the arrow $\tau(a):\tau(i)\to\tau(j)$, which will be used to construct the lifts of the anti-Poisson involutions in \Cref{Dnevencase2sub,Dnoddcase2sub}. Note that $\tau$ coincides with the nontrivial order-two automorphism of the McKay graph of type $\wt{D}_n$ induced by an element $g\in N_{\GL_2(\C)}(\Gamma)^-$ (cf. \Cref{DnevenAPIrem,DnoddAPIrem}). \Cref{Dnevencase1sub,Dnevencase2sub} study the preimages of fixed point loci of anti-Poisson involutions for Kleinian singularities of type $D_n$, with $n$ even. We have $X=\Spec\C[x,y,z]/(xy(y-x^{\frac{n-2}{2}})-z^2)$.

\subsection{Type $D_n,\ n$ even, case \RNum{1}} \label{Dnevencase1sub}
In this section, we consider the anti-Poisson involution
\[
  \theta(x)=x,\ \theta(y)=y,\  \theta(z)=-z.
\]
The fixed point locus $X^{\theta}=\Spec\C[X]/(z)=\Spec\C[x,y]/(xy(y-x^{\frac{n-2}{2}}))$ has three irreducible components $L_j,\ j=1,2,3$, and each of them is an $\A^1$. Following the notation in \Cref{liftsectionq1}, we have $\pi^{-1}(X^{\theta})_{\red}=\pi^{-1}(0)_{\red}\cup\tilde{L}_1\cup\tilde{L}_2\cup\tilde{L}_3$.

\begin{proposition}\label[proposition]{Dnevenliftz}
The involution $\ttheta\colon\M_{\chi}(\delta,w)\to\M_{\chi}(\delta,w)$ defined by
\begin{equation}\label{Dnevenliftzequation}
    \ttheta([B,B^*,l_0,k_0])=[-B,B^*,l_0,-k_0]
\end{equation}
is anti-symplectic. Moreover, it is a lift of $\theta\colon X\to X$.
\end{proposition}
\begin{proof}
We consider the following linear anti-symplectic involution on the representation space.
\begin{equation*}
    \Theta\colon M(\delta,w)\to M(\delta,w),\ \ (B,B^*,l_0,k_0)\mapsto (-B,B^*,l_0,-k_0).
\end{equation*}
It satisfies the conditions listed in \Cref{liftconditions}.
    \begin{enumerate}
    \item $g^{\Theta}=g$.
    \item $\Theta(x)=(-1)^2x=x,\ \Theta(y)=(-1)^{n-2}y=y,\ \Theta(z)=(-1)^{n-1}z=-z$.
\end{enumerate}
The involution $\ttheta$ defined in \cref{Dnevenliftzequation} is the descent of $\Theta$ to $\M_{\chi}(\delta,w)$. The rest follows from \Cref{liftconditions}.
\end{proof}
The action of $\ttheta$ on the exceptional fiber $\pi^{-1}(0)$ is described as follows.
\begin{proposition}\label[proposition]{Dnevencase1preimage}
\begin{enumerate}[label=(\arabic*)]
   \item 
   The involution $\ttheta$ preserves all the components $C_i,\ 1\leq i\leq n$. We have $\pi^{-1}(0)^{\ttheta}=\bigcup_{\substack{i\ \text{even}\\ 2\leq i\leq n-2}}C_i\cup\{p_1,p_2,p_3\}$ where $p_1\in C_1\setminus C_2,\ p_2\in C_{n-1}\setminus C_{n-2}$, and $p_3\in C_n\setminus C_{n-2}$.
   \item\label{Dnevencase1temp2}  With a suitable choice of labeling, we have $\tilde{L}_1\cap\pi^{-1}(0)=\tilde{L}_1\cap C_1=\{p_1\},\ \tilde{L}_2\cap\pi^{-1}(0)=\tilde{L}_2\cap C_{n-1}=\{p_2\}$, and $\tilde{L}_3\cap\pi^{-1}(0)=\tilde{L}_3\cap C_n=\{p_3\}$.
   \item\label{Dnevencase1temp3}  The scheme $\pi^{-1}(X^{\theta})$ is not reduced. As a divisor, we have $\pi^{-1}(X^{\theta})=\tilde{L}_1+\tilde{L}_2+\tilde{L}_3+\sum_{i=1}^na_iC_i$, where $a_i=i+1,\ 1\leq i\leq n-2$ and $a_{n-1}=a_n=\frac{n}{2}$.
\end{enumerate}
\end{proposition}
\begin{proof}
\begin{enumerate}[label=(\arabic*)]
    \item It is easy to see that $\ttheta(C_i)=C_i$ thanks to \Cref{Ciequation}. The component $C_{n-2}$ intersects with three other components $C_{n-3},\ C_{n-1},\ C_n$, hence there are at least three $\ttheta$-fixed points on $C_{n-2}$ (the intersection points with the adjacent components). It follows from \ref{autop2} of \Cref{liftoncomponent} that $\ttheta$ acts trivially on $C_{n-2}$. By \ref{autop3} of \Cref{liftoncomponent}, $\ttheta$ has exactly two fixed points on each of $C_{n-3},C_{n-1},C_n$, resulting in three isolated points in $\pi^{-1}(0)^{\ttheta}$. The fact that $\ttheta$ fixes each $C_i,\ 2\leq i\leq n-2$ with $i$ even, follows from \Cref{liftoncomponentcor}.
    \item This follows from \Cref{isolatedtransversal}.
    \item Recall the definition of $b_i$ in \cref{bidef}. We know that $
   (b_1,b_2,\cdots,b_{n-2},b_{n-1},b_n)=(1,0,\cdots,0,1,1)$ from \ref{Dnevencase1temp2}. The divisor $X^{\theta}=\ddiv(z)$ is principal. Then applying \Cref{multofcomponent}, we can determine that
   \begin{equation*}
            (a_1,\cdots,a_i,\cdots,a_{n-2},a_{n-1},a_n)=
            (2,\cdots,i+1,\cdots,n-1,\frac{n}{2},\frac{n}{2}).
   \end{equation*}
\end{enumerate}
\end{proof}

The fixed point locus $X^{\theta}$ and the preimage $\pi^{-1}(X^{\theta})$ are depicted in \Cref{Dnevencase1pic}. The number next to each component indicates its multiplicity. Thickened lines indicate that the multiplicities are larger than $1$. The $\ttheta$-fixed components of $\pi^{-1}(X^{\theta})$ are colored in red.

\begin{mtable}[ht]
\begin{tabular}{ccc}
    \begin{tabular}{c}\begin{tikzpicture}
    \draw[thick] (-0.6,1) to (0.6,-1);
    \draw[thick] (0,1.2) to (0,-1.2);
    \draw[thick] (0.6,1) to (-0.6,-1);
    \end{tikzpicture}\end{tabular} & $\qquad$ &
    \begin{tabular}{c}\begin{tikzpicture}
    \draw[line width=0.4mm,red] (-3,2) edge [bend left=25] node [below]{} (-3,-1.5);
    \node at (-3,-1.7) () {{\color{red} $n\minus1$}};
    \draw[line width=0.4mm] (-3.5,-1) edge [bend left=35] node [below]{} (-1,-1);
    \node at (-3.7,-1) () {$\frac{n}{2}$};
    \draw[thick,red] (-1,-0.2) to (-1.6,-1.7);
    \node at (-1.6,-1.9) () {{\color{red} $1$}};
    \draw[line width=0.4mm] (-3.5,1) edge [bend left=35] node [below]{} (-1,1);
    \node at (-3.7,1) () {$\frac{n}{2}$};
    \draw[thick,red] (-1.6,2.1) to (-1,0.6);
    \node at (-1.6,2.3) () {{\color{red} $1$}};
    \draw[line width=0.4mm] (-3.5,0) edge [bend left=35] node [below]{} (-1,0);
    \node at (-3.9,0) () {$n\minus 2$};
    \draw[line width=0.4mm,red] (-2,0) edge [bend left=35] node [below]{} (0.5,0);
    \node at (-2,-0.15) () {{\color{red} $n\minus 3$}};
    \draw[line width=0.4mm] (-0.5,0) edge [bend left=25] node [below]{} (0.75,0.3);
    \node at (-0.5,-0.15) () {$n\minus 4$};
    \node at (1.5,0.2) (cdots) {$\cdots$};
    \draw[line width=0.4mm] (3.5,0) edge [bend right=25] node [below]{} (2.25,0.3);
    \node at (3.5,-0.15) () {$4$};
    \draw[line width=0.4mm,red] (5,0) edge [bend right=35] node [below]{} (2.5,0);
    \node at (5,-0.15) () {{\color{red} $3$}};
    \draw[line width=0.4mm] (6.5,0) edge [bend right=35] node [below]{} (4,0);
    \node at (6.6,-0.15) () {$2$};
    \draw[thick,red] (6.4,0.85) to (5.7,-0.65);
    \node at (5.7,-0.9) () {{\color{red} $1$}};
    \end{tikzpicture}\end{tabular}\\
    $X^{\theta}$ & & $\pi^{-1}(X^{\theta})$
\end{tabular}
\caption{$D_n,\ n$ even, Type \RMN{1}}
\label{Dnevencase1pic}
\end{mtable}

\subsection{Type $D_n,\ n$ even, case \RNum{2}} \label{Dnevencase2sub}
In this section, we consider the anti-Poisson involution
\[
  \theta(x)=x,\ \theta(y)=x^{\frac{n-2}{2}}-y,\  \theta(z)=z.
\]
The fixed point locus $X^{\theta}=\Spec\C[X]/(2y-x^{\frac{n-2}{2}})=\Spec\C[x,z]/(x^{n-1}+4z^2)$ is a cusp (a single irreducible component $L$). Following the notation in \Cref{liftsectionq1}, we have $\pi^{-1}(X^{\theta})_{\red}=\pi^{-1}(0)_{\red}\cup\tilde{L}$. 

\begin{proposition}\label[proposition]{Dnevenliftyy'}
The involution $\ttheta\colon\M_{\chi}(\delta,w)\to\M_{\chi}(\delta,w)$ defined by
\begin{equation}\label{Dnevenliftyy'equation}
    \ttheta([B_{a},B^*_{a},l_0,k_0])=[-B_{\tau(a)},B^*_{\tau(a)},l_0,-k_0]
\end{equation}
is anti-symplectic. Moreover, it is a lift of $\theta\colon X\to X$.
\end{proposition}
\begin{proof}
We consider the following linear anti-symplectic involution on the representation space.
\begin{equation*}
    \Theta\colon M(\delta,w)\to M(\delta,w),\ \ (B_{a},B^*_{a},l_0,k_0)\mapsto (-B_{\tau(a)},B^*_{\tau(a)},l_0,-k_0).
\end{equation*}
It satisfies the conditions listed in \Cref{liftconditions}.
    \begin{enumerate}    
    \item $(g_i)_{i\in I}^{\Theta}=(g_{\tau(i)})_{i\in I}$.
    \item $\Theta(x)=(-1)^2x=x,\ \Theta(y)=(-1)^{n-2}y'\xlongequal{\cref{Dny+y'}}x^{\frac{n-2}{2}}-y,\ \Theta(z)=(-1)^{n-1}z'\xlongequal{\cref{Dnz+z'}}z$.
\end{enumerate}
The involution $\ttheta$ defined in \cref{Dnevenliftyy'equation} is the descent of $\Theta$ to $\M_{\chi}(\delta,w)$. The rest follows from \Cref{liftconditions}.
\end{proof}
The action of $\ttheta$ on the exceptional fiber $\pi^{-1}(0)$ is described as follows.
\begin{proposition}\label[proposition]{Dnevencase2preimage}
\begin{enumerate}[label=(\arabic*)]
   \item 
   The involution $\ttheta$ preserves the components $C_i,\ 1\leq i\leq n-2$, and swaps the two components $C_{n-1},C_n$. We have $\pi^{-1}(0)^{\ttheta}=\bigcup_{\substack{i\ \text{odd}\\1\leq i\leq n-3}}C_i\cup\{p\}$ where $p\in C_{n-2}\setminus (C_{n-3}\cup C_{n-1}\cup C_n)$.
   \item\label{Dnevencase2temp2} We have $\tilde{L}\cap\pi^{-1}(0)=\tilde{L}\cap C_{n-2}=\{p\}$.
   \item The scheme $\pi^{-1}(X^{\theta})$ is not reduced. As a divisor, we have $\pi^{-1}(X^{\theta})=\tilde{L}+\sum_{i=1}^na_iC_i$, where $a_i=i,\ 1\leq i\leq n-2$ and $a_{n-1}=a_n=\frac{n-2}{2}$.
\end{enumerate}
\end{proposition}
\begin{proof}
\begin{enumerate}[label=(\arabic*)]
    \item It is easy to see that $\ttheta(C_i)=C_{\tau(i)}$ thanks to \Cref{Ciequation}. The intersection points $C_{n-1}\cap C_{n-2},\ C_{n}\cap C_{n-2}$ are permuted by $\ttheta$. It follows from \ref{autop2} of \Cref{liftoncomponent} that there are exactly two $\ttheta$-fixed points on $C_{n-2}$. The intersection point $C_{n-3}\cap C_{n-2}$ is one of the $\ttheta$-fixed points on $C_{n-2}$, and the other one is an isolated point in $\pi^{-1}(0)^{\ttheta}$. It follows from \Cref{liftoncomponentcor} that the $\ttheta$-fixed components in $\{C_i\ |1\leq i\leq n-2\}$ appear alternately.
    \item This follows from \Cref{isolatedtransversal}.
    \item Recall the definition of $b_i$ in \cref{bidef}. We have $(b_1,\cdots,b_{n-3},b_{n-2},b_{n-1},b_n)=(0,\cdots,0,1,0,0)$ by \ref{E6case2temp2}. Moreover, $X^{\theta}=\ddiv(2y-x^{\frac{n-2}{2}})$ is a principal divisor. Then applying \Cref{multofcomponent}, we can determine that
\begin{equation*}
    (a_1,\cdots,a_i,\cdots,a_{n-2},a_{n-1},a_n)=(1,\cdots,i,\cdots,n-2,\frac{n-2}{2},\frac{n-2}{2}).
\end{equation*}
\end{enumerate}
\end{proof}

The fixed point locus $X^{\theta}$ and the preimage $\pi^{-1}(X^{\theta})$ are depicted in \Cref{Dnevencase2pic}. The number next to each component indicates its multiplicity. Thickened lines indicate that the multiplicities are larger than $1$. The $\ttheta$-fixed components of $\pi^{-1}(X^{\theta})$ are colored in red.

\begin{mtable}[ht]
\begin{tabular}{ccc}
    \begin{tabular}{c}
    \begin{tikzpicture}
    \draw[thick] (0.5,-1.2) edge [bend right=20] node [below]{} (-0.3,0);
    \draw[thick] (0.5,1.2) edge [bend left=20] node [below]{} (-0.3,0);
    \end{tikzpicture}\end{tabular} & $\qquad$ &
    \begin{tabular}{c}\begin{tikzpicture}
    \draw[line width=0.4mm] (-3,2) edge [bend left=25] node [below]{} (-3,-1.5);
    \node at (-3,-1.7) () {$n\minus2$};
    \draw[line width=0.4mm] (-3.5,-1) edge [bend left=35] node [below]{} (-1,-1);
    \node at (-3.7,-0.85) () {$\frac{n\minus2}{2}$};
    \draw[line width=0.4mm] (-3.5,1) edge [bend left=35] node [below]{} (-1,1);
    \node at (-3.7,1.2) () {$\frac{n\minus2}{2}$};
    \draw[thick,red] (-3.5,0.75) to (-0.75,0.75);
    \node at (-3.7,0.6) () {{\color{red} $1$}};
  
    \draw[line width=0.4mm,red] (-3.5,0) edge [bend left=35] node [below]{} (-1,0);
    \node at (-3.9,0) () {{\color{red} $n\minus3$}};
    \draw[line width=0.4mm] (-2,0) edge [bend left=35] node [below]{} (0.5,0);
    \node at (-2,-0.15) () {$n\minus4$};
    \draw[line width=0.4mm,red] (-0.5,0) edge [bend left=25] node [below]{} (0.75,0.3);
    \node at (-0.4,-0.15) () {{\color{red} $n\minus5$}};
    \node at (1.5,0.2) (cdots) {$\cdots$};
    \draw[line width=0.4mm,red] (3.5,0) edge [bend right=25] node [below]{} (2.25,0.3);
    \node at (3.5,-0.15) () {{\color{red} $3$}};
    \draw[line width=0.4mm] (5,0) edge [bend right=35] node [below]{} (2.5,0);
    \node at (5,-0.15) () {$2$};
    \draw[thick,red] (6.5,0) edge [bend right=35] node [below]{} (4,0);
    \node at (6.6,-0.15) () {{\color{red} $1$}};
    \end{tikzpicture}\end{tabular} \\
    $X^{\theta}$ & & $\pi^{-1}(X^{\theta})$
\end{tabular}
\caption{$D_n,\ n$ even, Type \RMN{2}}
\label{Dnevencase2pic}
\end{mtable}

\Cref{Dnoddcase1sub,Dnoddcase2sub} study the preimages of fixed point loci of anti-Poisson involutions for Kleinian singularities of type $D_n$, with $n$ odd. We have $X=\Spec\C[x,y,z]/(xy^2-z(z-x^{\frac{n-1}{2}}))$.

\subsection{Type $D_n,\ n$ odd, case \RNum{1}} \label{Dnoddcase1sub}
In this section, we consider the anti-Poisson involution
\[
  \theta(x)=x,\ \theta(y)=-y,\  \theta(z)=z.
\]
The fixed point locus $X^{\theta}=\Spec\C[X]/(y)=\Spec\C[x,z]/(z(z-x^{\frac{n-1}{2}}))$ has two irreducible components $L_j,\ j=1,2$, and each of them is an $\A^1$. Following the notation in \Cref{liftsectionq1}, we have $\pi^{-1}(X^{\theta})_{\red}=\pi^{-1}(0)_{\red}\cup\tilde{L}_1\cup\tilde{L}_2$. The argument in this section is very similar to that of \Cref{Dnevencase1sub}, with minor changes due to parity. We state the results, but omit the proofs.

\begin{proposition}\label[proposition]{Dnoddlifty}
The involution $\ttheta\colon\M_{\chi}(\delta,w)\to\M_{\chi}(\delta,w)$ defined by
\begin{equation}\label{Dnoddliftyequation}
    \ttheta([B,B^*,l_0,k_0])=[-B,B^*,l_0,-k_0]
\end{equation}
is anti-symplectic. Moreover, it is a lift of $\theta\colon X\to X$.
\end{proposition}

\begin{proposition}\label[proposition]{Dnoddcase1preimage}
\begin{enumerate}[label=(\arabic*)]
   \item 
   The involution $\ttheta$ preserves all the components $C_i,\ 1\leq i\leq n$. We have $\pi^{-1}(0)^{\ttheta}=\bigcup_{\substack{i\ \text{odd}\\1\leq i\leq n-2}}C_i\cup\{p_1,p_2\}$ where $p_1\in C_{n-1}\setminus C_{n-2}$ and $p_2\in C_n\setminus C_{n-2}$.
   \item\label{Dnoddcase1temp2}  With a suitable choice of labeling, we have $\ \tilde{L}_1\cap\pi^{-1}(0)=\tilde{L}_1\cap C_{n-1}=\{p_1\}$ and $\tilde{L}_2\cap\pi^{-1}(0)=\tilde{L}_2\cap C_n=\{p_2\}$.
   \item The scheme $\pi^{-1}(X^{\theta})$ is not reduced. As a divisor, we have $\pi^{-1}(X^{\theta})=\tilde{L}_1+\tilde{L}_2+\sum_{i=1}^na_iC_i$, where $a_i=i,\ 1\leq i\leq n-2$ and $a_{n-1}=a_n=\frac{n-1}{2}$.
\end{enumerate}
\end{proposition}

The fixed point locus $X^{\theta}$ and the preimage $\pi^{-1}(X^{\theta})$ are depicted in \Cref{Dnoddcase1pic}. The number next to each component indicates its multiplicity. Thickened lines indicate that the multiplicities are larger than $1$. The $\ttheta$-fixed components of $\pi^{-1}(X^{\theta})$ are colored in red.

\begin{mtable}[ht]
\begin{tabular}{ccc}
    \begin{tabular}{c}\begin{tikzpicture}
    \draw[thick] (-0.6,1) to (0.6,-1);
    \draw[thick] (0.6,1) to (-0.6,-1);
    \end{tikzpicture}\end{tabular} & $\qquad$ &
    \begin{tabular}{c}\begin{tikzpicture}
    \draw[line width=0.4mm,red] (-3,2) edge [bend left=25] node [below]{} (-3,-1.5);
    \node at (-3,-1.7) () {{\color{red} $n\minus2$}};
    \draw[line width=0.4mm] (-3.5,-1) edge [bend left=35] node [below]{} (-1,-1);
    \node at (-3.7,-0.85) () {$\frac{n\minus1}{2}$};
    \draw[thick,red] (-1,-0.2) to (-1.6,-1.7);
    \node at (-1.6,-1.9) () {{\color{red} $1$}};
    \draw[line width=0.4mm] (-3.5,1) edge [bend left=35] node [below]{} (-1,1);
    \node at (-3.7,1.2) () {$\frac{n\minus1}{2}$};
    \draw[thick,red] (-1.6,2.1) to (-1,0.6);
    \node at (-1.6,2.3) () {{\color{red} $1$}};
    \draw[line width=0.4mm] (-3.5,0) edge [bend left=35] node [below]{} (-1,0);
    \node at (-3.9,0) () {$n\minus 3$};
    \draw[line width=0.4mm,red] (-2,0) edge [bend left=35] node [below]{} (0.5,0);
    \node at (-2,-0.15) () {{\color{red} $n\minus 4$}};
    \draw[line width=0.4mm] (-0.5,0) edge [bend left=25] node [below]{} (0.75,0.3);
    \node at (-0.5,-0.15) () {$n\minus 5$};
    \node at (1.5,0.2) (cdots) {$\cdots$};
    \draw[line width=0.4mm] (3.5,0) edge [bend right=25] node [below]{} (2.25,0.3);
    \node at (3.5,-0.15) () {$2$};
    \draw[thick,red] (5,0) edge [bend right=35] node [below]{} (2.5,0);
    \node at (5.1,-0.15) () {{\color{red} $1$}};
  \end{tikzpicture}\end{tabular}\\
  $X^{\theta}$ & & $\pi^{-1}(X^{\theta})$
\end{tabular}
\caption{$D_n,\ n$ odd, Type \RMN{1}}
\label{Dnoddcase1pic}
\end{mtable}

\subsection{Type $D_n,\ n$ odd, case \RNum{2}} \label{Dnoddcase2sub}
In this section, we consider the anti-Poisson involution
\[
  \theta(x)=x,\ \theta(y)=y,\  \theta(z)=x^{\frac{n-1}{2}}-z.
\]
The fixed point locus $X^{\theta}=\Spec\C[X]/(2z-x^{\frac{n-1}{2}})=\Spec\C[x,y]/(x(4y^2+x^{n-2}))$ has two irreducible components $L_j,\ j=1,2$. One of them is $\A^1$ and the other is a cusp. Following the notation in \Cref{liftsectionq1}, we have $\pi^{-1}(X^{\theta})_{\red}=\pi^{-1}(0)_{\red}\cup\tilde{L}_1\cup\tilde{L}_2$. The argument in this section is very similar to that of \Cref{Dnevencase2sub}, with minor changes due to parity. We state the results, but omit the proofs. Recall the automorphism $\tau$ defined on the extended Dynkin diagram of type $\wt{D}_n$ in \cref{DnDynkinauto}.

\begin{proposition}\label[proposition]{Dnoddliftzz'}
The involution $\ttheta\colon\M_{\chi}(\delta,w)\to\M_{\chi}(\delta,w)$ defined by
\begin{equation}\label{Dnoddliftzz'equation}
    \ttheta([B_{a},B^*_{a},l_0,k_0])=[-B_{\tau(a)},B^*_{\tau(a)},l_0,-k_0]
\end{equation}
is anti-symplectic. Moreover, it is a lift of $\theta\colon X\to X$.
\end{proposition}

\begin{proposition}\label[proposition]{Dnoddcase2preimage}
\begin{enumerate}[label=(\arabic*)]
   \item 
   The involution $\ttheta$ preserves the components $C_i,\ 1\leq i\leq n-2$, and swaps the two components $C_{n-1},C_n$. We have $\pi^{-1}(0)^{\ttheta}=\bigcup_{\substack{i\ \text{even}\\ 2\leq i\leq n-3}}C_i\cup\{p_1,p_2\}$ where $p_1\in C_1\setminus C_2$ and $p_2\in C_{n-2}\setminus (C_{n-3}\cup C_{n-1}\cup C_n)$.
   \item\label{Dnoddcase2temp2} With a suitable choice of labeling, $\tilde{L}_1\cap\pi^{-1}(0)=\tilde{L}_1\cap C_1=\{p_1\},\ \tilde{L}_2\cap\pi^{-1}(0)=\tilde{L}_2\cap C_{n-2}=\{p_2\}$.
   \item The scheme $\pi^{-1}(X^{\theta})$ is not reduced. As a divisor, we have $\pi^{-1}(X^{\theta})=\tilde{L}_1+\tilde{L}_2+\sum_{i=1}^na_iC_i$, where $a_i=i+1,\ 1\leq i\leq n-2$ and $a_{n-1}=a_n=\frac{n-1}{2}$.
\end{enumerate}
\end{proposition}

The fixed point locus $X^{\theta}$ and the preimage $\pi^{-1}(X^{\theta})$ are depicted in \Cref{Dnoddcase2pic}. The number next to each component indicates its multiplicity. Thickened lines indicate that the multiplicities are larger than $1$. The $\ttheta$-fixed components of $\pi^{-1}(X^{\theta})$ are colored in red.

\begin{mtable}[ht]
\begin{tabular}{ccc}
    \begin{tabular}{c}
    \begin{tikzpicture}
    \draw[thick] (0.5,-1.2) edge [bend right=20] node [below]{} (-0.3,0);
    \draw[thick] (0.5,1.2) edge [bend left=20] node [below]{} (-0.3,0);
    \draw[thick] (-0.3,1.2) to (-0.3,-1.2);
    \end{tikzpicture}\end{tabular} & $\qquad$ &
    \begin{tabular}{c}\begin{tikzpicture}
    \draw[line width=0.4mm] (-3,2) edge [bend left=25] node [below]{} (-3,-1.5);
    \node at (-3,-1.7) () {$n\minus1$};
    \draw[line width=0.4mm] (-3.5,-1) edge [bend left=35] node [below]{} (-1,-1);
    \node at (-3.7,-0.85) () {$\frac{n\minus1}{2}$};
    \draw[line width=0.4mm] (-3.5,1) edge [bend left=35] node [below]{} (-1,1);
    \node at (-3.7,1.2) () {$\frac{n\minus1}{2}$};
    \draw[thick,red] (-3.5,0.75) to (-0.75,0.75);
    \node at (-3.7,0.6) () {{\color{red} $1$}};
  
    \draw[line width=0.4mm,red] (-3.5,0) edge [bend left=35] node [below]{} (-1,0);
    \node at (-3.9,0) () {{\color{red} $n\minus2$}};
    \draw[line width=0.4mm] (-2,0) edge [bend left=35] node [below]{} (0.5,0);
    \node at (-2,-0.15) () {$n\minus3$};
    \draw[line width=0.4mm,red] (-0.5,0) edge [bend left=25] node [below]{} (0.75,0.3);
    \node at (-0.4,-0.15) () {{\color{red} $n\minus4$}};
    \node at (1.5,0.2) (cdots) {$\cdots$};
    \draw[line width=0.4mm,red] (3.5,0) edge [bend right=25] node [below]{} (2.25,0.3);
    \node at (3.5,-0.15) () {{\color{red} $3$}};
    \draw[line width=0.4mm] (5,0) edge [bend right=35] node [below]{} (2.5,0);
    \node at (5.1,-0.15) () {$2$};
    \draw[thick,red] (4.9,0.85) to (4.2,-0.65);
    \node at (4.2,-0.9) () {{\color{red} $1$}};
    \end{tikzpicture}\end{tabular}\\
    $X^{\theta}$ & & $\pi^{-1}(X^{\theta})$
\end{tabular}
\caption{$D_n,\ n$ odd, Type \RMN{2}}
\label{Dnoddcase2pic}
\end{mtable}

\section{Preimages of fixed point loci for type $E_6$ Kleinian singularity}\label{E6quiver}

\begin{mtable}[ht]
    \centering
\begin{tabular}{c}
    \scalebox{1.3}{\begin{tikzpicture}
        \node at (-2,-1.2) (E6) {$\widetilde{E}_6$};
        \filldraw[black] (1.5,-3) circle (2pt) node[anchor=east]{$W_0$};
        \node[] at (1.5,-3.3)  {\color{blue} $1$};
        \filldraw[black] (3,-3) circle (2pt) node[anchor=west]{$V_0$};
        \node[] at (3,-3.3)  {\color{blue} $1$};
        \filldraw[black] (3,-1.5) circle (2pt) node[anchor=west]{$V_1$};
        \node[] at (2.7,-1.5)  {\color{blue} $2$};
        
        \filldraw[black] (0,0) circle (2pt) node[anchor=south]{$V_4$};
        \node[] at (0,0.7)  {\color{blue} $1$};
        \filldraw[black] (1.5,0) circle (2pt) node[anchor=south]{$V_3$};
        \node[] at (1.5,0.7)  {\color{blue} $2$};
        \filldraw[black] (3,0) circle (2pt) node[anchor=south]{$V_2$};
        \node[] at (3,0.7)  {\color{blue} $3$};
        \filldraw[black] (4.5,0) circle (2pt) node[anchor=south]{$V_5$};
        \node[] at (4.5,0.7)  {\color{blue} $2$};
        \filldraw[black] (6,0) circle (2pt) node[anchor=south]{$V_6$};
        \node[] at (6,0.7)  {\color{blue} $1$};

        \draw[->,thick] (3.05,-1.4) -- (3.05,-0.1);
        \node at (3.5,-0.8)  {\scalebox{0.9}{$B_{2\leftarrow1}$}};
        \draw[->,thick] (2.95,-0.1) -- (2.95,-1.4);
        \node[] at (2.5,-0.8)  {\scalebox{0.9}{$B^*_{1\leftarrow 2}$}};

        \draw[->,thick] (3.05,-2.9) -- (3.05,-1.6);
        \node at (3.5,-2.15)  {\scalebox{0.9}{$B_{1\leftarrow0}$}};
        \draw[<-,thick] (2.95,-2.9) -- (2.95,-1.6);
        \node at (2.5,-2.15)  {\scalebox{0.9}{$B^*_{0\leftarrow1}$}};

        \draw[->,thick] (1.6,-2.95) -- (2.9,-2.95);
        \node[] at (2.25,-2.75)  {\scalebox{0.9}{$l_0$}};
        \draw[<-,thick] (1.6,-3.05) -- (2.9,-3.05);
        \node[] at (2.25,-3.25)  {\scalebox{0.9}{$k_0$}};
        
        \draw[<-,thick] (0.1,0.05) -- (1.4,0.05);
        \node at (0.7,0.25)  {\scalebox{0.9}{$B_{4\leftarrow3}$}};
        \draw[->,thick] (0.1,-0.05) -- (1.4,-0.05);
        \node[] at (0.7,-0.25)  {\scalebox{0.9}{$B^*_{3\leftarrow 4}$}};
        \draw[<-,thick] (1.6,0.05) -- (2.9,0.05);
        \node at (2.2,0.25)  {\scalebox{0.9}{$B_{3\leftarrow2}$}};
        \draw[->,thick] (1.6,-0.05) -- (2.9,-0.05);
        \node[] at (2.2,-0.25) {\scalebox{0.9}{$B^*_{2\leftarrow 3}$}};
        
        \draw[->,thick] (3.1,0.05) -- (4.4,0.05);
        \node at (3.7,0.25)  {\scalebox{0.9}{$B_{5\leftarrow2}$}};
        \draw[<-,thick,] (3.1,-0.05) -- (4.4,-0.05);
        \node[] at (3.7,-0.25)  {\scalebox{0.9}{$B^*_{2\leftarrow 5}$}};
        \draw[->,thick] (4.6,0.05) -- (5.9,0.05);
        \node at (5.2,0.25) {\scalebox{0.9}{$B_{6\leftarrow5}$}};
        \draw[<-,thick] (4.6,-0.05) -- (5.9,-0.05);
        \node[] at (5.2,-0.25)  {\scalebox{0.9}{$B^*_{5\leftarrow 6}$}};
    \end{tikzpicture}}
\end{tabular}
\caption{Extended Dynkin quiver $\widetilde{E}_6$}
\label{quiverE6}
\end{mtable}

We realize the Kleinian singularity of type $E_6: X=\Spec\C[x,y,z]/(z^2+zx^2+y^3)$ as a Nakajima quiver variety explicitly (\Cref{quiverE6}). Recall that $\delta=(\delta_i)_{0\leq i\leq 6}=(1,2,3,2,1,2,1)$. We have $\dim W_0=1,\ \dim V_i=\delta_i$. The number (colored in blue) next to each vector space indicates its dimension. The representation space is 
\[
M(\delta,w)=T^*\big(\bigoplus_{a\in\Omega}\Hom(V_{t(a)},V_{h(a)})\oplus\Hom(W_0,V_0)\big).
\]
The group $\GL(\delta)=\prod_{i=0}^6\GL(V_i)$ acts on $M(\delta,w)$ naturally with the moment map 
\[
\mu=(\mu_i)_{0\leq i\leq 6}\colon M(\delta,w)\to\gl(\delta),
\]
where
\begin{align}
    \mu_0&=-B^*_{0\la 1}B_{1\la 0}+l_0k_0, \notag \\
    \mu_1&=B_{1\la 0}B^*_{0\la 1}-B^*_{1\la 2}B_{2\la 1}, \notag \\
    \mu_2&=B_{2\la 1}B^*_{1\la 2}-B^*_{2\leftarrow 3}B_{3\leftarrow 2}-B^*_{2\leftarrow 5}B_{5\leftarrow 2}, \notag \\
    \mu_3&=B_{3\la 2}B^*_{2\la 3}-B^*_{3\la 4}B_{4\la 3}, \label{E6momentmap} \\
    \mu_4&=B_{4\la 3}B^*_{3\la 4}, \notag\\
    \mu_5&=B_{5\la 2}B^*_{2\la 5}-B^*_{5\la 6}B_{6\la 5}, \notag \\
    \mu_6&=B_{6\la 5}B^*_{5\la 6}. \notag
\end{align}
By \cref{ADEmomentmap}, setting $\mu=0$ implies that $B^*_{0\la 1}B_{1\la 0}=l_0k_0=0$. By \Cref{Ciequation}, we have
\[
    C_i=\big\{[B,B^*,l_0,0]\in\M_{\chi}(\delta,w)\big| \rank(B_{a})_{\substack{a\in\bar{\Omega}\\ t(a)=i}}<\dim V_i \big\},\ 1\leq i\leq 6.
\]
According to \Cref{wframinginv} and \Cref{ADExyz}, the algebra of invariant functions $\C[\mu^{-1}(0)]^{\GL(\delta)}$ is generated by trace functions of cycles that start and end at the vertex $0$ of degrees $6,8,12$. The explicit generators are given in \Cref{E6generatorprop}.

\begin{proposition}\label[proposition]{E6generatorprop}
We have $\C[\mu^{-1}(0)]^{\GL(\delta)}\simeq\C[x,y,z]/(z^2+zx^2+y^3)$, where
\begin{equation}\label{E6xyz}
    \begin{aligned}
    x:=&B^*_{0\la 1}B^*_{1\la 2}B^*_{2\la 3}B_{3\la 2}B_{2\la 1}B_{1\la 0}, \\
    y:=&B^*_{0\la 1}B^*_{1\la 2}B^*_{2\la 3}B^*_{3\la 4}B_{4\la 3}B_{3\la 2}B_{2\la 1}B_{1\la 0}, \\
    z:=&B^*_{0\la 1}B^*_{1\la 2}B^*_{2\la 5}B^*_{5\la 6}B_{6\la 5}B_{5\la 2}B^*_{2\la 3}B^*_{3\la 4}B_{4\la 3}B_{3\la 2}B_{2\la 1}B_{1\la 0}. \\
\end{aligned}
\end{equation}
\end{proposition}
To prove \Cref{E6generatorprop}, we need the following lemma.
\begin{lemma}\label[lemma]{E6lemma}
\begin{enumerate}[label=(\arabic*)]
    \item\label{E6012} $B^*_{1\la 2}B_{2\la 1}B_{1\la 0}=B^*_{0\la 1}B^*_{1\la 2}B_{2\la 1}=0$.
    \item\label{E635cube} $(B^*_{2\la 3}B_{3\la 2})^3=(B^*_{2\la 5}B_{5\la 2})^3=0$. 
    \item\label{E6yequiv} $B_{4\la 3}B_{3\la 2}B^*_{2\la 5}B^*_{5\la 6}B_{6\la 5}B_{5\la 2}B^*_{2\la 3}B^*_{3\la 4}=B^*_{0\la 1}B^*_{1\la 2}B^*_{2\la 5}B^*_{5\la 6}B_{6\la 5}B_{5\la 2}B_{2\la 1}B_{1\la 0}=y$.
    \item\label{E6zequiv} $B^*_{0\la 1}B^*_{1\la 2}B^*_{2\la 3}B_{3\la 2}B^*_{2\la 5}B_{5\la 2}B^*_{2\la 3}B_{3\la 2}B^*_{2\la 5}B_{5\la 2}B_{2\la 1}B_{1\la 0}=z$.
\end{enumerate}
\end{lemma}
\begin{proof}
\begin{enumerate}[label=(\arabic*)]
    \item We have $
    \underline{B^*_{1\la 2}B_{2\la 1}}B_{1\la 0}\stackrel{\mu_1}{=}B_{1\la 0}\underline{B^*_{0\la 1}B_{1\la 0}}\stackrel{\mu_0}{=}0$. Similarly, $B^*_{0\la 1}B^*_{1\la 2}B_{2\la 1}=0$.
    \item We have  $(B^*_{2\la 3}B_{3\la 2})^3\stackrel{\mu_3}{=}B^*_{2\la 3}B^*_{3\la 4}\underline{B_{4\la 3}B^*_{3\la 4}}B_{4\la 3}B_{3\la 2}\stackrel{\mu_4}{=}0$. Similarly, $(B^*_{2\la 5}B_{5\la 2})^3=0$.
    \item We have
    \begin{align*}
        B_{6\la 5}B_{5\la 2}B^*_{2\la 3}B^*_{3\la 4}B_{4\la 3}B_{3\la 2}B^*_{2\la 5}B^*_{5\la 6}\xlongequal{\mu_3,\ \mu_5}&\tr((B^*_{2\la 3}B_{3\la 2})^2(B^*_{2\la 5}B_{5\la 2})^2) \\
        \xlongequal{\ref{E635cube},\ \mu_2}\tr((B^*_{2\la 3}B_{3\la 2})^2(B_{2\la 1}B^*_{1\la 2})^2) \xlongequal{\ \ \mu_1\ \ }&B^*_{0\la 1}B^*_{1\la 2}(B^*_{2\la 3}B_{3\la 2})^2B_{2\la 1}B_{1\la 0}=y.
    \end{align*}
    Similarly, $B^*_{0\la 1}B^*_{1\la 2}B^*_{2\la 5}B^*_{5\la 6}B_{6\la 5}B_{5\la 2}B_{2\la 1}B_{1\la 0}=y$.
    \item We can compute that
    \begin{align*}
        &B^*_{0\la 1}B^*_{1\la 2}B^*_{2\la 3}B_{3\la 2}B^*_{2\la 5}B_{5\la 2}B^*_{2\la 3}B_{3\la 2}B^*_{2\la 5}B_{5\la 2}B_{2\la 1}B_{1\la 0}\\
        \stackrel{\mu_2}{=}&B^*_{0\la 1}B^*_{1\la 2}(B_{2\la 1}B^*_{1\la 2}-B^*_{2\la 5}B_{5\la 2})B^*_{2\la 5}B_{5\la 2}B^*_{2\la 3}B_{3\la 2}(B_{2\la 1}B^*_{1\la 2}-B^*_{2\la 3}B_{3\la 2})B_{2\la 1}B_{1\la 0} \\
        \stackrel{\ref{E6012}}{=}&B^*_{0\la 1}B^*_{1\la 2}(B^*_{2\la 5}B_{5\la 2})^2(B^*_{2\la 3}B_{3\la 2})^2B_{2\la 1}B_{1\la 0}\xlongequal{\mu_3,\ \mu_5}z.
    \end{align*}
\end{enumerate}
\end{proof}

\begin{proof}[Proof of \Cref{E6generatorprop}]
    We first show that $x,y,z$ generate $\C[\mu^{-1}(0)]^{\GL(\delta)}$. It suffices to show that the degree-$d$ component $\C[\mu^{-1}(0)]^{\GL(\delta)}_d$ with $d=6,8,12$, is contained in the subalgebra generated by $x,y,z$. We only need to consider trace functions of cycles such that the vertex $0$ only appears once; otherwise, the function can be written as the product of two trace functions of lower degrees. 
    
    The cycles of degree $6$ that start and end at the vertex $0$ are $0\to1\to2\to3\to2\to1\to0,\ 0\to1\to2\to5\to2\to1\to0,\ 0\to 1\to2\to1\to2\to1\to0$. The corresponding trace functions are
    \begin{equation}\label{E6xx'}
    \begin{aligned}
        x\stackrel{\text{def}}{=}&B^*_{0\la 1}B^*_{1\la 2}B^*_{2\la 3}B_{3\la 2}B_{2\la 1}B_{1\la 0}, \\
        x':=&B^*_{0\la 1}B^*_{1\la 2}B^*_{2\la 5}B_{5\la 2}B_{2\la 1}B_{1\la 0}, \\
        x+x'\stackrel{\mu_2}{=}&B^*_{0\la 1}B^*_{1\la 2}B_{2\la 1}\underline{B^*_{1\la 2}B_{2\la 1}B_{1\la 0}}\xlongequal{\ref{E6012}}0.
    \end{aligned}
    \end{equation}
    This shows both that $x+x'=0$ and that $\C[\mu^{-1}(0)]^{\GL(\delta)}_6=\C x$.

    Next, we consider the cycles that start and end at the vertex $0$ of degree $8$. If the cycle only involves vertices $0,1,2$, then one can show that the corresponding trace function is $0$ by using \ref{E6012} of \Cref{E6lemma}. If the cycle contains (exactly) one of the vertices $4,6$, then it will be of the form $0\to1\to2\to3\to4\to3\to2\to1\to0,\ 0\to1\to2\to5\to6\to5\to2\to1\to0$. The corresponding trace functions are
    \begin{equation}\label{E6y}
        \begin{aligned}
            B^*_{0\la 1}B^*_{1\la 2}B^*_{2\la 3}B^*_{3\la 4}B_{4\la 3}B_{3\la 2}B_{2\la 1}B_{1\la 0}&\stackrel{\text{def}}{=}y,\\
            B^*_{0\la 1}B^*_{1\la 2}B^*_{2\la 5}B^*_{5\la 6}B_{6\la 5}B_{5\la 2}B_{2\la 1}B_{1\la 0}&\stackrel{\ref{E6yequiv}}{=}y.\\
            \end{aligned}
    \end{equation}
     Meanwhile, if the cycle contains neither $4$ nor $6$, then it will be of the form $ 0\to1\to2\to3\to2\to3\to2\to1\to0,\ 0\to1\to2\to3\to2\to5\to2\to1\to0,\ 0\to1\to2\to5\to2\to3\to2\to1\to0,\ 0\to1\to2\to5\to2\to5\to2\to1\to0$. Similar to the proof of \ref{E6yequiv} in \Cref{E6lemma}, one can show that the corresponding trace functions are either $y$ or $-y$. It follows that $\C[\mu^{-1}(0)]^{\GL(\delta)}_8=\C y$.

    Lastly, we consider the following two cycles of degree $12$ that start and end at $0$. 
    \begin{equation}\label{E6zz'}
    \begin{aligned}
      z\stackrel{\text{def}}{=}&B^*_{0\la 1}B^*_{1\la 2}B^*_{2\la 5}B^*_{5\la 6}B_{6\la 5}B_{5\la 2}B^*_{2\la 3}B^*_{3\la 4}B_{4\la 3}B_{3\la 2}B_{2\la 1}B_{1\la 0}, \\
      z':=&B^*_{0\la 1}B^*_{1\la 2}B^*_{2\la 3}B^*_{3\la 4}B_{4\la 3}B_{3\la 2}B^*_{2\la 5}B^*_{5\la 6}B_{6\la 5}B_{5\la 2}B_{2\la 1}B_{1\la 0}.
    \end{aligned}
    \end{equation}
    We have $z'\xlongequal{\mu_3,\ \mu_5}B^*_{0\la 1}B^*_{1\la 2}(B^*_{2\la 3}B_{3\la 2})^2(B^*_{2\la 5}B_{5\la 2})^2B_{2\la 1}B_{1\la 0}$.
    Using \ref{E6zequiv} in \Cref{E6lemma}, we get
    \begin{equation}\label{E6z+z'}
    \begin{aligned}
        z+z'=&B^*_{0\la 1}B^*_{1\la 2}B^*_{2\la 3}B_{3\la 2}(B^*_{2\la 5}B_{5\la 2}B^*_{2\la 3}B_{3\la 2}+B^*_{2\la 3}B_{3\la 2}B^*_{2\la 5}B_{5\la 2})B^*_{2\la 5}B_{5\la 2}B_{2\la 1}B_{1\la 0} \\
        \stackrel{\ref{E635cube}}{=}&B^*_{0\la 1}B^*_{1\la 2}B^*_{2\la 3}B_{3\la 2}(B^*_{2\la 3}B_{3\la 2}+B^*_{2\la 5}B_{5\la 2})^2B^*_{2\la 5}B_{5\la 2}B_{2\la 1}B_{1\la 0} \\
        \stackrel{\mu_2}{=}&B^*_{0\la 1}B^*_{1\la 2}B^*_{2\la 3}B_{3\la 2}(B_{2\la 1}B^*_{1\la 2})^2B^*_{2\la 5}B_{5\la 2}B_{2\la 1}B_{1\la 0} \\
        \stackrel{\mu_1}{=}&B^*_{0\la 1}B^*_{1\la 2}B^*_{2\la 3}B_{3\la 2}B_{2\la 1}B_{1\la 0}B^*_{0\la 1}B^*_{1\la 2}B^*_{2\la 5}B_{5\la 2}B_{2\la 1}B_{1\la 0}=xx'\xlongequal{\cref{E6xx'}}-x^2.
    \end{aligned}
    \end{equation}
    
    Similarly to the case of degree $8$, we can show that $\C[\mu^{-1}(0)]^{\GL(\delta)}_{12}=\C z+\C x^2$. Therefore, $\C[\mu^{-1}(0)]^{\GL(\delta)}$ is generated by the invariant functions $x,y,z$ defined in \cref{E6xyz}. Finally, we have that
    \begin{equation}\label{E6relation}
        \begin{aligned}
           zz'=&B^*_{0\la 1}B^*_{1\la 2}B^*_{2\la 5}B^*_{5\la 6}B_{6\la 5}B_{5\la 2}B^*_{2\la 3}B^*_{3\la 4}\underline{B_{4\la 3}B_{3\la 2}B_{2\la 1}B_{1\la 0}} \\
           &\underline{\cdot B^*_{0\la 1}B^*_{1\la 2}B^*_{2\la 3}B^*_{3\la 4}}B_{4\la 3}B_{3\la 2}B^*_{2\la 5}B^*_{5\la 6}B_{6\la 5}B_{5\la 2}B_{2\la 1}B_{1\la 0}\\
           \stackrel{\text{def}}{=}&yB^*_{0\la 1}B^*_{1\la 2}B^*_{2\la 5}B^*_{5\la 6}\underline{B_{6\la 5}B_{5\la 2}B^*_{2\la 3}B^*_{3\la 4}B_{4\la 3}B_{3\la 2}B^*_{2\la 5}B^*_{5\la 6}}B_{6\la 5}B_{5\la 2}B_{2\la 1}B_{1\la 0} \\
           \stackrel{\ref{E6yequiv}}{=}&y^2B^*_{0\la 1}B^*_{1\la 2}B^*_{2\la 5}B^*_{5\la 6}B_{6\la 5}B_{5\la 2}B_{2\la 1}B_{1\la 0}\stackrel{\ref{E6yequiv}}{=}y^3,
        \end{aligned}
    \end{equation}
Combining \cref{E6z+z',E6relation}, we obtain that $z^2+zx^2+y^3=0$. It follows that there is a surjective homomorphism
\begin{equation}\label{E6surj}
\C[x,y,z]/(z^2+zx^2+y^3)\twoheadrightarrow\C[\mu^{-1}(0)]^{\GL(\delta)},
\end{equation}
which turns out to be an isomorphism by the same argument used in type $A_n$ for \cref{Ansurj}.
\end{proof}

\subsection{Type $E_6$, case \RNum{1}} \label{E6case1sub}
In this section, we consider the anti-Poisson involution
\[
  \theta(x)=-x,\ \theta(y)=y,\  \theta(z)=z.
\]
The fixed point locus $X^{\theta}=\Spec\C[X]/(x)=\Spec\C[y,z]/(z^2+y^3)$ is a cusp (a single irreducible component $L$). Following the notation in \Cref{liftsectionq1}, we have $\pi^{-1}(X^{\theta})_{\red}=\pi^{-1}(0)_{\red}\cup\tilde{L}$.

\begin{proposition}\label[proposition]{E6liftx}
The involution $\ttheta\colon\M_{\chi}(\delta,w)\to\M_{\chi}(\delta,w)$ defined by
\begin{equation}\label{E6liftxequation}
    \ttheta([B,B^*,l_0,k_0])=[-B,B^*,l_0,-k_0]
\end{equation}
is anti-symplectic. Moreover, it is a lift of $\theta\colon X\to X$.
\end{proposition}
\begin{proof}
We consider the following linear anti-symplectic involution on the representation space.
\begin{equation*}
    \Theta\colon M(\delta,w)\to M(\delta,w),\ \ (B,B^*,l_0,k_0)\mapsto (-B,B^*,l_0,-k_0).
\end{equation*}
It satisfies the conditions listed in \Cref{liftconditions}.
    \begin{enumerate}
    \item $g^{\Theta}=g$.
    \item $\Theta(x)=(-1)^3x=-x,\ \Theta(y)=(-1)^4y=y,\ \Theta(z)=(-1)^6z=z$.
\end{enumerate}
The involution $\ttheta$ defined in \cref{E6liftxequation} is the descent of $\Theta$ to $\M_{\chi}(\delta,w)$. The rest follows from \Cref{liftconditions}.
\end{proof}

\begin{proposition}\label[proposition]{E6case1preimage}
\begin{enumerate}[label=(\arabic*)]
   \item The involution $\ttheta$ preserves all the components $C_i,\ 1\leq i\leq 6$. We have $\pi^{-1}(0)^{\ttheta}=C_2\cup C_4\cup C_6\cup\{p\}$ where $p\in C_1\setminus C_2$.
   \item\label{E6case1temp2} We have $\tilde{L}\cap\pi^{-1}(0)=\tilde{L}\cap C_1=\{p\}$.
   \item The scheme $\pi^{-1}(X^{\theta})$ is not reduced. As a divisor, we have $\pi^{-1}(X^{\theta})=\tilde{L}+\sum_{i=1}^6\delta_iC_i=\tilde{L}+\pi^{-1}(0)$. Recall that $(\delta_i)_{1\leq i\leq 6}=(2,3,2,1,2,1)$.
\end{enumerate}
\end{proposition}
\begin{proof}
    \begin{enumerate}[label=(\arabic*)]
        \item It is easy to see that $\ttheta(C_i)=C_i$ thanks to \Cref{Ciequation}. The component $C_2$ intersects three other components $C_1,C_3,C_5$, so there are at least three $\ttheta$-fixed points on $C_2$.
        It follows from \ref{autop2} of \Cref{liftoncomponent} that $\ttheta$ acts trivially on $C_2$. By \ref{autop3} of \Cref{liftoncomponent}, $\ttheta$ has exactly two fixed points on each of $C_1,C_3,C_5$, and $\ttheta$ acts trivially on the components $C_4,C_6$. Therefore, there is a unique isolated point in $\pi^{-1}(0)^{\ttheta}$, which lies in $C_1\setminus C_2$.
        \item This follows from \Cref{isolatedtransversal}.
        \item Recall the definition of $b_i$ from \cref{bidef}. We have $(b_1,b_2,\cdots,b_6)=(1,0,\cdots,0)$ from \ref{E6case1temp2}. Then the multiplicities in $\pi^{-1}(X^{\theta})$ can be determined by \Cref{multofcomponent} because $X^{\theta}=\ddiv(x)$ is a principal divisor.
    \end{enumerate}
\end{proof}

The fixed point locus $X^{\theta}$ and the preimage $\pi^{-1}(X^{\theta})$ are depicted in \Cref{E6case1pic}. The number next to each component indicates its multiplicity. Thickened lines indicate that the multiplicities are larger than $1$. The $\ttheta$-fixed components of $\pi^{-1}(X^{\theta})$ are colored in red.

\begin{mtable}[ht]
\begin{tabular}{ccc}
    \begin{tabular}{c}
    \begin{tikzpicture}
    \draw[thick] (0.5,-1.2) edge [bend right=20] node [below]{} (-0.3,0);
    \draw[thick] (0.5,1.2) edge [bend left=20] node [below]{} (-0.3,0);
    \end{tikzpicture}\end{tabular} & $\qquad$ &
    \begin{tabular}{c}\begin{tikzpicture}
    \draw[line width=0.4mm,red] (-3,2) edge [bend left=25] node [below]{} (-3,-1.5);
    \node at (-3,-1.7) () {\textcolor{red}{$3$}};
    \draw[line width=0.4mm] (-3.5,-1) edge [bend left=35] node [below]{} (-1,-1);
    \node at (-3.7,-1) () {$2$};
    \draw[line width=0.4mm] (-3.5,1) edge [bend left=35] node [below]{} (-1,1);
    \node at (-3.7,1) () {$2$};
    \draw[thick,red] (-1.1,0.85) to (-1.7,-0.5);
    \node at (-1,0.7) () {{\color{red} $1$}};
  
    \draw[line width=0.4mm] (-3.5,0) edge [bend left=35] node [below]{} (-1,0);
    \node at (-3.7,0) () {$2$};
    
    \draw[thick,red] (-2,1) edge [bend left=35] node [below]{} (0.5,1);
    \node at (0.6,0.85) () {\textcolor{red}{$1$}};
    \draw[thick,red] (-2,-1) edge [bend left=35] node [below]{} (0.5,-1);
    \node at (0.6,-1.15) () {\textcolor{red}{$1$}};
\end{tikzpicture}\end{tabular}\\
$X^{\theta}$ & & $\pi^{-1}(X^{\theta})$
\end{tabular}
\caption{$E_6$, Type \RMN{1}}
\label{E6case1pic}
\end{mtable}

\subsection{Type $E_6$, case \RNum{2}} \label{E6case2sub}
In this section, we consider the anti-Poisson involution
\[
  \theta(x)=x,\ \theta(y)=y,\  \theta(z)=-z-x^2.
\]
The fixed point locus $X^{\theta}=\Spec\C[X]/(2z+x^2)=\Spec\C[x,y]/(x^4-4y^3)$ is a cusp (a single irreducible component $L$). Following the notation in \Cref{liftsectionq1}, we have $\pi^{-1}(X^{\theta})_{\red}=\pi^{-1}(0)_{\red}\cup\tilde{L}$. Consider the following automorphism defined on the extended Dynkin diagram of type $\widetilde{E}_6$. 
\begin{equation}\label{E6Dynkinauto}
\tau(i)=i,\ i=0,1,2;\ \tau(3)=5,\ \tau(4)=6,\ \tau(5)=3,\ \tau(6)=4.
\end{equation}
It induces an automorphism of the extended Dynkin quiver $\widetilde{E}_6$ in \Cref{quiverE6} by sending an arrow $a:i\to j$ to the arrow $\tau(a):\tau(i)\to\tau(j)$. Note that $\tau$ coincides with the nontrivial automorphism of the McKay graph of type $\tilde{E}_6$ induced by an element $g\in N_{\GL_2(\C)}(\Gamma)^-$ (cf. \Cref{E6APIrem}).

\begin{proposition}\label[proposition]{E6liftz}
The involution $\ttheta\colon\M_{\chi}(\delta,w)\to\M_{\chi}(\delta,w)$ defined by
\begin{equation}\label{E6liftzequation}
    \ttheta([B_{a},B^*_{a},l_0,k_0])=[-B_{\tau(a)},B^*_{\tau(a)},l_0,-k_0]
\end{equation}
is anti-symplectic. Moreover, it is a lift of $\theta\colon X\to X$.
\end{proposition}
\begin{proof}
We consider the following linear anti-symplectic involution on the representation space. 
\begin{equation*}
     \Theta\colon M(\delta,w)\to M(\delta,w),\ \ (B_{a},B^*_{a},l_0,k_0)\mapsto (-B_{\tau(a)},B^*_{\tau(a)},l_0,-k_0).
    \end{equation*}
It satisfies the conditions listed in \Cref{liftconditions}.
    \begin{enumerate}
    \item $(g_i)_{i\in I}^{\Theta}=(g_{\tau(i)})_{i\in I}$.
    \item $\Theta(x)=(-1)^3x'\xlongequal{\cref{E6xx'}}x,\ \Theta(y)\xlongequal{\Cref{E6lemma},\ \ref{E6yequiv}}(-1)^4y=y,\ \Theta(z)=(-1)^6z'
    \xlongequal{\cref{E6z+z'}}-z-x^2$.
\end{enumerate}
The involution $\ttheta$ defined in \cref{E6liftzequation} is the descent of $\Theta$ to $\M_{\chi}(\delta,w)$. The rest follows from \Cref{liftconditions}.
\end{proof}

\begin{proposition}\label[proposition]{E6case2preimage}
\begin{enumerate}[label=(\arabic*)]
   \item 
   The involution $\ttheta$ preserves the components $C_1,C_2$. It swaps $C_3$ with $C_5$, and $C_4$ with $C_6$. We have $\pi^{-1}(0)^{\ttheta}=C_1\cup\{p\}$ where $p\in C_2\setminus (C_1\cup C_3\cup C_5)$.
   \item\label{E6case2temp2} We have $\tilde{L}\cap\pi^{-1}(0)=\tilde{L}\cap C_2=\{p\}$.
   \item The scheme $\pi^{-1}(X^{\theta})$ is not reduced. As a divisor, we have $\pi^{-1}(X^{\theta})=\tilde{L}+\sum_{i=1}^6a_iC_i$, where $(a_i)_{1\leq i\leq 6}=(3,6,4,2,4,2)$.
\end{enumerate}
\end{proposition}
\begin{proof}
    \begin{enumerate}[label=(\arabic*)]
        \item It is easy to see that $\ttheta(C_i)=C_{\tau(i)}$ thanks to \Cref{Ciequation}. The intersection points $C_2\cap C_3,\ C_2\cap C_5$ are permuted by $\ttheta$. It follows from \ref{autop2} of \Cref{liftoncomponent} that there are exactly two $\ttheta$-fixed points on $C_2$. The intersection point $C_1\cap C_2$ is one of the $\ttheta$-fixed points on $C_2$, and the other one is an isolated point in $\pi^{-1}(0)^{\ttheta}$.
        \item This follows from \Cref{isolatedtransversal}. 
        \item Recall the definition of $b_i$ in \cref{bidef}. We have $(b_1,b_2,b_3,\cdots,b_6)=(0,1,0,\cdots,0)$ by \ref{E6case2temp2}. Since $X^{\theta}=\ddiv(2z+x^2)$ is a principal divisor, the $a_i$'s are determined by \Cref{multofcomponent}. 
    \end{enumerate}
\end{proof}

The fixed point locus $X^{\theta}$ and the preimage $\pi^{-1}(X^{\theta})$ are depicted in \Cref{E6case2pic}. The number next to each component indicates its multiplicity. Thickened lines indicate that the multiplicities are larger than $1$. The $\ttheta$-fixed components of $\pi^{-1}(X^{\theta})$ are colored in red.

\begin{mtable}[ht]
\begin{tabular}{ccc}
    \begin{tabular}{c}
    \begin{tikzpicture}
    \draw[thick] (0.5,-1.2) edge [bend right=20] node [below]{} (-0.3,0);
    \draw[thick] (0.5,1.2) edge [bend left=20] node [below]{} (-0.3,0);
    \end{tikzpicture}\end{tabular} & $\qquad$ &
    \begin{tabular}{c}\begin{tikzpicture}
    \draw[line width=0.4mm] (-3,2) edge [bend left=25] node [below]{} (-3,-1.5);
    \node at (-3,-1.7) () {$6$};
    \draw[line width=0.4mm] (-3.5,-1) edge [bend left=35] node [below]{} (-1,-1);
    \node at (-3.7,-1) () {$4$};
    \draw[line width=0.4mm] (-3.5,1) edge [bend left=35] node [below]{} (-1,1);
    \node at (-3.7,1) () {$4$};
    \draw[thick,red] (-3.5,0.6) to (-0.75,0.6);
    \node at (-3.7,0.6) () {{\color{red} $1$}};
  
    \draw[line width=0.4mm,red] (-3.5,-0.2) edge [bend left=35] node [below]{} (-1,-0.2);
    \node at (-3.7,-0.4) () {{\color{red} $3$}};
    
    \draw[line width=0.4mm] (-2,1) edge [bend left=35] node [below]{} (0.5,1);
    \node at (0.6,0.85) () {$2$};

    \draw[line width=0.4mm] (-2,-1) edge [bend left=35] node [below]{} (0.5,-1);
    \node at (0.6,-1.15) () {$2$};
    \end{tikzpicture}\end{tabular}\\
    $X^{\theta}$ & & $\pi^{-1}(X^{\theta})$
\end{tabular}
\caption{$E_6$, Type \RMN{2}}
\label{E6case2pic}
\end{mtable}

\section{Preimages of fixed point loci for type $E_7$ Kleinian singularity}\label{E7quiver}

\begin{mtable}[ht]
    \centering
\begin{tabular}{c}
    \scalebox{1.3}{\begin{tikzpicture}
        \node at (-3,0) (E7) {$\widetilde{E}_7$};
        
        \filldraw[black] (-1.5,1.5) circle (2pt) node[anchor=east]{$W_0$};
        \node[] at (-1.2,1.5)  {\color{blue} $1$};
        \filldraw[black] (3,-1.5) circle (2pt) node[anchor=east]{$V_7$};
        \node[] at (3.3,-1.5)  {\color{blue} $2$};
        \filldraw[black] (-1.5,0) circle (2pt) node[anchor=east]{$V_0$};
        \node[] at (-1.5,-0.3)  {\color{blue} $1$};
        \filldraw[black] (0,0) circle (2pt) node[anchor=south]{$V_1$};
        \node[] at (0,-0.3)  {\color{blue} $2$};
        \filldraw[black] (1.5,0) circle (2pt) node[anchor=south]{$V_2$};
        \node[] at (1.5,-0.3)  {\color{blue} $3$};
        \filldraw[black] (3,0) circle (2pt) node[anchor=south]{$V_3$};
        \node[] at (3,0.7)  {\color{blue} $4$};
        \filldraw[black] (4.5,0) circle (2pt) node[anchor=south]{$V_4$};
        \node[] at (4.5,-0.3)  {\color{blue} $3$};
        \filldraw[black] (6,0) circle (2pt) node[anchor=south]{$V_5$};
        \node[] at (6,-0.3)  {\color{blue} $2$};
        \filldraw[black] (7.5,0) circle (2pt) node[anchor=south]{$V_6$};
        \node[] at (7.5,-0.3)  {\color{blue} $1$};

        \draw[->,thick] (-1.55,1.4) -- (-1.55,0.1);
        \node[] at (-1.75,0.8)  {\scalebox{0.9}{$l_0$}};
        \draw[<-,thick] (-1.45,1.4) -- (-1.45,0.1);
        \node[] at (-1.25,0.8)  {\scalebox{0.9}{$k_0$}};

        \draw[->,thick] (3.05,-1.4) -- (3.05,-0.1);
        \node at (3.5,-0.85)  {\scalebox{0.9}{$B_{7\leftarrow3}$}};
        \draw[->,thick] (2.95,-0.1) -- (2.95,-1.4);
        \node[] at (2.5,-0.85)  {\scalebox{0.9}{$B^*_{3\leftarrow 7}$}};

        \draw[->,thick] (-1.4,0.05) -- (-0.1,0.05);
        \node at (-0.7,0.25)  {\scalebox{0.9}{$B_{1\leftarrow0}$}};
        \draw[<-,thick] (-1.4,-0.05) -- (-0.1,-0.05);
        \node[] at (-0.7,-0.25)  {\scalebox{0.9}{$B^*_{0\leftarrow 1}$}};
        \draw[->,thick] (0.1,0.05) -- (1.4,0.05);
        \node at (0.7,0.25)  {\scalebox{0.9}{$B_{2\leftarrow1}$}};
        \draw[<-,thick] (0.1,-0.05) -- (1.4,-0.05);
        \node[] at (0.7,-0.25)  {\scalebox{0.9}{$B^*_{1\leftarrow 2}$}};
        \draw[->,thick] (1.6,0.05) -- (2.9,0.05);
        \node at (2.2,0.25)  {\scalebox{0.9}{$B_{3\leftarrow2}$}};
        \draw[<-,thick] (1.6,-0.05) -- (2.9,-0.05);
        \node[] at (2.2,-0.25) {\scalebox{0.9}{$B^*_{2\leftarrow 3}$}};
        
        \draw[->,thick] (3.1,0.05) -- (4.4,0.05);
        \node at (3.7,0.25)  {\scalebox{0.9}{$B_{4\leftarrow3}$}};
        \draw[<-,thick,] (3.1,-0.05) -- (4.4,-0.05);
        \node[] at (3.7,-0.25)  {\scalebox{0.9}{$B^*_{3\leftarrow 4}$}};
        \draw[->,thick] (4.6,0.05) -- (5.9,0.05);
        \node at (5.2,0.25) {\scalebox{0.9}{$B_{5\leftarrow4}$}};
        \draw[<-,thick] (4.6,-0.05) -- (5.9,-0.05);
        \node[] at (5.2,-0.25)  {\scalebox{0.9}{$B^*_{4\leftarrow 5}$}};
        \draw[->,thick] (6.1,0.05) -- (7.4,0.05);
        \node at (6.7,0.25) {\scalebox{0.9}{$B_{6\leftarrow5}$}};
        \draw[<-,thick] (6.1,-0.05) -- (7.4,-0.05);
        \node[] at (6.7,-0.25)  {\scalebox{0.9}{$B^*_{5\leftarrow 6}$}};
    \end{tikzpicture}}
\end{tabular}
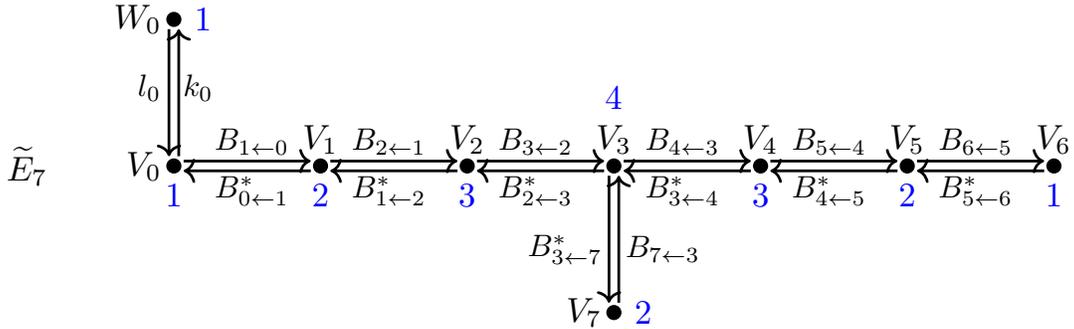
\caption{Extended Dynkin quiver $\widetilde{E}_7$}
\label{quiverE7}
\end{mtable}

We realize the Kleinian singularity of type $E_7: X=\Spec\C[x,y,z]/(x^3y+y^3+z^2)$ as a Nakajima quiver variety explicitly (\Cref{quiverE7}). Recall that $\delta=(\delta_i)_{0\leq i\leq 7}=(1,2,3,4,3,2,1,2)$. We have $\dim W_0=1,\ \dim V_i=\delta_i$. The number (colored in blue) next to each vector space indicates its dimension. 
The representation space is 
\[
M(\delta,w)=T^*\big(\bigoplus_{a\in\Omega}\Hom(V_{t(a)},V_{h(a)})\oplus\Hom(W_0,V_0)\big).
\]
The group $\GL(\delta)=\prod_{i=0}^7\GL(V_i)$ acts on $M(\delta,w)$ naturally with the moment map
\[
\mu=(\mu_i)_{0\leq i\leq 7}\colon M(\delta,w)\to\gl(\delta).
\]

According to \Cref{wframinginv} and \Cref{ADExyz}, the algebra of invariant functions $\C[\mu^{-1}(0)]^{\GL(\delta)}$ is generated by trace functions $x,y,z$ with $\deg x=8,\ \deg y=12,\ \deg z=18$.
Note that the lift of the unique anti-Poisson involution in type $E_7$ can be constructed without knowing the explicit forms of $x,y,z$, thanks to \Cref{halfdeglemma} and the observation on the gradings of $x,y,z$. The lift will be constructed in \Cref{E7lift}. 

Although the explicit forms of $x,y,z$ as trace functions are not needed for the main purpose of the paper, it is still practical to have an explicit and pictorial description of the functions on Kleinian singularities as trace functions of cycles on the extended Dynkin quiver. This will be discussed in \Cref{E7generatorprop}. 

\begin{lemma}\label[lemma]{halfdeglemma}
    Let $Q=(I,\Omega)$ be a quiver whose underlying graph has no cycles (for example, extended Dynkin quiver of types $D$ or $E$). Let $q$ be a cycle on the doubled quiver $\bar{Q}=(I,\bar{\Omega})$. Then $q$ contains $\frac{\deg q}{2}$ arrows from $\Omega$ and $\frac{\deg q}{2}$ arrows from $\Omega^*$.
\end{lemma}
\begin{proof}
    Omitted.
\end{proof}

\subsection{Lift of anti-Poisson involution}\label{E7liftsub}
Recall from \Cref{E7API} that the Kleinian singularity of type $E_7$ has a unique conjugacy class of anti-Poisson involutions, which is given by
\[
  \theta(x)=x,\ \theta(y)=y,\  \theta(z)=-z.
\]
The fixed point locus $X^{\theta}=\Spec\C[x,y]/(y(x^3+y^2))$ has two irreducible components $L_j,\ j=1,2$. One of them is $\A^1$ and the other is a cusp. Following the notation in \Cref{liftsectionq1}, we have $\pi^{-1}(X^{\theta})_{\red}=\pi^{-1}(0)_{\red}\cup\tilde{L}_1\cup\tilde{L}_2$. 

\begin{proposition}\label[proposition]{E7lift}
The involution $\ttheta\colon\M_{\chi}(\delta,w)\to\M_{\chi}(\delta,w)$ defined by
\begin{equation}\label{E7liftequation}
    \ttheta([B,B^*,l_0,k_0])=[-B,B^*,l_0,-k_0]
\end{equation}
is anti-symplectic. Moreover, it is a lift of $\theta\colon X\to X$.
\end{proposition}
\begin{proof}
We consider the following linear anti-symplectic involution on the representation space.
\begin{equation*}
    \Theta\colon M(\delta,w)\to M(\delta,w),\ \ (B,B^*,l_0,k_0)\mapsto(-B,B^*,l_0,-k_0).
\end{equation*}
It satisfies the conditions listed in \Cref{liftconditions}.
    \begin{enumerate}
    \item $g^{\Theta}=g$.
    \item Recall that $\deg x=8,\ \deg y=12,\ \deg z=18$. By \Cref{halfdeglemma}, we have $\Theta(x)=(-1)^4x=x,\ \Theta(y)=(-1)^6y=y,\ \Theta(z)=(-1)^9z=-z$.
\end{enumerate}
The involution $\ttheta$ defined in \cref{E7liftequation} is the descent of $\Theta$ to $\M_{\chi}(\delta,w)$. The rest follows from \Cref{liftconditions}. 
\end{proof}

\begin{proposition}\label[proposition]{E7preimage}
\begin{enumerate}[label=(\arabic*)]
   \item 
   The involution $\ttheta$ preserves all the components $C_i,\ 1\leq i\leq 7$. We have $\pi^{-1}(0)^{\ttheta}=C_1\cup C_3\cup C_5\cup\{p_1,p_2\}$ where $p_1\in C_6\setminus C_5,\ p_2\in C_7\setminus C_3$.
   \item\label{E7temp2} With a suitable choice of labeling, $\tilde{L}_1\cap\pi^{-1}(0)=\tilde{L}_1\cap C_6=\{p_1\},\ \tilde{L}_2\cap\pi^{-1}(0)=\tilde{L}_2\cap C_7=\{p_2\}$.
   \item The scheme $\pi^{-1}(X^{\theta})$ is not reduced. As a divisor, we have $\pi^{-1}(X^{\theta})=\tilde{L}_1+\tilde{L}_2+\sum_{i=1}^7a_iC_i$, where $(a_i)_{1\leq i\leq 7}=(3,6,9,7,5,3,5)$.
\end{enumerate}
\end{proposition}
\begin{proof}
    \begin{enumerate}[label=(\arabic*)]
        \item It is easy to see that $\ttheta(C_i)=C_{i}$ thanks to \Cref{Ciequation}. The component $C_3$ intersects three other components $C_2,\ C_4,\ C_7$, so there are at least three $\ttheta$-fixed points on $C_3$.
        It follows from \ref{autop2} of \Cref{liftoncomponent} that $\ttheta$ acts trivially on $C_3$. By \ref{autop3} of \Cref{liftoncomponent}, $\ttheta$ has exactly two fixed points on each of the components $C_2,C_4,C_6,C_7$, and $\ttheta$ acts trivially on the components $C_1,C_5$. Therefore, there are two isolated points in $\pi^{-1}(0)^{\ttheta}$. One of them lies in $C_6\setminus C_5$, and the other lies in $C_7\setminus C_3$.
        \item This follows from \Cref{isolatedtransversal}. 
        \item Recall the definition of $b_i$ from \cref{bidef}. We have $(b_1,\cdots,b_5,b_6,b_7)=(0,\cdots,0,1,1)$ from \ref{E7temp2}. Since $X^{\theta}=\ddiv(z)$ is a principal divisor, $a_i$'s can be determined by \Cref{multofcomponent}. 
    \end{enumerate}
\end{proof}

The fixed point locus $X^{\theta}$ and the preimage $\pi^{-1}(X^{\theta})$ are depicted in \Cref{E7pic}. The number next to each component indicates its multiplicity. Thickened lines indicate that the multiplicities are larger than $1$. The $\ttheta$-fixed components of $\pi^{-1}(X^{\theta})$ are colored in red.

\begin{mtable}[ht]
\begin{tabular}{ccc}
    \begin{tabular}{c}
    \begin{tikzpicture}
    \draw[thick] (0.5,-1.2) edge [bend right=20] node [below]{} (-0.3,0);
    \draw[thick] (0.5,1.2) edge [bend left=20] node [below]{} (-0.3,0);
    \draw[thick] (-0.3,1.2) to (-0.3,-1.2);
    \end{tikzpicture}\end{tabular} & $\qquad$ &
    \begin{tabular}{c}\begin{tikzpicture}
    \draw[line width=0.4mm,red] (-3,2) edge [bend left=25] node [below]{} (-3,-1.5);
    \node at (-3,-1.7) () {\textcolor{red}{$a_3=9$}};
    \draw[line width=0.4mm] (-3.5,-1) edge [bend left=35] node [below]{} (-1,-1);
    \node at (-3.9,-0.8) () {$a_4=7$};
    \draw[line width=0.4mm] (-3.5,1) edge [bend left=35] node [below]{} (-1,1);
    \node at (-3.9,1.2) () {$a_2=6$};
    \draw[line width=0.4mm] (-3.5,0) edge [bend left=35] node [below]{} (-1,0);
    \node at (-3.9,0.2) () {$a_7=5$};
    
    \draw[line width=0.4mm,red] (-2,1) edge [bend left=35] node [below]{} (0.5,1);
    \node at (0.6,0.85) () {\textcolor{red}{$a_1=3$}};
    \draw[line width=0.4mm,red] (-2,-1) edge [bend left=35] node [below]{} (0.5,-1);
    \node at (0.6,-1.15) () {\textcolor{red}{$a_5=5$}};
    \draw[line width=0.4mm] (-0.5,-1) edge [bend left=35] node [below]{} (2,-1);
    \node at (2.1,-1.15) () {$a_6=3$};

    \draw[thick,red] (-1.1,0.85) to (-1.7,-0.5);
    \node at (-1,0.7) () {{\color{red} $1$}};
    \draw[thick,red] (2,-0.15) to (1.3,-1.6);
    \node at (1.3,-1.8) () {{\color{red} $1$}};
\end{tikzpicture}\end{tabular}\\
$X^{\theta}$ & & $\pi^{-1}(X^{\theta})$
\end{tabular}
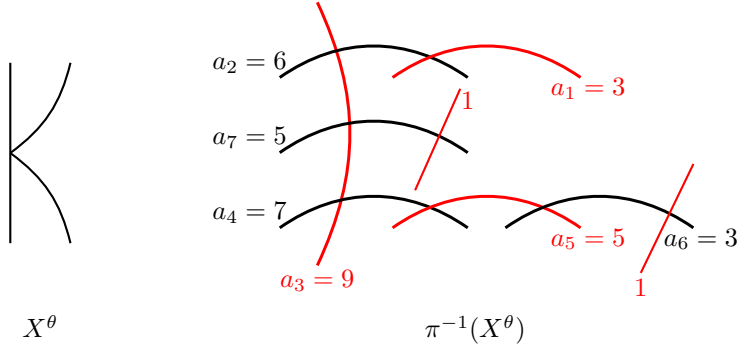
\caption{$E_7$, fixed point locus and its preimage}
\label{E7pic}
\end{mtable}

\subsection{Explicit generators $x,y,z$ for type $E_7$}\label{E7generatorsub}

In this section, we discuss how to find the explicit forms of the generators $x,y,z$ of $\C[\mu^{-1}(0)]^{\GL(\delta)}$ as trace functions.

\begin{proposition}\label[proposition]{E7generatorprop}
We have $\C[\mu^{-1}(0)]^{\GL(\delta)}\simeq\C[x,y,z]/(x^3y+y^3+z^2)$, where
\begin{equation}\label{E7xyz}
    \begin{aligned}
    x:=&B^*_{0\la 1}B^*_{1\la 2}B^*_{2\la 3}B^*_{3\la 4}B_{4\la 3}B_{3\la 2}B_{2\la 1}B_{1\la 0}, \\
    y:=&B^*_{0\la 1}B^*_{1\la 2}B^*_{2\la 3}(B^*_{3\la 4}B_{4\la 3})^3B_{3\la 2}B_{2\la 1}B_{1\la 0}, \\
    z:=&B^*_{0\la 1}B^*_{1\la 2}B^*_{2\la 3}(B^*_{3\la 4}B_{4\la 3})^2B^*_{3\la 7}B_{7\la 3}(B^*_{3\la 4}B_{4\la 3})^3B_{3\la 2}B_{2\la 1}B_{1\la 0}.
\end{aligned}
\end{equation}
\end{proposition}
We give a proof that is different from those in types $D_n$ and $E_6$ (cf. \Cref{Dngeneratorprop,E6generatorprop}). The reason is that there are too few one-dimensional vector spaces among the vector spaces $V_i,\ 0\leq i\leq 7$ in type $E_7$, making it very hard to show that $x,y,z$ generate $\C[\mu^{-1}(0)]^{\GL(\delta)}$ and to find the relation among $x,y,z$ by manipulating the ADHM equation in \cref{ADEmomentmap0}. \Cref{E7lemma} will be needed in the proof of \Cref{E7generatorprop}. Consider the following two points $(B^i,B^{*,i},1,0)\in M(\delta,w),\ i=1,2$, given in \Cref{E7twopoints}. 
Let $x_i,y_i,z_i$ denote the values of the functions $x,y,z$ at the points $(B^i,B^{*,i},1,0),\ i=1,2$.

{
\captionsetup[table]{name=Table}
\begin{table}[ht]
\begin{tabular}{c|c}
    $(B^1,B^{*,1},1,0)$ & $(B^2,B^{*,2},1,0)$ \\
    \hline
    {$\!\begin{aligned}
        B^1_{1\la 0}&=\begin{pmatrix}
            0 \\ 1
        \end{pmatrix},\ B^{*,1}_{0\la 1}=\begin{pmatrix}
           -8 & 0
        \end{pmatrix} \\
        \ B^1_{2\la 1}&=\begin{pmatrix}
            1 & 0 \\ 
            0 & 0 \\
            0 & 1
        \end{pmatrix},\ B^{*,1}_{1\la 2}=\begin{pmatrix}
            0 & 1 & 0 \\ 
           -8 & 0 & 0 
        \end{pmatrix} \\
        \ B^1_{3\la 2}&=\begin{pmatrix}
            2 & 0 & 0 \\ 
           -2 & 1 & 0 \\
            4 & 0 & 0 \\
           -4 & 0 & 1
        \end{pmatrix},\ B^{*,1}_{2\la 3}=\begin{pmatrix}
            1 & 1 & 0 & 0 \\ 
           -2 & 0 & 1 & 0 \\
           -4 & 0 & 0 & 0
        \end{pmatrix} \\
        B^1_{4\la 3}&=\begin{pmatrix}
            0 & 1 & 0 & 0 \\ 
            0 & 0 & 1 & 0 \\
            0 & 0 & 0 & 1 
        \end{pmatrix},\ B^{*,1}_{3\la 4}=\begin{pmatrix}
            1 & 0 & 0 \\ 
            0 & 1 & 0 \\
            0 & 0 & 1 \\
            0 & 0 & 0 
        \end{pmatrix} \\
        B^1_{5\la 4}&=\begin{pmatrix}
            0 & 1 & 0 \\
            0 & 0 & 1
        \end{pmatrix},\ B^{*,1}_{4\la 5}=\begin{pmatrix}
            1 & 0 \\
            0 & 1 \\
            0 & 0
        \end{pmatrix} \\
        B^1_{6\la 5}&=\begin{pmatrix}
            0 & 1 
        \end{pmatrix},\ B^{*,1}_{5\la 6}=\begin{pmatrix}
            1 \\
            0
        \end{pmatrix} \\
        B^1_{7\la 3}&=\begin{pmatrix}
            2 & 1 & 0 & 0 \\
            4 & 0 & 0 & 1
        \end{pmatrix},\ B^{*,1}_{3\la 7}=\begin{pmatrix}
            1 & 0 \\
           -2 & 0 \\
            4 & -1 \\ 
           -4 & 0
        \end{pmatrix}
    \end{aligned}$} & {$\!\begin{aligned}
        B^2_{1\la 0}&=\begin{pmatrix}
            0 \\ 1
        \end{pmatrix},\ B^{*,2}_{0\la 1}=\begin{pmatrix}
            4 & 0
        \end{pmatrix} \\
        \ B^2_{2\la 1}&=\begin{pmatrix}
            1 & 0 \\ 
            0 & 0 \\
            0 & 1
        \end{pmatrix},\ B^{*,2}_{1\la 2}=\begin{pmatrix}
            0 & 1 & 0 \\ 
            4 & 0 & 0 
        \end{pmatrix} \\
        \ B^2_{3\la 2}&=\begin{pmatrix}
           -1 & 0 & 0 \\ 
           -2 & 1 & 0 \\
           -2 & 0 & 0 \\
           -4 & 0 & 1
        \end{pmatrix},\ B^{*,2}_{2\la 3}=\begin{pmatrix}
           -2 & 1 & 0 & 0 \\ 
           -2 & 0 & 1 & 0 \\
           -4 & 0 & 0 & 0
        \end{pmatrix} \\
        B^2_{4\la 3}&=\begin{pmatrix}
            0 & 1 & 0 & 0 \\ 
            0 & 0 & 1 & 0 \\
            0 & 0 & 0 & 1 
        \end{pmatrix},\ B^{*,2}_{3\la 4}=\begin{pmatrix}
            1 & 0 & 0 \\ 
            0 & 1 & 0 \\
            0 & 0 & 1 \\
            0 & 0 & 0 
        \end{pmatrix} \\
        B^2_{5\la 4}&=\begin{pmatrix}
            0 & 1 & 0 \\
            0 & 0 & 1
        \end{pmatrix},\ B^{*,2}_{4\la 5}=\begin{pmatrix}
            1 & 0 \\
            0 & 1 \\
            0 & 0
        \end{pmatrix} \\
        B^2_{6\la 5}&=\begin{pmatrix}
            0 & 1 
        \end{pmatrix},\ B^{*,2}_{5\la 6}=\begin{pmatrix}
            1 \\
            0
        \end{pmatrix} \\
        B^2_{7\la 3}&=\begin{pmatrix}
           -1 & 1 & 0 & 0 \\
           -2 & 0 & 0 & 1
        \end{pmatrix},\ B^{*,2}_{3\la 7}=\begin{pmatrix}
           -2 & 0 \\
           -2 & 0 \\
           -2 & -1 \\ 
           -4 & 0
        \end{pmatrix}
    \end{aligned}$}
\end{tabular}
\caption{Two generic points in $E_7$}
\label{E7twopoints}
\end{table}
}

\begin{lemma}\label[lemma]{E7lemma}
   \begin{enumerate}
       \item $(B^i,B^{*,i},1,0)\in\mu^{-1}(0),\ i=1,2$.
       \item $x_1=-8,\ y_1=16,\ z_1=64$.
       \item $x_2=4,\ y_2=-8,\ z_2=-32$.
   \end{enumerate}
\end{lemma}
\begin{proof}
    The result follows by direct matrix multiplications, using the definitions of the moment maps in \cref{ADEmomentmap0} and the functions $x,y,z$ in \Cref{E7generatorprop}.
\end{proof}

\begin{proof}[Proof of \Cref{E7generatorprop}]
    Recall from \Cref{wframinginv} and \Cref{ADExyz} that the algebra of invariant functions $\C[\mu^{-1}(0)]^{\GL(\delta)}$ is generated by trace functions of degrees $8,12,18$. Note that none of the degrees can be written as a positive integral linear combination of the others. It follows that
    \[
    \dim\C[\mu^{-1}(0)]^{\GL(\delta)}_8=\dim\C[\mu^{-1}(0)]^{\GL(\delta)}_{12}=\dim\C[\mu^{-1}(0)]^{\GL(\delta)}_{18}=1.
    \]
    We know that $x,y,z$ are all nonzero functions in $\C[\mu^{-1}(0)]^{\GL(\delta)}$ from \Cref{E7lemma}, and that $\deg x=8,\deg y=12,\deg z=18$ from their definitions in \cref{E7xyz}. Therefore,
    \[
    \C[\mu^{-1}(0)]^{\GL(\delta)}_8=\C x,\ \C[\mu^{-1}(0)]^{\GL(\delta)}_{12}=\C y,\ \C[\mu^{-1}(0)]^{\GL(\delta)}_{18}=\C z.
    \]
    It follows that $\C[\mu^{-1}(0)]^{\GL(\delta)}$ is generated by $x,y,z$. Again by \Cref{wframinginv} and \Cref{ADExyz}, there must be a unique (up to scaling) nontrivial relation $F(x,y,z)=0$ of degree $36$ among $x,y,z$. The only monomials of degree $36$ are $x^3y,y^3,z^2$; 
    thus we can write
    \begin{equation}\label{E7abc}        
    F(x,y,z)=a x^3y+b y^3+c z^2=0,\ a,b,c\in \C,\ \text{not all zero}.
    \end{equation}
    Plugging the points $(x_i,y_i,z_i),\ i=1,2$ from \Cref{E7lemma} into \cref{E7abc}, we get 
    \begin{align*}
         F(x_1,y_1,z_1)=2^{12}(-2a+b+c)=0, \\
         F(x_2,y_2,z_2)=2^9(-a-b+2c)=0.
    \end{align*}
    It follows that $a=b=c$. 
    We may take $a=b=c=1$. Then we get a surjective homomorphism
    \begin{equation}\label{E7surj}
    \C[x,y,z]/(x^3y+y^3+z^2)\twoheadrightarrow\C[\mu^{-1}(0)]^{\GL(\delta)},
    \end{equation}
    which turns out to be an isomorphism because both sides are integral domains of dimension $2$.
\end{proof}

\section{Preimages of fixed point loci for type $E_8$ Kleinian singularity}\label{E8quiver}

\begin{mtable}[ht]
    \centering
\begin{tabular}{c}
    \scalebox{1.3}{\begin{tikzpicture}
        \node at (-3,0) (E8) {$\widetilde{E}_8$};

        \filldraw[black] (-1.5,1.5) circle (2pt) node[anchor=east]{$W_0$};
        \node[] at (-1.2,1.5)  {\color{blue} $1$};
        \filldraw[black] (6,-1.5) circle (2pt) node[anchor=east]{$V_8$};
        \node[] at (6.3,-1.5)  {\color{blue} $2$};
        \filldraw[black] (-1.5,0) circle (2pt) node[anchor=east]{$V_0$};
        \node[] at (-1.5,-0.3)  {\color{blue} $1$};
        \filldraw[black] (0,0) circle (2pt) node[anchor=south]{$V_1$};
        \node[] at (0,-0.3)  {\color{blue} $2$};
        \filldraw[black] (1.5,0) circle (2pt) node[anchor=south]{$V_2$};
        \node[] at (1.5,-0.3)  {\color{blue} $3$};
        \filldraw[black] (3,0) circle (2pt) node[anchor=south]{$V_3$};
        \node[] at (3,-0.3)  {\color{blue} $4$};
        \filldraw[black] (4.5,0) circle (2pt) node[anchor=south]{$V_4$};
        \node[] at (4.5,-0.3)  {\color{blue} $5$};
        \filldraw[black] (6,0) circle (2pt) node[anchor=south]{$V_5$};
        \node[] at (6,0.7)  {\color{blue} $6$};
        \filldraw[black] (7.5,0) circle (2pt) node[anchor=south]{$V_6$};
        \node[] at (7.5,-0.3)  {\color{blue} $4$};
        \filldraw[black] (9,0) circle (2pt) node[anchor=south]{$V_7$};
        \node[] at (9,-0.3)  {\color{blue} $2$};

        \draw[->,thick] (-1.55,1.4) -- (-1.55,0.1);
        \node[] at (-1.75,0.8)  {\scalebox{0.9}{$l_0$}};
        \draw[<-,thick] (-1.45,1.4) -- (-1.45,0.1);
        \node[] at (-1.25,0.8)  {\scalebox{0.9}{$k_0$}};

        \draw[->,thick] (6.05,-1.4) -- (6.05,-0.1);
        \node at (6.5,-0.85)  {\scalebox{0.9}{$B_{8\leftarrow5}$}};
        \draw[->,thick] (5.95,-0.1) -- (5.95,-1.4);
        \node[] at (5.5,-0.85)  {\scalebox{0.9}{$B^*_{5\leftarrow 8}$}};

        \draw[->,thick] (-1.4,0.05) -- (-0.1,0.05);
        \node at (-0.7,0.25)  {\scalebox{0.9}{$B_{1\leftarrow0}$}};
        \draw[<-,thick] (-1.4,-0.05) -- (-0.1,-0.05);
        \node[] at (-0.7,-0.25)  {\scalebox{0.9}{$B^*_{0\leftarrow 1}$}};
        \draw[->,thick] (0.1,0.05) -- (1.4,0.05);
        \node at (0.7,0.25)  {\scalebox{0.9}{$B_{2\leftarrow1}$}};
        \draw[<-,thick] (0.1,-0.05) -- (1.4,-0.05);
        \node[] at (0.7,-0.25)  {\scalebox{0.9}{$B^*_{1\leftarrow 2}$}};
        \draw[->,thick] (1.6,0.05) -- (2.9,0.05);
        \node at (2.2,0.25)  {\scalebox{0.9}{$B_{3\leftarrow2}$}};
        \draw[<-,thick] (1.6,-0.05) -- (2.9,-0.05);
        \node[] at (2.2,-0.25) {\scalebox{0.9}{$B^*_{2\leftarrow 3}$}};
        
        \draw[->,thick] (3.1,0.05) -- (4.4,0.05);
        \node at (3.7,0.25)  {\scalebox{0.9}{$B_{4\leftarrow3}$}};
        \draw[<-,thick,] (3.1,-0.05) -- (4.4,-0.05);
        \node[] at (3.7,-0.25)  {\scalebox{0.9}{$B^*_{3\leftarrow 4}$}};
        \draw[->,thick] (4.6,0.05) -- (5.9,0.05);
        \node at (5.2,0.25) {\scalebox{0.9}{$B_{5\leftarrow4}$}};
        \draw[<-,thick] (4.6,-0.05) -- (5.9,-0.05);
        \node[] at (5.2,-0.25)  {\scalebox{0.9}{$B^*_{4\leftarrow 5}$}};
        \draw[->,thick] (6.1,0.05) -- (7.4,0.05);
        \node at (6.7,0.25) {\scalebox{0.9}{$B_{6\leftarrow5}$}};
        \draw[<-,thick] (6.1,-0.05) -- (7.4,-0.05);
        \node[] at (6.7,-0.25)  {\scalebox{0.9}{$B^*_{5\leftarrow 6}$}};
        \draw[->,thick] (7.6,0.05) -- (8.9,0.05);
        \node at (8.2,0.25) {\scalebox{0.9}{$B_{7\leftarrow6}$}};
        \draw[<-,thick] (7.6,-0.05) -- (8.9,-0.05);
        \node[] at (8.2,-0.25)  {\scalebox{0.9}{$B^*_{6\leftarrow 7}$}};
    \end{tikzpicture}}
\end{tabular}
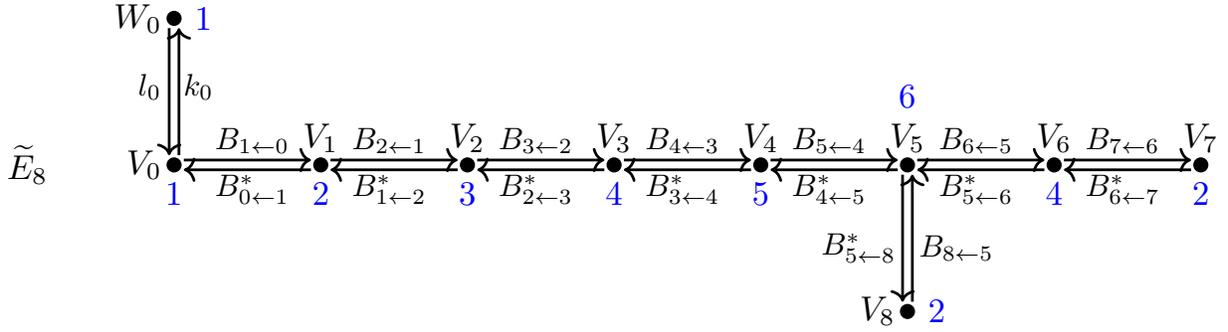
\caption{Extended Dynkin quiver $\widetilde{E}_8$}
\label{quiverE8}
\end{mtable}

We realize the Kleinian singularity of type $E_8: X=\Spec\C[x,y,z]/(x^5+y^3+z^2)$ as a Nakajima quiver variety explicitly (\Cref{quiverE8}). Recall that $\delta=(\delta_i)_{0\leq i\leq 8}=(1,2,3,4,5,6,4,2,3)$. We have $\dim W_0=1,\ \dim V_i=\delta_i$. The number (colored in blue) next to each vector space indicates its dimension. The representation space is 
\[
M(\delta,w)=T^*\big(\bigoplus_{a\in\Omega}\Hom(V_{t(a)},V_{h(a)})\oplus\Hom(W_0,V_0)\big).
\]
The group $\GL(\delta)=\prod_{i=0}^8\GL(V_i)$ acts on $M(\delta,w)$ naturally with the moment map 
\[
\mu=(\mu_i)_{0\leq i\leq 8}\colon M(\delta,w)\to\gl(\delta).
\]
According to \Cref{wframinginv} and \Cref{ADExyz}, the algebra of invariant functions $\C[\mu^{-1}(0)]^{\GL(\delta)}$ is generated by trace functions $x,y,z$ with $\deg x=12,\ \deg y=20,\ \deg z=30$. As in type $E_7$, we can construct the lift of the unique anti-Poisson involution in type $E_8$ without knowing the explicit forms of $x,y,z$ (cf. \Cref{E8lift}). The explicit forms of the trace functions $x,y,z$ will be discussed in \Cref{E8generatorprop}.

\subsection{Lift of anti-Poisson involution}\label{E8liftsub}
Recall from \Cref{E8API} that the Kleinian singularity of type $E_8$ has a unique conjugacy class of anti-Poisson involutions, which is given by
\[
  \theta(x)=x,\ \theta(y)=y,\ \theta(z)=-z.
\]
The fixed point locus $X^{\theta}=\Spec\C[x,y]/(x^5+y^3)$ is a cusp (a single irreducible component $L$). Following the notation in \Cref{liftsectionq1}, we have $\pi^{-1}(X^{\theta})_{\red}=\pi^{-1}(0)_{\red}\cup\tilde{L}$. The proofs of \Cref{E8lift,E8preimage} repeat those in type $E_7$. We leave them to the reader.

\begin{proposition}\label[proposition]{E8lift}
The involution $\ttheta\colon\M_{\chi}(\delta,w)\to\M_{\chi}(\delta,w)$ defined by
\begin{equation}\label{E8liftequation}
    \ttheta([B,B^*,l_0,k_0])=[-B,B^*,l_0,-k_0]
\end{equation}
is anti-symplectic. Moreover, it is a lift of $\theta\colon X\to X$.
\end{proposition}

\begin{proposition}\label[proposition]{E8preimage}
\begin{enumerate}[label=(\arabic*)]
   \item 
   The involution $\ttheta$ preserves all the components $C_i,\ 1\leq i\leq 8$. We have $\pi^{-1}(0)^{\ttheta}=C_1\cup C_3\cup C_5\cup C_7\cup\{p\}$ where $p\in C_8\setminus C_5$.
   \item\label{E8temp2} We have $\tilde{L}\cap\pi^{-1}(0)=\tilde{L}\cap C_8=\{p\}$.
   \item The scheme $\pi^{-1}(X^{\theta})$ is not reduced. As a divisor, we have $\pi^{-1}(X^{\theta})=\tilde{L}+\sum_{i=1}^8a_iC_i$, where $(a_i)_{1\leq i\leq 8}=(3,6,9,12,15,10,5,8)$.
\end{enumerate}
\end{proposition}

The fixed point locus $X^{\theta}$ and the preimage $\pi^{-1}(X^{\theta})$ are depicted in \Cref{E8pic}. The number next to each component indicates its multiplicity. Thickened lines indicate that the multiplicities are larger than $1$. The $\ttheta$-fixed components of $\pi^{-1}(X^{\theta})$ are colored in red.

\begin{mtable}[ht]
\begin{tabular}{ccc}
    \begin{tabular}{c}
    \begin{tikzpicture}
    \draw[thick] (0.5,-1.2) edge [bend right=20] node [below]{} (-0.3,0);
    \draw[thick] (0.5,1.2) edge [bend left=20] node [below]{} (-0.3,0);
    \end{tikzpicture}\end{tabular} & $\qquad$ &
    \begin{tabular}{c}\begin{tikzpicture}
    \draw[line width=0.4mm,red] (-3,2) edge [bend left=25] node [below]{} (-3,-1.5);
    \node at (-3,-1.7) () {\textcolor{red}{$a_5=15$}};
    \draw[line width=0.4mm] (-3.5,-1) edge [bend left=35] node [below]{} (-1,-1);
    \node at (-3.9,-0.75) () {$a_4=12$};
    \draw[line width=0.4mm] (-3.5,1) edge [bend left=35] node [below]{} (-1,1);
    \node at (-3.9,1.25) () {$a_6=10$};
    \draw[line width=0.4mm] (-3.5,0) edge [bend left=35] node [below]{} (-1,0);
    \node at (-3.9,0.2) () {$a_8=8$};
    
    \draw[line width=0.4mm,red] (-2,1) edge [bend left=35] node [below]{} (0.5,1);
    \node at (0.6,0.85) () {\textcolor{red}{$a_7=5$}};
    \draw[line width=0.4mm,red] (-2,-1) edge [bend left=35] node [below]{} (0.5,-1);
    \node at (-2,-1.15) () {\textcolor{red}{$a_3=9$}};
    \draw[line width=0.4mm] (-0.5,-1) edge [bend left=35] node [below]{} (2,-1);
    \node at (2.1,-1.15) () {$a_2=6$};
    \draw[line width=0.4mm,red] (1,-1) edge [bend left=35] node [below]{} (3.5,-1);
    \node at (3.6,-1.15) () {{\color{red} $a_1=3$}};

    \draw[thick,red] (-1.1,0.85) to (-1.7,-0.5);
    \node at (-1,0.7) () {{\color{red} $1$}};
\end{tikzpicture}\end{tabular}\\
$X^{\theta}$ & & $\pi^{-1}(X^{\theta})$
\end{tabular}
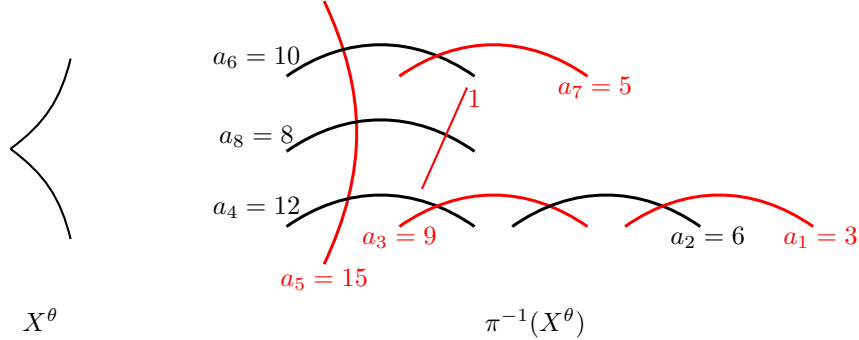
\caption{$E_8$, fixed point locus and its preimage}
\label{E8pic}
\end{mtable}

\subsection{Explicit generators $x,y,z$ for type $E_8$}
In this section, we discuss how to find the explicit forms of the generators $x,y,z$ of $\C[\mu^{-1}(0)]^{\GL(\delta)}$ as trace functions.
\begin{proposition}\label[proposition]{E8generatorprop}
We have $\C[\mu^{-1}(0)]^{\GL(\delta)}\simeq\C[x,y,z]/(x^5+y^3+z^2)$, where
\begin{equation}\label{E8xyz}
    \begin{aligned}
    x:=&B^*_{0\la 1}B^*_{1\la 2}B^*_{2\la 3}B^*_{3\la 4}B^*_{4\la 5}B^*_{5\la 8}B_{8\la 5}B_{5\la 4}B_{4\la 3}B_{3\la 2}B_{2\la 1}B_{1\la 0}, \\
    y:=&B^*_{0\la 1}B^*_{1\la 2}B^*_{2\la 3}B^*_{3\la 4}B^*_{4\la 5}(B^*_{5\la 6}B_{6\la 5})^2B^*_{5\la 8}B_{8\la 5}(B^*_{5\la 6}B_{6\la 5})^2B_{5\la 4}B_{4\la 3}B_{3\la 2}B_{2\la 1}B_{1\la 0}, \\
    z:=&B^*_{0\la 1}B^*_{1\la 2}B^*_{2\la 3}B^*_{3\la 4}B^*_{4\la 5}(B^*_{5\la 6}B_{6\la 5})^2(B_{5\la 4}B^*_{4\la 5})^3 \\
    &\cdot (B^*_{5\la 6}B_{6\la 5})^2B^*_{5\la 8}B_{8\la 5}(B^*_{5\la 6}B_{6\la 5})^2B_{5\la 4}B_{4\la 3}B_{3\la 2}B_{2\la 1}B_{1\la 0}.
\end{aligned}
\end{equation}
\end{proposition}
The proof of \Cref{E8generatorprop} is  identical to that of \Cref{E7generatorprop}, with the help of \Cref{E8lemma}. Consider the following two points $(B^i,B^{*,i},1,0)\in M(\delta,w),\ i=1,2$, given in \Cref{E8twopoints}. Let $x_i,y_i,z_i$ denote the values of the functions $x,y,z$ at the points $(B^i,B^{*,i},1,0),\ i=1,2$.

{
\captionsetup[table]{name=Table}
\begin{table}[ht]
\begin{tabular}{c|c}
    $(B^1,B^{*,1},1,0)$ & $(B^2,B^{*,2},1,0)$ \\
    \hline
    \scalebox{0.8}{
    {$\!\begin{aligned}
        B^1_{1\la 0}&=\begin{pmatrix}
            0 \\ 1
        \end{pmatrix},\ B^{*,1}_{0\la 1}=\begin{pmatrix}
           32 & 0
        \end{pmatrix} \\
        \ B^1_{2\la 1}&=\begin{pmatrix}
            1 & 0 \\ 
            0 & 0 \\
            0 & 1
        \end{pmatrix},\ B^{*,1}_{1\la 2}=\begin{pmatrix}
            0 & 1 & 0 \\ 
           32 & 0 & 0 
        \end{pmatrix} \\
        \ B^1_{3\la 2}&=\begin{pmatrix}
            1 & 0 & 0 \\ 
            0 & 1 & 0 \\
            0 & 0 & 0 \\
            0 & 0 & 1
        \end{pmatrix},\ B^{*,1}_{2\la 3}=\begin{pmatrix}
            0 & 1 & 0 & 0 \\ 
            0 & 0 & 1 & 0 \\
           32 & 0 & 0 & 0
        \end{pmatrix} \\
        B^1_{4\la 3}&=\begin{pmatrix}
            1 & 0 & 0 & 0 \\ 
            0 & 1 & 0 & 0 \\
            0 & 0 & 1 & 0 \\
            0 & 0 & 0 & 0 \\
            0 & 0 & 0 & 1
        \end{pmatrix},\ B^{*,1}_{3\la 4}=\begin{pmatrix}
            0 & 1 & 0 & 0 & 0 \\ 
            0 & 0 & 1 & 0 & 0\\
            0 & 0 & 0 & 1 & 0 \\
           32 & 0 & 0 & 0 & 0 \\
        \end{pmatrix} \\
        B^1_{5\la 4}&=\begin{pmatrix}
            2 & 0 & 0 & 0 & 0 \\ 
           -2 & 1 & 0 & 0 & 0 \\
            4 & 0 & 1 & 0 & 0 \\ 
            8 & 0 & 0 & 1 & 0 \\
          -16 & 0 & 0 & 0 & 0 \\
           16 & 0 & 0 & 0 & 1
        \end{pmatrix},\ B^{*,1}_{4\la 5}=\begin{pmatrix}
            1 & 1 & 0 & 0 & 0 & 0 \\ 
           -2 & 0 & 1 & 0 & 0 & 0 \\
           -4 & 0 & 0 & 1 & 0 & 0 \\
            8 & 0 & 0 & 0 & 1 & 0 \\
           16 & 0 & 0 & 0 & 0 & 0
        \end{pmatrix} \\
        B^1_{6\la 5}&=\begin{pmatrix}
            0 & 1 & 0 & 0 & 0 & 0 \\ 
            0 & 0 & 1 & 0 & 0 & 0 \\
            0 & 0 & 0 & 0 & 1 & 0 \\
            0 & 0 & 0 & 0 & 0 & 1
        \end{pmatrix},\ B^{*,1}_{5\la 6}=\begin{pmatrix}
            1 & 0 & 0 & 0 \\ 
            0 & 1 & 0 & 0 \\
            0 & 0 & 0 & 0 \\
            0 & 0 & 1 & 0 \\
            0 & 0 & 0 & 1 \\
            0 & 0 & 0 & 0
        \end{pmatrix} \\
        B^1_{7\la 6}&=\begin{pmatrix}
            0 & 1 & 0 & 0 \\
            0 & 0 & 0 & 1
        \end{pmatrix},\ B^{*,1}_{6\la 7}=\begin{pmatrix}
            1 & 0 \\
            0 & 0 \\
            0 & 1 \\
            0 & 0
        \end{pmatrix} \\
        B^1_{8\la 5}&=\begin{pmatrix}
            2 & 1 & 0 & 0 & 0 & 0 \\
           -8 & 0 & 0 & 1 & 0 & 0 \\
          -16 & 0 & 0 & 0 & 0 & 1
        \end{pmatrix},\ B^{*,1}_{5\la 8}=\begin{pmatrix}
            1 & 0 & 0 \\
           -2 & 0 & 0 \\
            4 & 1 & 0 \\ 
            8 & 0 & 0 \\
          -16 & 0 & -1 \\
           16 & 0 & 0
        \end{pmatrix}
    \end{aligned}$}} & \scalebox{0.8}{{$\!\begin{aligned}
        B^2_{1\la 0}&=\begin{pmatrix}
            0 \\ 1
        \end{pmatrix},\ B^{*,2}_{0\la 1}=\begin{pmatrix}
           8 & 0
        \end{pmatrix} \\
        \ B^2_{2\la 1}&=\begin{pmatrix}
            1 & 0 \\ 
            0 & 0 \\
            0 & 1
        \end{pmatrix},\ B^{*,2}_{1\la 2}=\begin{pmatrix}
            0 & 1 & 0 \\ 
            8 & 0 & 0 
        \end{pmatrix} \\
        \ B^2_{3\la 2}&=\begin{pmatrix}
            1 & 0 & 0 \\ 
            0 & 1 & 0 \\
            0 & 0 & 0 \\
            0 & 0 & 1
        \end{pmatrix},\ B^{*,2}_{2\la 3}=\begin{pmatrix}
            0 & 1 & 0 & 0 \\ 
            0 & 0 & 1 & 0 \\
            8 & 0 & 0 & 0
        \end{pmatrix} \\
        B^2_{4\la 3}&=\begin{pmatrix}
            1 & 0 & 0 & 0 \\ 
            0 & 1 & 0 & 0 \\
            0 & 0 & 1 & 0 \\
            0 & 0 & 0 & 0 \\
            0 & 0 & 0 & 1
        \end{pmatrix},\ B^{*,2}_{3\la 4}=\begin{pmatrix}
            0 & 1 & 0 & 0 & 0 \\ 
            0 & 0 & 1 & 0 & 0\\
            0 & 0 & 0 & 1 & 0 \\
            8 & 0 & 0 & 0 & 0 \\
        \end{pmatrix} \\
        B^2_{5\la 4}&=\begin{pmatrix}
           -1 & 0 & 0 & 0 & 0 \\ 
           -2 & 1 & 0 & 0 & 0 \\
           -2 & 0 & 1 & 0 & 0 \\ 
           -4 & 0 & 0 & 1 & 0 \\
           -4 & 0 & 0 & 0 & 0 \\
           -8 & 0 & 0 & 0 & 1
        \end{pmatrix},\ B^{*,2}_{4\la 5}=\begin{pmatrix}
           -2 & 1 & 0 & 0 & 0 & 0 \\ 
           -2 & 0 & 1 & 0 & 0 & 0 \\
           -4 & 0 & 0 & 1 & 0 & 0 \\
           -4 & 0 & 0 & 0 & 1 & 0 \\
           -8 & 0 & 0 & 0 & 0 & 0
        \end{pmatrix} \\
        B^2_{6\la 5}&=\begin{pmatrix}
            0 & 1 & 0 & 0 & 0 & 0 \\ 
            0 & 0 & 1 & 0 & 0 & 0 \\
            0 & 0 & 0 & 0 & 1 & 0 \\
            0 & 0 & 0 & 0 & 0 & 1
        \end{pmatrix},\ B^{*,2}_{5\la 6}=\begin{pmatrix}
            1 & 0 & 0 & 0 \\ 
            0 & 1 & 0 & 0 \\
            0 & 0 & 0 & 0 \\
            0 & 0 & 1 & 0 \\
            0 & 0 & 0 & 1 \\
            0 & 0 & 0 & 0
        \end{pmatrix} \\
        B^2_{7\la 6}&=\begin{pmatrix}
            0 & 1 & 0 & 0 \\
            0 & 0 & 0 & 1
        \end{pmatrix},\ B^{*,2}_{6\la 7}=\begin{pmatrix}
            1 & 0 \\
            0 & 0 \\
            0 & 1 \\
            0 & 0
        \end{pmatrix} \\
        B^2_{8\la 5}&=\begin{pmatrix}
           -1 & 1 & 0 & 0 & 0 & 0 \\
           -2 & 0 & 0 & 1 & 0 & 0 \\
           -4 & 0 & 0 & 0 & 0 & 1
        \end{pmatrix},\ B^{*,2}_{5\la 8}=\begin{pmatrix}
           -2 & 0 & 0 \\
           -2 & 0 & 0 \\
           -2 & 1 & 0 \\ 
           -4 & 0 & 0 \\
           -4 & 0 & -1 \\
           -8 & 0 & 0
        \end{pmatrix}
    \end{aligned}$}}
\end{tabular}
\caption{Two generic points in $E_8$}
\label{E8twopoints}
\end{table}
}

\begin{lemma}\label[lemma]{E8lemma}
   \begin{enumerate}
       \item $(B^i,B^{*,i},1,0)\in\mu^{-1}(0),\ i=1,2$.
       \item $x_1=-2^5,\ y_1=2^8,\ z_1=-2^{12}$.
       \item $x_2=-2^3,\ y_2=-2^5,\ z_2=-2^8$.
   \end{enumerate}
\end{lemma}

\section{Connections between anti-Poisson involutions and Lie algebra involutions}\label{liealgsection}

Let $\hg$ be a simple Lie algebra of types $ADE$ and let $\N$ denote the nilpotent cone of $\hg$. Note that $\N$ is also the closure of the regular nilpotent orbit $\bO_{\reg}$. Let $S$ be the Slodowy slice to a subregular nilpotent element $e\in\OO_{\sreg}$. Then $S\cap\N\simeq X$, where $X$ is the Kleinian singularity of the same type as $\hg$ (cf. \Cref{sregtheorem} below).

Recall from \Cref{nilp} that, for a Lie algebra involution $\sigma\colon\hg\to\hg$, we can decompose $\hg=\hk\oplus\hp$, where $\mathfrak{k}$ is the $1$-eigenspace and $\mathfrak{p}$ is the $(-1)$-eigenspace of $\sigma$. Set $\theta=-\sigma$; then $\theta$ is an anti-Poisson involution of $\hg$. Under the further assumption that $e\in\hp$, we can choose a normal $\sl_2$-triple such that the Slodowy slice $S$ is $\theta$-stable \cite[Theorems 9.4.1, 9.5.1]{CMnilp}. It is clear that $\theta$ preserves the nilpotent cone $\N$, and hence restricts to an anti-Poisson involution of $S\cap\N\simeq X$.

This raises the natural question of which anti-Poisson involutions of $X$ arise as restrictions of Lie algebra involutions. \Cref{Liealgbij} below addresses this, and is the main result of this section.

\begin{proposition}\label[proposition]{Liealgbij}
    All anti-Poisson involutions of a Kleinian singularity $X$ arise from Lie algebra involutions. Moreover, they are in bijection with quasi-split real forms of $\hg$, except in type $A_n$ with $n$ odd.
\end{proposition}

\subsection{Symmetries of the resolution diagram}

Let $\theta$ be an anti-Poisson involution of $X=\C^2/\Gamma$, and let $\tau$ denote the automorphism of the Dynkin diagram $\Delta_{\Gamma}$ induced by $\theta$ via \cref{diagramautohom}. Recall that $\Aut(\Delta_{\Gamma})$ denotes the automorphism group of the Dynkin diagram and $\Inv(\Delta_{\Gamma})/\sim$ denotes the set of conjugacy classes of involutions in $\Aut(\Delta_{\Gamma})$. We use $[\cdot]$ to denote a conjugacy class.

As we discussed in \Cref{liftsection}, any automorphism $\theta$ of $X=\C^2/\Gamma$ lifts uniquely to an automorphism $\ttheta$ of the minimal resolution $\wt{X}$ by the definition and uniqueness of the minimal resolution. Moreover, $\ttheta$ restricts to the exceptional fiber $\pi^{-1}(0)$, since $\theta$ fixes the singular point $0$.
For an automorphism of the exceptional fiber $\pi^{-1}(0)$, we also use $[\cdot]$ for the conjugacy class of its induced automorphism on the dual graph $\Delta_{\Gamma}$.

\begin{proposition}\label[proposition]{automatch}
    The action induced by $\ttheta$ on the dual graph $\Delta_\Gamma$ of the exceptional fiber $\pi^{-1}(0)$ coincides with the diagram automorphism $\tau$. Moreover, the following map
\begin{equation}\label{apiresolution}
    \phi_{\Gamma}:\API(\C^2/\Gamma)/\sim\to\Inv(\Delta_{\Gamma})/\sim,\ \ [\theta]\mapsto [\ttheta|_{\pi^{-1}(0)}]=[\tau],
\end{equation}
is a set-theoretic bijection except in type $A_n$ with $n$ odd.
\end{proposition}
\begin{proof}
Following the guidelines of 
\Cref{liftconditions}, the lifts $\ttheta$ are constructed as involutions of quiver varieties in \Cref{Anquiver,Dnquiver,E6quiver,E7quiver,E8quiver}, where the automorphism $\tau$ of the Dynkin diagram is used in the construction. A case-by-case analysis, together with \Cref{AnAPIrem,DnevenAPIrem,DnoddAPIrem,E6APIrem,E7APIrem,E8APIrem}, shows that the action induced by $\ttheta$ on the dual graph of $\pi^{-1}(0)$ agrees with the action of $\tau$ on $\Delta_{\Gamma}$.
\end{proof}

\begin{remark}\label[remark]{functoriality}
\begin{enumerate}
    \item An easy way to remember the correspondence in \Cref{automatch} is as follows. In type $A_n$ with $n\geqslant 2$, the Type \RMN{1} anti-Poisson involution induces the nontrivial diagram automorphism, while the Type \RMN{2} and, when defined, Type \RMN{3} anti-Poisson involutions induce the trivial one. In types $D_n$ and $E_6$, the Type \RMN{1} anti-Poisson involution induces the trivial diagram automorphism, while the Type \RMN{2} anti-Poisson involution induces the nontrivial one. In types $E_7$ and $E_8$, the unique anti-Poisson involution induces the trivial diagram automorphism.
    \item As suggested by a referee, \Cref{automatch} can also be proved via the functoriality of the geometric McKay correspondence. The geometric McKay correspondence established in \cite[Theorem 2.2]{GSVMcKay} associates each nontrivial $V_i\in\Irr(\Gamma)$ with a vector bundle $\mathcal{E}_i$ on the minimal resolution $\wt{X}$. The first Chern classes $c_1(\mathcal{E}_i)$ form a basis of the Picard group $\Pic(\wt{X})$, indexed by the irreducible components of $\pi^{-1}(0)$. Recall that any anti-Poisson involution $\theta$ comes from an element $g\in N_{\GL_2(\C)}(\Gamma)^-$ (cf. \Cref{autogrbij}). Twisting $V_i$ by an element $g$ (cf. \cref{reptwist}) corresponds to pulling back $\mathcal{E}_i$ by the lifted automorphism $\tilde{g}$ of $\wt{X}$. However, we did not pursue this approach in the current paper.
\end{enumerate}
\end{remark}

\subsection{Springer resolution and Dynkin curves}

In this section, we review the Lie-theoretic approach to realize the minimal resolution $\pi\colon\wt{X}\to X$, and provide a description of the irreducible components of $\pi^{-1}(0)$ in terms of Lie-theoretic data, following \cite[Chapter III, \S 5--6]{slodowy1980simple}.

Let $G$ be the simply connected algebraic group with Lie algebra $\hg$, and let $B\subset G$ be a Borel subgroup. Denote by $\hb=\Lie(B)$ its Lie algebra and let $\hn=[\hb,\hb]$ be the nilradical of $\hb$. The flag variety $G/B$, which can be identified with the variety of Borel subalgebras, carries a natural transitive $G$-action. An element $\hb'\in G/B$ can be written as $\hb'=\Ad g(\hb)$, where $g\in G$. 


Next, consider the Springer resolution
\begin{equation}\label{Springerresolution}
\Phi:\widetilde{\N}=\{(\hb',x)\in G/B\times\N|x\in [\hb',\hb']\}\longrightarrow\N,
\qquad
(\hb',x)\longmapsto x.
\end{equation}



\begin{theorem}\cite[Theorem]{Brieskorn}\label{sregtheorem}
Let $\hg$ be a simple Lie algebra of type $ADE$, let $\N$ be its nilpotent cone, and let $S$ be the Slodowy slice to a subregular nilpotent element $e\in\OO_{\sreg}$. Then $S\cap \N\simeq X$ is a Kleinian singularity of the same type as $\hg$, with the unique singular point at $e$. Moreover, the restriction of the Springer resolution in \cref{Springerresolution} to $\Phi^{-1}(S\cap\N)\to S\cap\N$ gives the minimal resolution of the Kleinian singularity $\pi\colon\wt{X}\to X$.
\end{theorem}




Next, fix a maximal torus $T$ contained in $B$ and let $\{\alpha_1,\cdots,\alpha_n\}$ be the set of simple roots determined by the pair $(T,B)$. For each $\alpha_i$, let $P_{i}$ be the minimal parabolic subgroup of $G$ with Lie algebra $\hb+\hg_{-\alpha_i}$, where $\hg_{-\alpha_i}$ is the root space associated to the negative simple root $-\alpha_i$. The fibers of the natural projection $G/B\to G/P_i$ have the form $gP_i/B$ for $g\in G$. These fibers are isomorphic to the projective line $\pr^1$, and they are called \emph{lines of type $\alpha_i$}. \Cref{ijintersect} below will be useful in the proof of \Cref{Liealgbij}.

\begin{proposition}\cite[\S 3.10, Proposition 1]{Steinbergconj}\label[proposition]{ijintersect}
Assume $i\neq j$. The intersection of a line of type $\alpha_i$ and a line of type $\alpha_j$ consists of at most one point in $G/B$.
\end{proposition}

\Cref{sregtheorem} tells us that the Springer fiber $\Phi^{-1}(e)$ over a subregular nilpotent element $e\in\OO_{\sreg}$ is also the exceptional fiber $\pi^{-1}(0)$ of the Kleinian singularity of the corresponding type.  \Cref{Springerfiber} below gives a description of the irreducible components of $\pi^{-1}(0)\simeq\Phi^{-1}(e)$ in terms of Lie-theoretic data. 

\begin{theorem}\cite[\S 3.10, Theorem 2]{Steinbergconj}\label{Springerfiber}
We have $\pi^{-1}(0)_{\red}\simeq\Phi^{-1}(e)_{\red}=\cup_{1\leqslant i\leqslant n} C_i$, whose dual graph is the Dynkin diagram of the corresponding type. Moreover, each $C_i$ is a line of type $\alpha_i$ for $1\leqslant i\leqslant n$, i.e., there exists $g_i\in G$ such that $C_i=g_iP_i/B$. 
\end{theorem}


\subsection{A Lie-theoretic approach to constructing the lift $\ttheta$}

In \Cref{Anquiver,Dnquiver,E6quiver,E7quiver,E8quiver}, we have provided a quiver-theoretic approach to constructing the lift $\ttheta$ following the guidelines of \Cref{liftconditions}. In this section, we provide an alternative method to construct the lift $\ttheta$: a Lie-theoretic approach using Lie algebra involutions and the Springer resolution.

Recall that $\sigma\colon\hg\to\hg$ is a Lie algebra involution. According to \cite[\S 4.4]{OVLieGroups}, up to conjugation, we may assume that $\sigma$ preserves the Cartan subalgebra $\mathfrak{t}=\Lie(T)$, the set of simple roots $\{\alpha_1,\cdots,\alpha_n\}$, and the Borel subalgebra $\hb$. 
Recall that $\theta=-\sigma$. Note that for any Borel subalgebra $\hb'$, we have $\theta(\hb')=\sigma(\hb')$, which is again a Borel subalgebra. 

Recalling the Springer resolution in \cref{Springerresolution}, the action of $\theta$ on $\hg$ lifts naturally to 
\begin{equation}
\ttheta:\widetilde{\N}\longrightarrow \widetilde{\N},
\qquad (\hb',x) \longmapsto(\theta(\hb'),\theta(x)),
\end{equation}
with the property $\Phi\circ\ttheta=\theta\circ\Phi$.
Recall from \Cref{sregtheorem} that 
\[
\Phi\colon \wt{X}\simeq\Phi^{-1}(S\cap\N)\to S\cap\N\simeq X
\]
is the minimal resolution of the Kleinian singularity. Since $\sigma$ preserves the set of simple roots, it induces an automorphism of the Dynkin diagram $\Delta_{\Gamma}$ \cite[\S 4.4, eq (1)]{OVLieGroups}. \Cref{lieliftoncomponent} below follows directly from \Cref{sregtheorem,Springerfiber} and \Cref{ijintersect}; we omit the proof.

\begin{proposition}\label[proposition]{lieliftoncomponent}
    Assume $e\in\OO_{\sreg}\cap\hp$, so that the anti-Poisson involution $\theta:=-\sigma$ restricts to $S\cap\N\simeq X$. Then the restriction of $\ttheta$ to $\Phi^{-1}(S\cap\N)\simeq\wt{X}$ is a lift of $\theta$. Moreover, we have $\ttheta(C_i)=C_{\sigma(i)}$.
\end{proposition}

Therefore, determining which anti-Poisson involutions arise from Lie algebra involutions reduces to finding the Lie algebra involutions for which $\OO_{\sreg}\cap\hp\neq\emptyset$, analyzing the action of $\ttheta$ on the Springer fiber $\Phi^{-1}(e)$, and matching this with the action of $\ttheta$ on $\pi^{-1}(0)$ constructed via the quiver-variety approach in \Cref{Anquiver,Dnquiver,E6quiver,E7quiver,E8quiver}.





\subsection{Nilpotent orbits in real forms}

We first determine which Lie algebra involutions satisfy $\OO_{\sreg}\cap\hp\neq\emptyset$.

Let $\hg_{\R}$ be a real form of $\hg$, i.e., a real Lie algebra such that $\hg_{\R}\otimes_{\R}\C\simeq\hg$. A real form is determined by an antilinear involution $u: \hg\to\hg$. For each such $u$, there exists a compact real form $u_c: \hg\to\hg$ that is compatible with $u$, i.e., $u\circ u_c=u_c\circ u$. Then the composition $\sigma= u\circ u_c\colon \hg\to\hg$ is a Lie algebra involution. Moreover, this construction gives a bijection between real forms $\hg_{\R}$ of $\hg$ up to isomorphism and the Lie algebra involutions of $\hg$ up to conjugation. For details, see \cite[\S 5.1]{OVLieGroups}. 

\begin{proposition}\label[proposition]{KSorbitbij}
Let $\sigma$ be a Lie algebra involution, and let $\hg_{\R}$ be the corresponding real form.
\begin{enumerate}[label=(\arabic*)]
    \item\label{KSorbitbij1}  $\OO_{\reg}\cap\hp\neq\emptyset$ if and only if $\OO_{\reg}\cap\hg_{\R}\neq\emptyset$.
    \item\label{KSorbitbij2} $\OO_{\sreg}\cap\hp\neq\emptyset$ if and only if $\OO_{\sreg}\cap\hg_{\R}\neq\emptyset$.
\end{enumerate}
\end{proposition}
\begin{proof}
    This follows directly from the Kostant-Sekiguchi bijection (cf. \cite[Theorem 9.5.1]{CMnilp}).
\end{proof}

\begin{remark}
Rothschild proved that $\OO_{\reg}\cap\hg_{\R}\neq\emptyset$ if and only if $\hg_{\R}$ is quasi-split \cite[Proposition 5.1]{Rothschild}. Combining this with \ref{KSorbitbij1} of \Cref{KSorbitbij}, we obtain that $\OO_{\reg}\cap\hp\neq\emptyset$ if and only if $\hg_{\R}$ is quasi-split. 


Note that the subregular nilpotent element $e$ is the unique singular point of $S\cap\N\simeq X$ and $(S\cap\N)\setminus\{e\}$ consists only of regular nilpotent elements. Moreover, $e$ is fixed by $\theta$ as we choose $e\in\OO_{\sreg}\cap\hp$, hence $e\in(S\cap\N)^{\theta}\simeq X^{\theta}$. That is to say, $X^{\theta}$ is not a single point precisely when the corresponding real form is quasi-split. We will not use this equivalence directly, but it gives a conceptual explanation for the appearance of quasi-split real forms in \Cref{Liealgbij} and for the exclusion of Type \RMN{3} in type $A_n$ when $n$ is odd; cf. \Cref{Anfixedloci}.



\end{remark}


The classification of nilpotent orbits in real forms of complex semisimple Lie algebras can be found in \cite[\S 9.3--9.6]{CMnilp}. We summarize the relevant results below.
\begin{enumerate}[label=(\arabic*)]
    \item\label{Anorbit} Type $A_n$. The Lie algebra is $\hg=\sl_{n+1}(\C)$. It is well known that the regular nilpotent element is given by the partition $(n+1)$ and the subregular nilpotent element is given by the partition $(n,1)$. According to \cite[Theorem 9.3.3]{CMnilp}, the split real form $\sl_{n+1}(\R)$ and the quasi-split real form $\su(\lceil\frac{n+1}{2}\rceil,\lfloor\frac{n+1}{2}\rfloor)$ contain both the regular and subregular nilpotent elements. When $n$ is odd, the non-split real form $\su(\frac{n+3}{2},\frac{n-1}{2})$ contains subregular nilpotent elements but not regular nilpotent elements. The other real forms do not contain subregular nilpotent elements.
    
    \item\label{Dnorbit} Type $D_n$. The Lie algebra is $\hg=\so_{2n}(\C)$. According to \cite[Theorem 5.1.4]{CMnilp}, the regular nilpotent element is given by the partition $(2n-1,1)$, and the subregular nilpotent element is given by the partition $(2n-3,3)$. According to \cite[Theorem 9.3.4]{CMnilp}, the split real form $\so(n,n)$ and the quasi-split real form $\so(n+1,n-1)$ contain both the regular and subregular nilpotent elements. The other real forms do not contain subregular nilpotent elements. 
\end{enumerate}

\begin{enumerate}
    \setcounter{enumi}{2}
    \item\label{E6orbit} Type $E_6$. The split real form $E\RMN{1}$ ($E_{6(6)}$) and the quasi-split real form $E\RMN{2}$ ($E_{6(2)}$) contain both the regular and subregular nilpotent elements. The other real forms do not contain the subregular nilpotent elements \cite[\S 9.6, Tables on pp.~153--154]{CMnilp}.

    \item\label{E7orbit} Type $E_7$. Only the split real form $E\RMN{5}$ ($E_{7(7)}$) has regular and subregular orbits.  The remaining real forms do not contain the subregular nilpotent elements \cite[\S 9.6, Table on pp.~156--157]{CMnilp}.
    
    \item\label{E8orbit} Type $E_8$. Only the split real form $E\RMN{8}$ ($E_{8(8)}$) has regular and subregular orbits. The remaining real forms do not contain the subregular nilpotent elements \cite[\S 9.6, Table on pp.~162--163]{CMnilp}.
\end{enumerate}

We now provide the proof of \Cref{Liealgbij}.

\begin{proof}[Proof of \Cref{Liealgbij}]
    We have listed above all real forms in types $ADE$  containing subregular nilpotent elements. It remains only to analyze the action of the corresponding Lie algebra involution on the Dynkin diagram, as listed in \cite[Table 7]{OVLieGroups}, and to match it with the action in \Cref{automatch}. This yields the following results.
    

    \begin{enumerate}
    \item Type $A_n$. 
    \begin{enumerate}[label=(\roman*)]
    \item The Lie algebra involution associated with the split real form $\sl_{n+1}(\R)$ is
    \begin{equation}\label{Aninvolutioni}
    \sigma:\hg\to\hg,\ M\mapsto-M^t,
    \end{equation}
    where $M^t$ denotes the transpose of a matrix $M$. The Lie algebra involution in \cref{Aninvolutioni} induces the nontrivial diagram automorphism on the Dynkin diagram $\Delta_{\Gamma}$; hence it restricts to the Type \RMN{1} anti-Poisson involution on $S\cap\N\simeq X$, for both odd and even $n$. 
    \item The Lie algebra involution associated with the quasi-split real form $\su(\lceil\frac{n+1}{2}\rceil,\lfloor\frac{n+1}{2}\rfloor)$ is
    \begin{equation}\label{Aninvolutionii}
    \sigma:\hg\to\hg,\ M\mapsto\Ad I_{\lceil\frac{n+1}{2}\rceil,\lfloor\frac{n+1}{2}\rfloor}(M),
    \end{equation}
    where the matrix $I_{k,l}$ refers to a diagonal matrix with $k$ entries equal to $1$ and $l$ entries equal to $-1$. The Lie algebra involution in \cref{Aninvolutionii} induces the trivial diagram automorphism on the Dynkin diagram $\Delta_{\Gamma}$; hence it corresponds to the Type \RMN{2} anti-Poisson involution when $n$ is even. When $n$ is odd, we have $(S\cap\N)^{\theta}\neq\{e\}$ because there exist regular $K$-nilpotent orbits in $\hp$ whose closures contain the chosen subregular nilpotent element $e$ (\cite[Proposition 3.8]{Sekiguchi}). Comparing this with the description of the fixed point locus $X^{\theta}$ in \Cref{Anfixedloci}, we conclude that the Lie algebra involution in \cref{Aninvolutionii} restricts to the Type \RMN{2} anti-Poisson involution on $S\cap\N\simeq X$ when $n$ is odd.
    \item (only when $n$ is odd) The Lie algebra involution associated with the non-split real form $\su(\frac{n+3}{2},\frac{n-1}{2})$ is 
    \begin{equation}\label{Aninvolutioniii}
      \sigma:\hg\to\hg,\ M\mapsto\Ad I_{\frac{n+3}{2},\frac{n-1}{2}}(M).  
    \end{equation}
    Note that we have $(S\cap\N)^{\theta}=\{e\}$ because $\OO_{\reg}\cap\hp=\emptyset$ (cf. \Cref{KSorbitbij} and \ref{Anorbit} above). Comparing this with the description of the fixed point locus $X^{\theta}$ in \Cref{Anfixedloci}, we conclude that the Lie algebra involution in \cref{Aninvolutioniii} restricts to the Type \RMN{3} anti-Poisson involution of $S\cap\N\simeq X$.
    \end{enumerate}
    \item Type $D_n$. \begin{enumerate}[label=(\roman*)]
        \item When $n$ is even, the Lie algebra involution corresponding to the split real form $\so(n,n)$ induces the trivial diagram automorphism of $\Delta_{\Gamma}$; hence it corresponds to the Type $\RMN{1}$ anti-Poisson involution. The Lie algebra involution corresponding to the quasi-split real form $\so(n+1,n-1)$ induces the nontrivial diagram automorphism; hence, it corresponds to the Type $\RMN{2}$ anti-Poisson involution.

        \item When $n$ is odd, the Lie algebra involution corresponding to the split real form $\so(n,n)$ induces the nontrivial order-two diagram automorphism of $\Delta_{\Gamma}$; hence it restricts to the Type $\RMN{2}$ anti-Poisson involution on $S\cap\N\simeq X$. The Lie algebra involution corresponding to the quasi-split real form $\so(n+1,n-1)$ induces the trivial diagram automorphism; hence, it restricts to the Type $\RMN{1}$ anti-Poisson involution on $S\cap\N\simeq X$.
    \end{enumerate}
    \item Type $E_6$. The Lie algebra involution corresponding to the split real form $E\RMN{1}$ ($E_{6(6)}$) induces the nontrivial diagram automorphism of $\Delta_{\Gamma}$; hence, it restricts to the Type $\RMN{2}$ anti-Poisson involution of $S\cap\N\simeq X$. The Lie algebra involution corresponding to the quasi-split real form $E\RMN{2}$ ($E_{6(2)}$) induces the trivial diagram automorphism; hence, it restricts to the Type $\RMN{1}$ anti-Poisson involution of $S\cap\N\simeq X$.
    \item Type $E_7$. There is a unique anti-Poisson involution up to conjugation. Hence, the restriction of the Lie algebra involution corresponding to the split form $E\RMN{5}$ ($E_{7(7)}$) to $S\cap\N\simeq X$ yields this unique anti-Poisson involution.
    \item Type $E_8$. There is a unique anti-Poisson involution up to conjugation. Hence, the restriction of the Lie algebra involution corresponding to the split form $E\RMN{8}$ ($E_{8(8)}$) to $S\cap\N\simeq X$ yields this unique anti-Poisson involution.
    \end{enumerate}


The results above also yield a bijection between quasi-split real forms and anti-Poisson involutions, except in type $A_n$ with $n$ odd. 
\end{proof}

\begin{remark}\label[remark]{typeAalt} 
    In type $A_n$, the elements of $S\cap\N\simeq X$ can be realized explicitly as matrices. Namely, the Slodowy slice $S$ can be written in matrix form, and intersecting with $\N$ amounts to imposing that all non-leading coefficients of the characteristic polynomial vanish. Therefore, whether the restriction of $\theta:=-\sigma$ is of Type \RMN{1}, \RMN{2}, or \RMN{3} can also be determined by examining its action on the generators of $\C[S\cap\N]$, bypassing the need to consider the lifted action $\ttheta$ on the exceptional fiber $\pi^{-1}(0)$ as in \Cref{lieliftoncomponent}. This approach involves some algebraic computation, but it is not too hard and is left as an exercise to the reader. Alternatively, one may consult the proofs of \cite[Theorem 7, Propositions 3.8, 3.10]{Sekiguchi}.
\end{remark}

\begin{remark}\label[remark]{SekiOhtacompare}

The fixed point loci $S\cap\N\cap\hp\simeq X^{\theta}$ in \Cref{Dnevenfixedloci,Dnoddfixedloci} for type $D$ were also analyzed by Sekiguchi \cite{Sekiguchi} and Ohta \cite{Ohta}. 


The Type \RMN{1} fixed point locus for $n$ even in \Cref{Dnevenfixedloci}  and the Type \RMN{2} fixed point locus for $n$ odd in \Cref{Dnoddfixedloci} correspond to 
the variety $S\cap\N\cap\hp\simeq\Spec\C[x,y]/(x^{n-1}+xy^2)$ associated with the symmetric pair $(\so(2n,\C),\so(n,\C)+\so(n,\C))$ in \cite[Table VI, (4-III)(2)]{Sekiguchi} for the nilpotent orbit of Jordan type $(2n-3,3)$.
The Type \RMN{2} fixed point locus for $n$ even in \Cref{Dnevenfixedloci} and the Type \RMN{1} fixed point locus for $n$ odd in \Cref{Dnoddfixedloci} correspond to 
the variety $S\cap\N\cap\hp\simeq\Spec\C[x,y]/(x^{n-1}+y^2)$ associated with the symmetric pair $(\so(2n,\C),\so(n+1,\C)+\so(n-1,\C))$ in \cite[Table VI, (4-V)(2)]{Sekiguchi} for the same Jordan type $(2n-3,3)$.

The $D_n$ Kleinian singularity can also be realized as the intersection $S\cap\N$ in type $C_{n-1}$, i.e., in $\sp(2n-2,\C)$ \cite[\S 6.4, Theorem]{slodowy1980simple}. For the symmetric pair $(\sp(2n-2,\C),\gl_{n-1}(\C))$, the same fixed point loci in type $D_n$ also appear; see \cite[Table VI, (4-VIII)(2)(3)]{Sekiguchi} and \cite[Table VIII, bottom row]{Ohta}. The regular-to-subregular degenerations in types $D_n$ and $C_{n-1}$ are also related by column removal of the corresponding partitions. Related column/row-removal reductions in the symmetric-pair setting are discussed in \cite[Remark 10.3.6]{Liquiver} and \cite[Theorem 4]{Ohta}.







\end{remark}

\subsection*{Acknowledgements}
I would like to express my sincere gratitude to Ivan Losev for introducing this problem, as well as for many enlightening discussions and valuable feedback that improved this work. I am grateful to the anonymous referee for suggesting the Lie-theoretic approach developed in \Cref{liealgsection}. I would also like to thank Sam Evens for teaching me about split and quasi-split real forms, and Kenta Suzuki for pointing me to the reference \cite{Artin66}. Additionally, I thank David Bai and Do Kien Hoang for insightful conversations. This work was partially supported by the National Science Foundation under grant DMS-$2001139$.


\subsection*{Data Availability}
Data sharing is not applicable to this article as no datasets were generated or analyzed during the current study.

\section*{Declarations}
\subsection*{Conflict of Interest}
The author has no relevant conflict of interest to declare.

\printbibliography
\end{document}